\documentclass[a4paper,11pt]{amsart}
\usepackage[a4paper,margin=31mm]{geometry}
\usepackage{amsmath,amssymb,mathtools}
\usepackage{enumitem}
\usepackage{hyperref}
\usepackage[T1]{fontenc}
\usepackage{lmodern}

\newtheorem{theorem}{Theorem}[section]
\newtheorem{proposition}[theorem]{Proposition}
\newtheorem{lemma}[theorem]{Lemma}

\theoremstyle{definition}
\newtheorem{definition}[theorem]{Definition}
\theoremstyle{remark}

\numberwithin{equation}{section}

\newcommand{\N}{\mathbb N}
\newcommand{\Z}{\mathbb Z}
\newcommand{\R}{\mathbb R}
\newcommand{\C}{\mathbb C}

\newcommand{\E}{\mathbb E}
\newcommand{\e}{\mathrm e}
\newcommand{\1}{\mathbf 1}
\newcommand{\poly}{\operatorname{poly}}
\newcommand{\Lip}{\operatorname{Lip}}
\newcommand{\TV}{\operatorname{TV}}
\newcommand{\lcm}{\operatorname{lcm}}

\newcommand{\cC}{\mathcal C}
\newcommand{\eps}{\varepsilon}
\newcommand{\asum}{\sideset{}{^{\ast}}\sum}

\title[Quantitative Gowers uniformity in short intervals]
{Quantitative Gowers uniformity of the primes\\
in intervals of length $X^{5/8+\varepsilon}$}

\author{Joni Ter\"av\"ainen}
\address{Department of Pure Mathematics and Mathematical Statistics,  University of Cambridge, CB3 0WB, UK}
\email{joni.p.teravainen@gmail.com}

\begin{document}
\begin{abstract}
We prove the first quantitative bounds for the Gowers norms of the von Mangoldt function minus a Siegel corrected model function in short intervals.  The result applies to intervals $(X,X+H]$ with $H\geq X^{5/8+\varepsilon}$ and gives quasipolynomial savings for the $U^k(X,X+H]$ norm.  The main new ingredient is an efficient non-abelian Type II inverse theorem with explicit dependence on the nilsequence dimension.  We prove this inverse theorem by applying Leng's efficient equidistribution theorem directly to an unweighted corner correlation in four variables, hence replacing the factorisation and subgroup argument in the non-abelian Type II proof of Matom\"aki, Shao, Tao and the author.

In joint work with Florian Richter, we use this quantitative result as one ingredient in proving the existence of infinitely many sum-product patterns of the form $\{x,x+y,xy\}$ in the shifted primes $\mathbb{P}-1$. 
\end{abstract}
\maketitle

\section{Statement of the result}

In this paper, we establish quantitative Gowers norm bounds for the von Mangoldt function $\Lambda$, suitably normalised, in short intervals. To state such results, we first need to define a suitable model function to which $\Lambda$ can be compared. For $w\geq 2$, we first define the \emph{Cram\'er model} $\Lambda_w$ by
\begin{equation*}
 \Lambda_w(n)=\frac{P(w)}{\varphi(P(w))}\1_{(n,P(w))=1},\qquad P(w)=\prod_{p\leq w}p.
\end{equation*}
To refine this model, we take into account the possible Siegel zero of $L$-functions. 
A level $w$ \emph{Siegel zero} is a real zero $\beta$ of some Dirichlet $L$-function $L(s,\chi)$ satisfying $\beta\geq1-c_0/\log w$, where $c_0>0$ is a small absolute constant and $\chi$ is a primitive real Dirichlet
character modulo an integer $q\leq w$ called the \emph{exceptional conductor}.  The Landau--Page theorem
\cite[Corollary 11.10]{MV} gives at most one such pair $(\beta,\chi)$.  We then define the \emph{Siegel model} $\widetilde{\Lambda}_w$ by
\begin{equation*}
 \widetilde\Lambda_w(n)=\Lambda_w(n)\bigl(1-\chi(n)n^{\beta-1}\bigr)
\end{equation*}
when the pair $(\beta,\chi)$ exists, and we set $\widetilde\Lambda_w=\Lambda_w$ otherwise.

We will consider the Gowers norms of $\Lambda-\widetilde{\Lambda}_w$ over short intervals. For $k\geq1$ and a finitely supported function $f\colon \Z\to\C$, we define the integer Gowers norm
\begin{equation*}
 \|f\|_{U^k(\Z)}
 \coloneqq \left(
 \sum_{n,h_1,\ldots,h_k\in\Z}
 \prod_{\omega\in\{0,1\}^k}
 \cC^{|\omega|}f(n+\omega_1h_1+\cdots+\omega_kh_k)
 \right)^{1/2^k},
\end{equation*}
where $\cC$ denotes complex conjugation.
If $I\subset\Z$ is a nonempty finite interval and $f\colon I\to\C$, we extend $f$ by zero outside $I$ and define the normalised interval Gowers norm
by
\begin{equation*}
 \|f\|_{U^k(I)}
 \coloneqq \frac{\|f\1_I\|_{U^k(\Z)}}{\|\1_I\|_{U^k(\Z)}}.
\end{equation*}
These quantities are norms when $k\geq2$ and seminorms when $k=1$.

With this notation, it was shown in~\cite[Proposition 4.6]{MTW} that the normalised von Mangoldt function obeys the quantitative Gowers norm bound
\begin{align}\label{eq:quantgowers}
\|\Lambda-\widetilde{\Lambda}_{w}\|_{U^k[X]}\ll_{k} \exp(-(\log w)^{c_k})    
\end{align}
in the range $\exp((\log \log X)^{1/c_k})\leq w\leq \exp((\log X)^{1/10})$ for some $c_k>0$. This is a slight extension of the work of Leng~\cite{Leng} that obtained quasipolynomial bounds for the Gowers norms of the primes, using the Leng--Sah--Sawhney quantitative inverse theorem~\cite{LSS} as one input. This improved on the work of Tao and the author~\cite{TT}, which in turn improved on the original Green--Tao--Ziegler qualitative bound in~\cite{GT-linear},~\cite{GT-Mobius},~\cite{GTZ} (with these earlier works using slightly different normalisations).

For Gowers norms over short intervals, a qualitative bound
\begin{align}\label{eq:shortgowers}
\|\Lambda-\widetilde{\Lambda}_{w}\|_{U^k(X,X+H]}=o_{w\to \infty}(1)+o_{X\to \infty; w}(1)    
\end{align}
was proved by Matom\"aki, Shao, Tao and the author~\cite{MSTT}, in the range $H\geq X^{5/8+\varepsilon}$, with $\varepsilon>0$ fixed (using in fact $\Lambda_w$ as the model function, but this is easily seen to be equivalent in the qualitative setting). We remark that inserting the Leng--Sah--Sawhney quasipolynomial inverse theorem into the arguments in~\cite{MSTT} does not lead to an immediate efficient quantification of the argument due to the need to track the dimension dependence of the nilsequences involved and, most importantly, the weak dimension dependencies in the Type II non-abelian nilsequence arguments of that paper.  

In this paper, we give a different, quantitative proof of this Type II estimate and prove the following simultaneous strengthening of~\eqref{eq:quantgowers} and~\eqref{eq:shortgowers}.

\begin{theorem}[Short interval quantitative uniformity]
\label{thm:corrected-main}
Let $k\geq2$ be an integer, let $\eps>0$, and set $a_k=2^{-k-2}$. There is some $c=c(k,\eps)>0$ such that the following holds.
If
\begin{equation}
 \label{eq:main-H-range}
2\leq  X^{5/8+\eps}\leq H\leq X
\end{equation}
and
\begin{equation}
 \label{eq:main-w-range}
 2\leq w\leq\exp\bigl((\log X)^{1/10}\bigr),
\end{equation}
then
\begin{equation}
 \label{eq:corrected-main-bound}
 \|\Lambda-\widetilde\Lambda_w\|_{U^k(X,X+H]}
 \ll_{k,\eps}w^{-a_k}+\exp\bigl(-(\log X)^c\bigr).
\end{equation}
\end{theorem}

\subsection{Applications and extensions} The quantification of the short interval Gowers norm estimate for the primes in~\cite{MSTT} naturally leads to quantifications of many consequences of such arguments, such as asymptotics for the linear pattern sums of the form
\begin{align*}
\sum_{\mathbf n\in(X,X+H]^d\cap\Z^d}
\Lambda(L_1(\mathbf n))\cdots\Lambda(L_k(\mathbf n)),
\end{align*}
with $(L_1,\ldots, L_k)$ a finite complexity system of affine-linear forms in $d$ variables, where qualitative estimates were obtained in~\cite[Theorem 1.7]{MSTT}. Here, the input from the current paper should produce arbitrary power of logarithm savings on the error term in the local-to-global asymptotic. We leave the details of this to the interested reader. 

The main new application where qualitative estimates would not suffice is in a joint work with Florian Richter~\cite{RT}, where we use this result as one ingredient for showing the existence of infinitely many sum-product patterns $\{x,x+y,xy\}$, with $x,y\in \mathbb{N}$ arbitrarily large, in the set of shifted primes $\mathbb{P}-1$. 

Finally, most of the arguments in this paper are not specific to the von Mangoldt function, and one would expect that the same methods can be used to prove quantitative short interval Gowers norm bounds also for some other arithmetic functions, such as the M\"obius function (as in the qualitative setting in~\cite{MSTT}). There are however some minor technicalities to do with proving major arc estimates for the normalised function $\mu-\mu_{\textnormal{Siegel}}$, where the Siegel model $\mu_{\textnormal{Siegel}}$ is defined in~\cite{TT}. We leave the details of resolving these to the interested reader.

\subsection{Outline of the proof}

We give a brief outline of the main ideas in the proof of Theorem~\ref{thm:corrected-main}. 

For simplicity, we assume in this outline that $w=\exp((\log X)^{1/10})$ and $X^{5/8+\varepsilon}\leq H\leq X^{1-\varepsilon}$. The extension of the roughness parameter range is handled by estimating the Gowers norm of the difference between model functions with different parameters, and the length extension is handled by splitting long intervals into shorter ones. 

The starting point of the argument is the quasipolynomial inverse theorem of Leng, Sah and Sawhney~\cite{LSS}. Applied to the $1$-bounded function $(\Lambda-\widetilde{\Lambda}_w)/(2\log X)$ on $(X,X+H]$, this shows that the only way~\eqref{eq:corrected-main-bound} can fail is if there exists a nilmanifold $G/\Gamma$ of step and degree at most $k-1$, of complexity $\leq \exp((\log X)^{c'})$ and dimension $\leq (\log X)^{c'}$, and a $1$-bounded nilsequence $F(g(\cdot)\Gamma)$ of Lipschitz norm $\leq \exp((\log X)^{c'})$ such that
\begin{align}\label{eq:lambda-correlate}
\left|\sum_{X<n\leq X+H}(\Lambda-\widetilde{\Lambda}_w)(n)F(g(n)\Gamma)\right|\geq \exp(-(\log X)^{c'})H    
\end{align}
where $c'<\sqrt{c}$ and $c>0$ is small enough in terms of $k,\varepsilon$. A vertical Fourier reduction either lowers the step directly or allows us to assume that $F$ is a vertical nilcharacter of nonzero frequency $\xi$ and that the last nontrivial lower central subgroup of $G$ is one-dimensional.

As in~\cite{MSTT}, we apply Heath-Brown's identity to decompose $\Lambda-\widetilde{\Lambda}_w$ into type I, Type $I_2$ and Type II sums; see Definition~\ref{def:type-sums} (we also obtain character-twisted type $I$ sums in the case of a Siegel zero; we ignore this technicality here). 
The number of terms in this decomposition is polylogarithmic, so the pigeonhole principle provides us with a function $f$ that is one of these three types, with either good $\ell^2$ norm control or total variation norm control for each convolution factor, such that we have the nilsequence correlation bound
\begin{align}\label{eq:fFsum}
\left|\sum_{X<n\leq X+H}f(n)F(g(n)\Gamma)\right|\geq \exp(-(\log X)^{c'})H.        
\end{align}
In the Type I, Type $I_2$ and non-abelian Type II cases,
we then seek from the assumption~\eqref{eq:fFsum} a factorisation through a lower step nilmanifold:
\begin{align}\label{eq:single-factorisation}
g=\varepsilon g_*\gamma,\qquad \varepsilon, \gamma\in \poly(\mathbb{Z},G),\qquad g_{*}\in \poly(\mathbb{Z},G_*),    
\end{align}
with appropriate smoothness bounds on $\varepsilon$, rationality bounds on $\gamma$, complexity bounds on the new nilmanifold $G_*/\Gamma_*$, and crucially, lower step for $G_{*}$ than $G$. As in the work of Leng~\cite{Leng}, this descent on step rather than dimension is crucial for preserving the good shape of savings, as the step is bounded whereas the dimension can be large. 

The argument then naturally splits into four steps:
\begin{enumerate}[label=\textnormal{(\roman*)}]
\item Type I correlation large $\implies$ factorisation through a lower step nilmanifold. 
\item Type $I_2$ correlation large $\implies$ factorisation through a lower step nilmanifold.
\item Type II: factorisation through a lower step nilmanifold and ruling out the abelian case.
\item Step descent and handling the step $0$ case.
\end{enumerate}

The need for quantitative dimension dependence for bounds leads to additional difficulties compared with~\cite{MSTT} when handling such steps\footnote{The inverse statements in that paper yield a horizontal character with controlled smoothness, which enables dimension reduction but not necessarily a step reduction.}, and in particular requires the use of the quantitative equidistribution theory of Leng~\cite{Leng}. The main new ingredients are
\begin{enumerate}
    \item Concatenation of local factorisations;

    \item A direct four-variable Type II argument;

    \item A quantitatively stronger abelian Type II argument, with the Siegel correction.
\end{enumerate}
We now explain these ingredients more precisely.

\emph{Ingredient (1).} In the Type I, Type $I_2$ and non-abelian Type II cases, we show that \eqref{eq:fFsum} forces a factorisation through a lower step nilmanifold. For Type I sums, applying the usual reductions to remove the bounded variation sequence at the cost of passing to progressions, and applying Leng's efficient equidistribution theorem (in the slightly modified form of Theorem~\ref{thm:Leng-equidistribution}), the conclusion from~\eqref{eq:fFsum} is that for many integers $a$ in a suitable range (depending on the Type I parameters), we have factorisations of the type
\begin{align*}
 g(a(r_a+q_an))=\varepsilon_a(n)g_a(n)\gamma_a(n),\qquad  r_a\in\Z,\qquad 1\leq q_a\leq\exp((\log X)^{c''}),   
\end{align*}
for some $c''>0$, where $\varepsilon_a$ is smooth, $\gamma_a$ is rational and $g_{a}$ takes values in a lower step subgroup $G_a$, which may depend on $a$. We then show in Lemma~\ref{lem:concatenate-dilated-factorisations} that this leads to a single factorisation~\eqref{eq:single-factorisation} for the whole polynomial sequence $g$, with control on the complexity bounds. To our knowledge, such arguments have not been stated before even in the qualitative setting, although the concatenation of smooth classical polynomial sequences, such as the projections $\eta\circ g$ to a horizontal character, is standard (and plays an important technical role here as well). For Type $I_2$ sums, after performing a hyperbola decomposition as in~\cite{MSTT}, and again applying Leng's theorem after some manipulations to remove bounded variation coefficients, the assumption~\eqref{eq:fFsum} leads, for many pairs $(b,q)$, to (essentially) factorisations of the type
\begin{align*}
g((m_{b,q}+nq)(n_{b,q}-nb))=\varepsilon_{b,q}(n) g_{b,q}(n)\gamma_{b,q}(n),  
\end{align*}
where $\varepsilon_{b,q}$ is smooth, $\gamma_{b,q}$ is rational and $g_{b,q}$ takes values in a lower step subgroup with bounds on its complexity. In Lemma~\ref{lem:concatenate-quadratic-path-factorisations}, we show that this again necessarily leads to a single factorisation~\eqref{eq:single-factorisation}. In the Type $I_2$ argument, this quadratic path concatenation is first applied to the dilated sequences $g(a\cdot)$, after which we concatenate in the outer variable $a$, similarly to the Type I case.

\emph{Ingredient (2).} After having handled the steps (i) and (ii), we are left with handling the Type II sum implication (iii). The handling of such sums in the non-abelian case is the most novel part of the argument, and here our arguments differ from the analysis of non-abelian Type II sums in~\cite{MSTT}, where one performed a reduction to a thin set and a dimension reduction argument. We instead argue as follows. After using Cauchy--Schwarz to remove the coefficients, we apply Leng's theorem directly to the polynomial sequence
\begin{align*}
 (g(ab),g(a'b),g(ab'),g(a'b'))\in \poly(\mathbb{Z}^4, G^4).   
\end{align*}
Leng's theorem provides a set of horizontal characters on $G^4$  whose compositions with this polynomial sequence have controlled $C^\infty$ norms, and it also gives a commutator vanishing condition on their common kernel in $G^4$ when composed with the vertical character $\Xi=(\xi,-\xi,-\xi,\xi)$. Now, if these horizontal characters are of the form $(\lambda_{1,j},\lambda_{2,j},\lambda_{3,j},\lambda_{4,j}),$
then looking at $\lambda_{1,j}+\lambda_{2,j}$ and $\lambda_{1,j}+\lambda_{3,j}$ and applying concatenation arguments eventually transfers this information back to $G$, and the commutator calculation isolates one coordinate to force a genuine drop in step.

\emph{Ingredient (3).} Using ingredient (2), we can complete the non-abelian case of step (iii) of the argument. For the abelian case, we can assume that $F(g(n)\Gamma)=b_0e(P(n))$ is a classical polynomial phase. 
The abelian Type II theorem of \cite{MSTT} reduces the phase, after splitting to progressions, to $n^{iT}$ times a factor of controlled total variation. Large $|T|$ is excluded using stronger high frequency estimates for the Type II component. For small $|T|$, we remove the phase from the original correlation with $\Lambda-\widetilde\Lambda_w$ by splitting into slightly shorter intervals and apply estimates in arithmetic progressions and short intervals incorporating the Siegel correction. This leads to the need for some strengthenings of arguments from~\cite{MSTT}; these are handled in Section~\ref{sec:short-arithmetic-models}.

Finally, by freezing the smooth and rational factors on suitable progressions, we transfer the original correlation~\eqref{eq:lambda-correlate} with $\Lambda-\widetilde\Lambda_w$ to a lower step nilsequence. We can iterate this, and at step zero the nilsequence is constant, so the estimates over short progressions give a contradiction;
see Section~\ref{sec:step-descent}.

\subsection{Acknowledgements} The author is grateful to Florian Richter for helpful comments and for inspiring the author to work on this problem. The author was supported by funding from the European Union's Horizon Europe research and innovation programme under ERC grant agreement no. 101162746.

The author used ChatGPT to assist with the initial drafting of some statements and with polishing and proofreading. The author reviewed all such suggestions and takes full responsibility for the mathematical statements, arguments, and final wording.

\section{Notation}
\label{subsec:notation}

We follow the general conventions of \cite[Section~1.5]{MSTT}. We write $\N=\{1,2,\ldots\}$ and $[H]=\{1,\ldots,\lfloor H\rfloor\}$ for real $H\geq1$. We write $\1_E$ for the indicator of a set $E$, and
$\e(t)=e^{2\pi i t}$ for $t\in\R$. For a nonempty finite set $S$, we write
$\E_{n\in S}f(n)\coloneqq |S|^{-1}\sum_{n\in S}f(n)$. If $A$ and $B$ are nonnegative quantities, then $A\ll_{\mathcal P}B$ means that $A\leq C_{\mathcal P}B$ for a constant depending only on
the parameters in $\mathcal P$.

If $u<v$, then $(u,v]$ denotes $(u,v]\cap\Z$ whenever it occurs as a summation set or an argument of a discrete norm.  For $x\in\R$, let $\|x\|_{\R/\Z}$ denote its distance from the nearest integer.  For
$x\in\R/\Z$, use the same notation for the distance of any representative from the nearest integer.

\emph{Smoothness norms.} We follow the convention in \cite[Section~2.2]{MRSTT}.  Let $d\geq1$ and $t\geq1$ be integers, let $\phi\colon\Z^t\to\R/\Z$ be a polynomial of degree at most $d$, and let $I=I_1\times\cdots\times I_t$ be a nonempty integer box.  Write $|\mathbf j|=\sum_{i=1}^t|j_i|$ for $\mathbf j\in\Z^t$, and let
$\partial_i\phi(\mathbf n)=\phi(\mathbf n+\mathbf e_i)-\phi(\mathbf n)$. We define
\begin{equation*}
 \|\phi\|_{C^\infty(I)}
 \coloneqq \max_{\substack{\mathbf j\in\Z_{\geq0}^t\\|\mathbf j|\leq d}}
 \ \max_{\mathbf n\in I}
 \left(\prod_{i=1}^t|I_i|^{j_i}\right)
 \|\partial_1^{j_1}\cdots\partial_t^{j_t}\phi(\mathbf n)\|_{\R/\Z}.
\end{equation*}
It is instructive to compare this norm with the coefficient norm in \cite[Definition 2.7]{GT-nilmanifold} and
\cite[Definition 2.2]{Leng}.  Write $I=\mathbf a+[N_1]\times\cdots\times[N_t]$, where $N_i\geq1$.
Then the binomial expansion of $\phi$ takes the form
\begin{equation*}
 \phi(\mathbf a+\mathbf n)=\sum_{\substack{\mathbf j\in\Z_{\geq0}^t\\|\mathbf j|\leq d}}
   \alpha_{\mathbf j}\prod_{i=1}^t\binom{n_i}{j_i},
\end{equation*}
and we define
\begin{equation*}
 \|\phi\|_{\mathrm{coef},I}\coloneqq\max_{\substack{\mathbf j\in\Z_{\geq0}^t\\0<|\mathbf j|\leq d}}
 \left(\prod_{i=1}^tN_i^{j_i}\right)
 \|\alpha_{\mathbf j}\|_{\R/\Z}.
\end{equation*}
Applying the one-variable calculation in \cite[Section~2.2]{MRSTT}
successively in each coordinate yields
\begin{equation}
 \label{eq:Cinfty-norm-comparison}
 \|\phi\|_{\mathrm{coef},I}\ll_{d,t}1+\|\phi\|_{C^\infty(I)},
 \qquad \|\phi\|_{C^\infty(I)}\ll_{d,t}1+\|\phi\|_{\mathrm{coef},I}.
\end{equation}

\emph{Maximal sums and total variation norms.}
If $P$ is a finite arithmetic progression and $a\colon P\to\C$, define the maximal sum by
\begin{equation*}
 \left|\sum_{n\in P}a(n)\right|^*
 \coloneqq \sup_{P'\subseteq P}\left|\sum_{n\in P'}a(n)\right|,
\end{equation*}
where $P'$ ranges over all arithmetic subprogressions of $P$.  We define the maximal sum to be $0$ when $P=\varnothing$.  We define the total variation norms by
\begin{align*}
 \|a\|_{\TV(P)}
 &\coloneqq \sup_{n\in P}|a(n)|
+\sum_{j=1}^{m-1}|a(n_{j+1})-a(n_j)|,\\
 \|a\|_{\TV(P;q)}
 &\coloneqq \sum_{b\bmod q}\|a\|_{\TV(P\cap(b+q\Z))},
\end{align*}
where $n_1<\cdots<n_m$ are the elements of $P$.  Empty progressions are defined to have total variation zero.  For $a\colon\N\to\C$ and $q\geq1$, set
\begin{equation*}
 \|a\|_{\TV(\N;q)}
 \coloneqq \sup_{N\geq1}\|a\|_{\TV([N];q)}.
\end{equation*}
If an arithmetic progression $P$ is written in its canonical form
$P=\{a,a+r,\ldots,a+(N-1)r\}$ with $N\geq2$ and $r\geq1$, we define
\begin{equation}\label{eq:modP}
 \textnormal{mod}(P)\coloneqq r,\qquad \operatorname{span}(P)\coloneqq (N-1)r+1.
\end{equation}
For a singleton $P$, define
$\textnormal{mod}(P)=\operatorname{span}(P)=1$.

\section{Quantitative nilsequence inputs}

In this section we collect various quantitative nilsequence inputs needed for the main theorem:
the quasipolynomial inverse theorem, a vertical reduction, and Leng's efficient equidistribution theorem with a factorisation conclusion. We explicitly control the dependence on dimension in various bounds.

\subsection{Nilsequence conventions}

\begin{definition}[Filtered nilmanifolds, degree, and complexity]
\label{def:filtered-nilmanifolds}
Let $d\geq0$ and $m\geq0$ be integers and let $M\geq1$.  A filtered nilmanifold $G/\Gamma$ of degree at
most $d$, dimension $m$, and complexity at most $M$ consists of the following data:
\begin{itemize}
    \item A connected, simply connected nilpotent Lie group $G$, equipped with a filtration $G_{\bullet}=(G_{i})_{i\geq 0}$ of closed connected subgroups
\[
 G=G_{0}=G_{1}\geq G_{2}\geq\cdots \geq G_{d+1}=\{\1_{G}\}
\]
such that we have the commutator relations
\[
 [G_{i},G_{j}]\subseteq G_{i+j}\qquad\text{for all }i\geq0\text{ and }j\geq0.
\]

\item A discrete cocompact subgroup $\Gamma\leq G$ such that
$\Gamma\cap G_i$ is discrete and cocompact in $G_i$ for every $i\geq0$.

\item A Mal'cev basis
$\mathcal X=(X_1,\ldots,X_m)$ of the Lie algebra $\log G$, with
\[
 \log G_i
 =\operatorname{span}_{\R}(X_{m-\dim G_i+1},\ldots,X_m),
\]
for all $i\geq 0$, with the empty span interpreted as $\{0\}$.  When $m=0$, the basis is empty and $G/\Gamma$ is the one-point nilmanifold. Moreover, the basis $\mathcal X$ has complexity at most $M$ in the sense that
\[
 [X_i,X_j]=\sum_{\ell>\max(i,j)}c_{ij\ell}X_\ell,
\]
where every $c_{ij\ell}$ is rational with numerator and denominator at most $M$ in absolute value when written in lowest terms, and 
\[
 \Gamma=\left\{\exp(n_1X_1)\cdots\exp(n_mX_m)\colon 
  n_1,\ldots,n_m\in\Z
 \right\}.
\]
\end{itemize}
\end{definition}
We use the right-invariant Mal'cev metric $d_G$ of \cite[Definition 3.4]{LSS} and the induced quotient metric $d_{G/\Gamma}$ on $G/\Gamma$. For $F\colon G/\Gamma\to\C$, we write
\[
 \|F\|_{\Lip}=\|F\|_\infty+\sup_{x\ne y}\frac{|F(x)-F(y)|}{d_{G/\Gamma}(x,y)}.
\]
We say that $F$ is $M$-Lipschitz if $\|F\|_{\Lip}\leq M$.

We will also need higher order commutators $[w_1,\ldots, w_s]$ for $s\geq 2$ and $w_1,\ldots, w_s$ in a Lie group $G$. These are defined as the iterated bracket $[[[w_1,w_2],w_3],\ldots, w_s]$. We also set $[w_1]=w_1$.

We then introduce several standard definitions, following conventions of~\cite{Leng} unless otherwise stated.

\begin{definition}[Polynomial sequences]
\label{def:polynomial-rationality}
Let $t\geq 1$ be an integer. For a map $g\colon \Z^t\to G$ and $\mathbf h\in\Z^t$, write
$\partial_{\mathbf h}g(\mathbf n)=g(\mathbf n+\mathbf h)g(\mathbf n)^{-1}$.
The map $g$ is polynomial for $G_{\bullet}$ if
$$\partial_{\mathbf h_j}\cdots\partial_{\mathbf h_1}g(\mathbf n)\in G_{j}$$
for every integer $j\geq0$, every
$\mathbf n\in\Z^t$, and every $\mathbf h_1,\ldots,\mathbf h_j\in\Z^t$. We denote the set of such polynomials by $\poly(\mathbb{Z}^t,G)$. Whenever we refer to a polynomial sequence on a filtered Lie group, we mean a polynomial sequence for the specified filtration. 
\end{definition}

We remark that our convention for rational sequences is slightly weaker than that of~\cite[Definition 2.6]{Leng}. Namely, the exponent $r\leq Q$ may depend on the value of the sequence. Lemma \ref{lem:periodicrational} states the quantitative consequence needed below.

\begin{definition}[Rationality]
For $Q\geq1$, a connected closed subgroup $H\leq G$ is \emph{$Q$-rational} if $\log H$ has a basis whose coordinates in $X_1,\ldots,X_m$ are rational numbers of height at most $Q$.  An element $\gamma\in G$ is $Q$-rational if $\gamma^r\in\Gamma$ for an integer $r$ satisfying $1\leq r\leq Q$.  A sequence is $Q$-rational if all
of its values are $Q$-rational.

A surjective Lie group homomorphism $\pi\colon G\to G'$ is a rational quotient map if its kernel is a connected rational normal subgroup of $G$.  We equip the quotient nilmanifold with lattice $\pi(\Gamma)$ and filtration $(\pi(G_{i}))_{i\geq 0}$.
\end{definition}

\begin{definition}[Step]
\label{def:step}
We write
\begin{equation*}
 G^{(1)}=G,\qquad G^{(i+1)}=[G,G^{(i)}]
 \quad\text{for }i\geq1
\end{equation*}
for the lower central series of $G$.  The step is the largest $s$ for which $G^{(s)}$ is nontrivial, and if $G$ is the one point group, we define it as zero.  Thus a nontrivial abelian group has step one.  When the step is $s\geq1$, $G^{(s)}$ is connected and central.
\end{definition}

\begin{definition}[Characters]\label{def:characters}
Let $G/\Gamma$ be a nilmanifold. A horizontal character is a continuous homomorphism $\eta\colon G\to\R$ satisfying $\eta(\Gamma)\subseteq\Z$. The height $|\eta|$ is the maximum absolute value of its coordinates in the dual Mal'cev basis. 

Suppose that $G$ has step $s\geq1$. A \emph{vertical character} is a continuous homomorphism $\xi\colon G^{(s)}\to\mathbb R$ such that $\xi(\Gamma\cap G^{(s)})\subseteq\mathbb Z$. Since $\xi(\Gamma\cap G^{(s)})\subseteq\mathbb Z$, the map $h\mapsto\e(\xi(h))$ descends to a character of $G^{(s)}/(\Gamma\cap G^{(s)})$. We define the size $|\xi|$ of $\xi$ by
\begin{align*}
 |\xi|=\sup_{\substack{h,h'\in G^{(s)}\\h\Gamma\ne h'\Gamma}}\frac{|\e(\xi(h))-\e(\xi(h'))|}{d_{G/\Gamma}(h\Gamma,h'\Gamma)}.
\end{align*}
\end{definition}

\begin{definition}[Nilcharacters]
Let $G/\Gamma$ be a nilmanifold of step $s\geq1$, and let $\xi\colon G^{(s)}\to\mathbb R$ be a vertical character. A function $F\colon G/\Gamma\to\C$ is a vertical nilcharacter of frequency $\xi$ if
\begin{equation}
 \label{eq:def-vertical-character}
 F(hx)=\e(\xi(h))F(x)\qquad\text{for }h\in G^{(s)},\ x\in G/\Gamma.
\end{equation}
\end{definition}

\begin{definition}[Smooth sequences]
\label{def:smooth-sequences}
For an integer $t\geq1$ and a box $\mathbf I=\prod_{i=1}^t(x_i+[N_i])\subset\Z^t$ with $N_i\geq1$,
a sequence $\varepsilon\colon\Z^t\to G$ is \emph{$(Q,\mathbf I)$-smooth} if
\begin{equation}
 \label{eq:box-smoothness}
 d_G(\varepsilon(\mathbf n),1_G)\leq Q,\qquad
 d_G(\varepsilon(\mathbf n),\varepsilon(\mathbf n-\mathbf e_i))
 \leq Q/N_i
 \quad\text{for every }\mathbf n\in\mathbf I\text{ and }1\leq i\leq t.
\end{equation}
\end{definition}

It is instructive to compare this notion of smoothness with~\cite[Definition 2.6]{Leng}.
For $t=1$, Definition \ref{def:smooth-sequences} agrees up to a factor $2$ with \cite[Definition 2.6]{Leng}. To see this, write $I=a+[N]$ and note that the convention of~\cite{Leng} requires
\begin{align*}
 d_G(\varepsilon(a),1_G)\leq Q,\qquad d_G(\varepsilon(a+n),\varepsilon(a+n-1))\leq Q/N\quad\text{for }n\in[N].
\end{align*}
Summing these bounds implies $d_G(\varepsilon(a+n),1_G)\leq2Q$ on $I$.  Conversely, Definition \ref{def:smooth-sequences} gives
\begin{align*}
 d_G(\varepsilon(a),1_G)\leq d_G(\varepsilon(a),\varepsilon(a+1))+d_G(\varepsilon(a+1),1_G)\leq2Q,
\end{align*}
so a sequence that is $(Q,I)$-smooth in our convention is $(2Q,I)$-smooth in the convention of~\cite{Leng}.

We also have the following simple lemma that connects rationality and periodicity.

\begin{lemma}[Periodicity of rational polynomial sequences]\label{lem:periodicrational} Let $G/\Gamma$ be a nilmanifold of degree $k\geq 1$, dimension $D\geq 1$ and complexity $M\geq 1$, and let $Q\geq 1$. If $\gamma\colon \Z\to G$ is a $Q$-rational polynomial sequence, then it is periodic modulo $\Gamma$ with period of size at most $\max(M,Q)^{(D+1)^{C_k}}$ for some $C_k>0$. 
\end{lemma}

\begin{proof}
Let $\psi\colon G\to\R^D$ be the coordinate map associated with the fixed Mal'cev basis. If $Q<2$, then
$\gamma(n)\in\Gamma$ for every $n$, so period $1$ suffices. Otherwise, applying the proof of \cite[Lemma B.13(i)]{Leng} separately to $\gamma(0),\ldots,\gamma(k)$ and taking the
product of the resulting denominators yields an integer
\begin{align*}
 Q_1\leq\max(M,Q)^{(D+1)^{O_k(1)}}
\end{align*}
such that
\begin{align*}
 \psi(\gamma(n))\in Q_1^{-1}\Z^D\qquad\text{for }0\leq n\leq k.
\end{align*}

We write the binomial basis expansion of $\psi\circ \gamma$ as
\begin{align*}
\psi(\gamma(n))=\sum_{j=0}^k\alpha_j\binom{n}{j},\qquad \alpha_j=\Delta^j(\psi\circ\gamma)(0),
\end{align*}
where $\Delta f(n)\coloneqq f(n+1)-f(n)$.  Taking discrete differences in this expansion of the polynomial $\psi\circ\gamma$ therefore produces the same denominator bound for its coefficients. The proof of
\cite[Lemma B.14]{Leng} then gives the stated period bound after increasing $C_k$.
\end{proof}

We use the quasipolynomial inverse theorem of Leng--Sah--Sawhney \cite[Theorem 1.2]{LSS} in the following normalised form.

\begin{lemma}[Quasipolynomial inverse theorem]
\label{lem:LSS-inverse}
For every integer $k\geq2$ there is $A_k\geq1$ with the following property.  Let $0<\delta\leq1/2$, let $I\subset\Z$ be a nonempty finite interval, and let $f\colon I\to\C$ be $1$-bounded with
$$\|f\|_{U^k(I)}\geq\delta.$$
Then there exist a filtered nilmanifold $G/\Gamma$, a polynomial sequence $g\colon \Z\to G$, a function
$F\colon G/\Gamma\to\C$, an integer $D\geq2$, and real parameters $M\geq2$ and $0<\rho\leq1$ with the following properties.
\begin{itemize}[leftmargin=2.2em]
 \item The nilmanifold $G/\Gamma$ has degree and step at most $k-1$, dimension at most $D$ and complexity at most $M$.
 \item The parameters satisfy
 \begin{equation}
  \label{eq:LSS-parameters}
  D\leq\bigl(\log(2/\delta)\bigr)^{A_k},\qquad
  M\leq\exp\!\left(\bigl(\log(2/\delta)\bigr)^{A_k}\right),\qquad \rho^{-1}\leq\exp\!\left(\bigl(\log(2/\delta)\bigr)^{A_k}\right).
 \end{equation}
 \item The function $F$ is $1$-bounded and $M$-Lipschitz, and
 \begin{equation}\label{eq:LSS}
  \left|\E_{n\in I}f(n)F(g(n)\Gamma)\right|\geq\rho.
 \end{equation}
\end{itemize}
\end{lemma}

\begin{proof}
We translate $I$ to $[|I|]$ and apply \cite[Theorem 1.2]{LSS} with $s=k-1$.  That theorem produces a $K$-Lipschitz function $F'$ such that~\eqref{eq:LSS} holds with $\overline{F'}$ in place of $F$.  We then set $F=\overline{F'}/\max(1,\|\overline{F'}\|_\infty)$ and enlarge $M$ to $\max(M,K)$ to obtain~\eqref{eq:LSS} with $\rho$ replaced by $\rho/K$.  Since $\max(1,\|\overline{F'}\|_\infty)\leq\max(1,K)$, the
bounds in \eqref{eq:LSS-parameters} follow after increasing $A_k$.
\end{proof}

The next lemma either transfers large correlation with a nilsequence over an interval to a quotient nilmanifold of smaller step, or replaces $F$ by a vertical nilcharacter of nonzero frequency on a quotient whose last nontrivial lower central subgroup is one-dimensional, with quantitative bounds.

\begin{lemma}[Efficient vertical reduction]
\label{lem:vertical-reduction}
Let $k\geq2$ and $1\leq s\leq k-1$ be integers.  There is $C_k\geq1$ such that the following holds.  Let
$0<\rho\leq1/2$, let $I\subset\Z$ be a nonempty finite interval, and let $f\colon I\to\C$ be $1$-bounded.
Let $D\geq2$ be an integer and $M\geq2$, and let $G/\Gamma$ be an $s$-step filtered
nilmanifold of degree at most $k-1$, dimension at most $D$, and complexity at most $M$.  Suppose that
$g\colon \Z\to G$ is a polynomial, $F\colon G/\Gamma\to\C$ is $1$-bounded and $M$-Lipschitz, and
\begin{equation*}
 \left|\E_{n\in I}f(n)F(g(n)\Gamma)\right|\geq\rho.
\end{equation*}
Then there is $Q$ satisfying
\begin{equation}
 \label{eq:vertical-cost}
 2\leq Q\leq(M/\rho)^{D^{C_k}}
\end{equation}
such that one of the following cases holds. 
\begin{enumerate}[label=\textnormal{(\roman*)}]
 \item \emph{Zero frequency case.} Let $\pi_0\colon G\to G/G^{(s)}$ be the quotient map $g\mapsto gG^{(s)}$, and set $G_0=G/G^{(s)}$, $\Gamma_0=\pi_0(\Gamma)$, and $g_0=\pi_0\circ g$. Define
 \begin{equation}\label{eq:verticalaverage}
  F_0(\pi_0(x)\Gamma_0)=\int_{G^{(s)}/(G^{(s)}\cap\Gamma)}
          F(hx\Gamma)\,\mathrm dh,
 \end{equation}
 where $\mathrm dh$ is normalised Haar probability measure. The quotient nilmanifold $G_0/\Gamma_0$ (with filtration $(G_{0,i})_{i\geq 0}=(\pi_0(G_{i}))_{i\geq 0}$) has degree at most $k-1$, step at most
 $s-1$, dimension at most $D$, and complexity at most $Q$, while $F_0$ is $1$-bounded with Lipschitz norm at most $Q$.  Moreover, we have
 \begin{equation*}
  \left|\E_{n\in I}f(n)F_0(g_0(n)\Gamma_0)\right|\geq\rho/Q.
 \end{equation*}
 \item \emph{Nonzero frequency case.}  There exist a rational quotient map $\pi\colon G\to G'$ and a quotient
 nilmanifold $G'/\Gamma'$, where $\Gamma'=\pi(\Gamma)$ and $G'_{i}=\pi(G_{i})$, of degree at most $k-1$, step
 $s$, dimension at most $D$, and complexity at most $Q$, whose $s$th lower central subgroup $(G')^{(s)}$ is one-dimensional, such that for some $1$-bounded $Q$-Lipschitz vertical nilcharacter $F'\colon G'/\Gamma'\to\C$ with a nonzero frequency $\xi$ we have
 \begin{equation*}
  |\xi|\leq Q,\qquad \left|\E_{n\in I}f(n)F'(g'(n)\Gamma')\right|\geq\rho/Q,
 \end{equation*}
 where $g'=\pi\circ g$.
\end{enumerate}
\end{lemma}

\begin{proof}
The subgroup $G^{(s)}$ is $M^{D^{O_k(1)}}$-rational by the remark following \cite[Lemma A.6]{Leng}. By \cite[Lemma 3.10]{LSS}, quotienting by $G^{(s)}$ gives complexity bounds of the same form, and the corresponding Lipschitz bound for functions invariant under $G^{(s)}$.

We define $F^{(0)}(x\Gamma)=F_0(\pi_0(x)\Gamma_0)$ using the vertical average in~\eqref{eq:verticalaverage}. This average is invariant under left translation by \(G^{(s)}\) by the invariance of the Haar measure, so it defines a function on $G/G^{(s)}$. The function $F^{(0)}$ is $1$-bounded, and has Lipschitz norm at
most $M$.  If the normalised correlation of $f$ with $F^{(0)}$ is at least $\rho/2$, \cite[Lemma 3.10]{LSS} gives case~\textnormal{(i)}, after enlarging $Q$ to satisfy \eqref{eq:vertical-cost}.

Otherwise, the normalised correlation of $f$ with $(F-F^{(0)})/2$ is at least $\rho/4$.  This function is $1$-bounded, and it has Lipschitz norm at most $M$ and zero average on every
$G^{(s)}$-orbit.  We apply the vertical Fourier expansion of \cite[Lemma A.6]{Leng} to the action of
$G^{(s)}$ with uniform approximation error $\rho/100$.  The zero Fourier coefficient vanishes, so the approximation has the form
\begin{equation*}
 \frac{F-F^{(0)}}2=E+\sum_{\xi\in\Xi}F_\xi,\qquad
 \|E\|_\infty\leq\rho/100,\qquad F_\xi(hx)=\e(\xi(h))F_\xi(x),
\end{equation*}
where every $\xi$ is nonzero, and the number and sizes of the frequencies, the Lipschitz norms of the $F_\xi$, and
$\sum_{\xi\in\Xi}\|F_\xi\|_\infty$ are all at most $(M/\rho)^{D^{C_k}}$ after increasing $C_k$. 
By the triangle inequality, 
\begin{equation*}
 \sum_{\xi\in\Xi}\left|\E_{n\in I}f(n)F_\xi(g(n)\Gamma)\right|
 \geq\rho/4-\rho/100\geq\rho/5.
\end{equation*}
Hence, for some $\xi\in\Xi$, the function
$\widetilde F=F_\xi/\max(1,\|F_\xi\|_\infty)$ satisfies
\begin{equation*}
 \left|\E_{n\in I}f(n)\widetilde F(g(n)\Gamma)\right|\geq\rho/(5W),
\end{equation*}
where 
\begin{equation*}
 W=\sum_{\xi\in\Xi}\max(1,\|F_\xi\|_\infty)\ll (M/\rho)^{D^{C_k}}.
\end{equation*}

To apply \cite[Lemma 2.2]{Leng}, set
\begin{equation*}
 K=\ker\xi\leq G^{(s)},\qquad G'=G/K,\qquad \Gamma'=\pi(\Gamma),
\end{equation*}
where $\pi\colon G\to G/K$ is the quotient map $g\mapsto gK$. Since $G^{(s)}$ is connected and $\xi$ is integer-valued on $G^{(s)}\cap\Gamma$, its kernel $K$ is a connected rational subgroup. Since $K\subseteq G^{(s)}$ and $G^{(s)}$ is central, $K$ is a normal subgroup of $G$. Moreover, $\widetilde F(kx)=\widetilde F(x)$ for $k\in K$, since $\xi(k)=0$. Hence $\widetilde F$ descends to $G/K$, while $\xi$ descends to a nonzero character $\xi'$ on $G^{(s)}/K$. As $K=\ker\xi$, the quotient $G^{(s)}/K$ is one-dimensional. Finally,
\[
(G/K)^{(s)}=G^{(s)}/K,
\]
so this is precisely the last nontrivial lower central subgroup of $G'$. 

Now, the formulas
\begin{equation*}
 \begin{gathered}
 F'(\pi(x)\Gamma')=\widetilde F(x\Gamma),\qquad g'=\pi\circ g,\\
 \xi'(\pi(h))=\xi(h)\quad\text{for all }h\in G^{(s)}
 \end{gathered}
\end{equation*}
define a function, a polynomial sequence, and a nonzero vertical frequency on the quotient $G'/\Gamma'$.  Taking $Q\geq5W$ and applying the cited lemma gives the bounds in case~\textnormal{(ii)}, after increasing $C_k$ in \eqref{eq:vertical-cost}.
\end{proof}

The following definition relating to factorisations of polynomial sequences will appear in several key statements later on.

\begin{definition}[Factorisation on an interval]
\label{def:factorisation-package}
Let $G/\Gamma$ be a filtered nilmanifold, let $I=u+[N]$ with $N\geq1$, and let $g\in\poly(\Z,G)$. For $Q\geq2$, a
\emph{$Q$-factorisation} of $g$ on $I$ through $H\leq G$ consists of polynomial sequences $\varepsilon,h,\gamma\colon\Z\to G$ satisfying
\begin{equation*}
 g(n)=\varepsilon(n)h(n)\gamma(n)\quad\text{for every }n\in I,\qquad
 h\in\poly(\Z,H),
\end{equation*}
where $H$ is connected and $Q$-rational, $\varepsilon$ is $(Q,I)$-smooth, and $\gamma$ is $Q$-rational.
\end{definition}

\begin{lemma}[Factorisation through a common kernel]
\label{lem:character-factorisation}\label{lem:commonkernel}
Let $d\geq1$ and $t\geq1$ be integers. There is $C=C(d,t)\geq1$ with the following property. Let $D\geq2$ be an integer, let $M\geq2$ and $Q\geq2$, and let $G/\Gamma$ be a filtered nilmanifold of degree at most $d$, dimension at most $D$, and complexity at most $M$. Let $g\colon\Z^t\to G$ be polynomial, and let
\begin{align*}
 \mathbf I=\prod_{i=1}^t(x_i+[N_i]),\qquad \mathbf x\in\Z^t,\qquad N_i\in\N\quad\text{for every }1\leq i\leq t.
\end{align*}
Let $r\geq1$ be an integer, and suppose that $\eta_1,\ldots,\eta_r$ are linearly independent nonzero horizontal characters satisfying
\begin{align*}
 |\eta_j|\leq Q,\qquad \|\eta_j\circ g\|_{C^\infty(\mathbf I)}\leq Q
 \quad\text{for every }1\leq j\leq r.
\end{align*}
Set
\begin{align*}
 R=(MQ)^{D^C},\qquad H=\bigcap_{j=1}^r\ker\eta_j.
\end{align*}
Then $H$ is connected and $R$-rational, and there are polynomial sequences $\varepsilon,\gamma\colon\Z^t\to G$ and $h\in\poly(\Z^t,H)$ such that
\begin{align*}
 g(\mathbf n)=\varepsilon(\mathbf n)h(\mathbf n)\gamma(\mathbf n)\quad\text{for every }\mathbf n\in\mathbf I,
\end{align*}
where $H$ carries the induced filtration $(H\cap G_i)_{i\geq0}$, $\varepsilon$ is $(R,\mathbf I)$-smooth, and $\gamma$ is $R$-rational. The nilmanifold $H/(H\cap\Gamma)$ has degree at most $d$, dimension at most $D$, and complexity at most $R$ in an induced rational basis. When $t=1$, this is an $R$-factorisation of $g$ on $\mathbf I$, and $\gamma(\cdot)\Gamma$ has period at most $R$.
\end{lemma}

\begin{proof}
The exponential map $\exp\colon\log G\to G$ is a diffeomorphism since $G$ is connected, simply connected and nilpotent. We have $\eta(\exp X)=\mathrm d\eta(X)$ for $X\in\log G$, so $\exp X\in\ker\eta$ if and only if $X\in\ker\mathrm d\eta$. Hence
\begin{equation*}
 \bigcap_{j=1}^r\ker\eta_j=\exp\left(\bigcap_{j=1}^r\ker \mathrm d\eta_j\right).
\end{equation*}
The intersection on the right is a vector subspace of $\log G$, and hence is connected. Hence, the common kernel is connected as an exponential of a connected set. 

We set $g_{\mathbf x}(\mathbf m)=g(\mathbf x+\mathbf m)$. Set $T=C_{d,t}(1+Q)$, where $C_{d,t}$ is
large enough for \eqref{eq:Cinfty-norm-comparison}. We apply the factorisation theorem of \cite[Lemma A.1]{Leng} to $g_{\mathbf x}$ and the independent characters with size bound $Q$ and inverse smoothness parameter
$\delta_A^{-1}=T$. That lemma furnishes polynomial sequences $\varepsilon_{\mathbf x},h_{\mathbf x},\gamma_{\mathbf x}$, where $h_{\mathbf x}\in\poly(\Z^t,H)$ and
\begin{equation*}
 g_{\mathbf x}(\mathbf m)=\varepsilon_{\mathbf x}(\mathbf m)
 h_{\mathbf x}(\mathbf m)\gamma_{\mathbf x}(\mathbf m)\quad\text{for every }\mathbf m\in\prod_{i=1}^t[N_i].
\end{equation*}
By \cite[Lemma A.1]{Leng}, $\gamma_{\mathbf x}$ is $R_0$-rational, $\varepsilon_{\mathbf x}$ is $(R_0,(N_1,\ldots,N_t))$-smooth in Leng's convention, and the induced nilmanifold has complexity at most
\begin{equation*}
 R_0=(C_{d,t}MQT)^{D^{c'_{d,t}}}
\end{equation*}
for a suitable $c'_{d,t}>0$. Since
$\log H=\bigcap_{j=1}^r\ker\mathrm d\eta_j$ and the independent integer forms $\mathrm d\eta_j$ have coefficients bounded by $Q$, Cramer's rule gives a basis of $\log H$ with rational coordinate heights at most $(DQ)^{O(D)}\leq R_0$. Thus $H$ is $R_0$-rational. We translate the three polynomial sequences by $\mathbf x$ to obtain
the stated factorisation.

To verify \eqref{eq:box-smoothness}, we use the coefficient bounds in the construction of $\varepsilon_{\mathbf x}$ in \cite[Lemma A.1]{Leng}. These bounds and the metric comparison in \cite[Lemma B.3]{Leng} give, after increasing $c'_{d,t}$ in the definition of $R_0$,
\begin{align}\label{eq:dG}
 d_G(\varepsilon_{\mathbf x}(\mathbf0),1_G)\leq R_0,\qquad d_G(\varepsilon_{\mathbf x}(\mathbf m),\varepsilon_{\mathbf x}(\mathbf m-\mathbf e_i))\leq R_0/N_i
\end{align}
whenever $1\leq m_i\leq N_i$ and $0\leq m_j\leq N_j$ for $j\neq i$. We can take a coordinate path from $\mathbf0$ to any $\mathbf m\in\prod_i[N_i]$, and summing~\eqref{eq:dG} along this path implies
\begin{align*}
 d_G(\varepsilon_{\mathbf x}(\mathbf m),1_G)\leq R_0+\sum_{i=1}^tm_iR_0/N_i\leq(t+1)R_0.
\end{align*}
Translating by $\mathbf x$ gives $(R,\mathbf I)$-smoothness after enlarging $R$. 

When $t=1$, Lemma \ref{lem:periodicrational} bounds the period of $\gamma(\cdot)\Gamma$ by $(MR_0)^{D^{c_d}}$. Increasing $C=C(d,t)$ so that
\begin{equation*}
 Q\leq R,\qquad
 (t+1)R_0\leq R,\qquad
 (MR_0)^{D^{c_d}}\leq R,
\end{equation*}
we obtain all the asserted bounds. 
\end{proof}

After applying Lemma~\ref{lem:LSS-inverse}, an important input for obtaining quantitative Gowers norm bounds is Leng's efficient equidistribution theorem. We state it in a form that additionally gives a factorisation of the relevant polynomial sequence. 

\begin{theorem}[Leng's efficient equidistribution with factorisation]
\label{thm:Leng-equidistribution}
Let $d\geq 1$ and $t\geq 1$ be integers. There is $C=C(d,t)\geq1$ with the following property. Let $D\geq2$ be an integer and let $M\geq2$. Let $G/\Gamma$ have degree at most $d$, step $s\geq1$, dimension at most
$D$, and complexity at most $M$. Let $g\colon\Z^t\to G$ be polynomial. Write
\begin{equation*}
 \mathbf I=\prod_{i=1}^t(x_i+[N_i]),\qquad \mathbf x=(x_1,\ldots,x_t)\in\Z^t,\qquad
 N_i\geq2\quad\text{for every }1\leq i\leq t.
\end{equation*}
Let $0<\delta<1/10$, and let $F\colon G/\Gamma\to\C$ be a $1$-bounded, $M$-Lipschitz vertical nilcharacter with nonzero frequency $\xi$ satisfying $|\xi|\leq M/\delta$. Suppose that
\begin{equation}
 \label{eq:Lengdelta}
 \left|\E_{\mathbf n\in\mathbf I}F(g(\mathbf n)\Gamma)\right|\geq\delta.
\end{equation}
Set $Q=(M/\delta)^{D^C}$. Then one of the following holds:
\begin{enumerate}[label=\textnormal{(\roman*)}]
    \item We have 
    $$\min_iN_i\leq Q.$$
\item There 
are a connected $Q$-rational subgroup $G_*\leq G$ and polynomial sequences $\varepsilon,\gamma\colon\Z^t\to G$ and
$g_*\in\poly(\Z^t,G_*)$ such that
\begin{equation}
 \label{eq:Leng-factorisation}
 g(\mathbf n)=\varepsilon(\mathbf n)g_*(\mathbf n)\gamma(\mathbf n)
 \quad\text{for every }\mathbf n\in\mathbf I.
\end{equation}
The nilmanifold $G_*/(G_*\cap\Gamma)$, with the induced filtration $(G_{*,i})_{i\geq 0}=(G_*\cap G_i)_{i\geq 0}$, has degree at most $d$, dimension at most $D$, and complexity at most $Q$ in an induced rational basis. The sequence $\gamma$ is $Q$-rational, $\varepsilon$ is $(Q,\mathbf I)$-smooth in Definition \ref{def:smooth-sequences}, and when $t=1$, $\gamma(\cdot)\Gamma$ is periodic with
period at most $Q$. When $t=1$, \eqref{eq:Leng-factorisation} is a $Q$-factorisation of $g$ on $\mathbf I$ through $G_*$. Moreover, for every $s$-tuple of elements $w_1,\ldots,w_s$ of $G_*$, we have
\begin{equation}
 \label{eq:Leng-bracket-conclusion}
 \xi([w_1,\ldots,w_s])=0.
\end{equation}
If $G^{(s)}$ is one-dimensional, then $G_*$ has step at most $s-1$.
\end{enumerate}
\end{theorem}

\begin{proof}
We apply \cite[Theorem 4]{Leng} to $g_{\mathbf x}(\mathbf m)=g(\mathbf x+\mathbf m)$ and $F$ on
$[N_1]\times\cdots\times[N_t]$. Let $Q_0=(M/\delta)^{D^{c_{d,t}}}$ dominate the interval length and
horizontal character bounds of that theorem. That result either gives 
$$\min_iN_i\leq Q_0\leq Q,$$ which suffices, or it produces nonzero horizontal characters $\eta_1,\ldots,\eta_r$ with $1\leq r\leq D$ such that 
\begin{equation*}
 |\eta_j|\leq Q_0,\qquad\|\eta_j\circ g\|_{C^\infty(\mathbf I)}\leq Q_0
 \quad\text{for }1\leq j\leq r,
\end{equation*}
 using \eqref{eq:Cinfty-norm-comparison}. We take a maximal linearly independent
sublist of $\eta_1,\ldots,\eta_r$ and continue to denote this list by the same notation. Since horizontal characters are real-valued homomorphisms, the common kernel
\begin{equation*}
G_*=\bigcap_{j=1}^r\ker\eta_j
\end{equation*}
is unchanged by the operation of taking this sublist. The same theorem gives
\eqref{eq:Leng-bracket-conclusion} for $G_*$.

We apply Lemma \ref{lem:character-factorisation} to $g$ on $\mathbf I$ with character bound $Q_0$. That lemma gives \eqref{eq:Leng-factorisation}, with $G_*$ connected and all rationality, complexity and smoothness parameters and period bounded by
\begin{align*}
 (MQ_0)^{D^{C_{d,t}}}\leq Q,
\end{align*}
after increasing $C=C(d,t)$ in the definition of $Q$. When $t=1$, the same applies to the rational coset period.

Finally, when $G^{(s)}$ is one-dimensional, the nonzero homomorphism $\xi$ is injective on $G^{(s)}$. Every $s$-fold commutator in $G_*$ belongs to $G^{(s)}$, so \eqref{eq:Leng-bracket-conclusion} implies that it is
trivial.
\end{proof}

\section{Type I and Type \texorpdfstring{$I_2$}{I2} estimates in short intervals}

In this section we establish versions of the Type I and Type $I_2$ arguments in \cite[Sections~5 and~8]{MSTT} that have an explicit dependence on dimension.  We first formulate the precise convolution classes used in the decomposition of the von Mangoldt function. 

\begin{definition}[Type I, Type II, and Type $I_2$ sums]
\label{def:type-sums}
Let $0<\delta<1$, let $A_I\geq1$ and $A_{I_2}\geq1$, and let $1\leq A_{II}^-\leq A_{II}^+$.  For arithmetic functions $a,b\colon \N\to\C$, write
\[
 (a*b)(n)=\sum_{d\mid n}a(d)b(n/d).
\]
\begin{enumerate}[label=\textnormal{(\roman*)}]
 \item A $(\delta,A_I)$ Type I sum is an arithmetic function $f=\alpha*\beta$ such that $\alpha$ is supported on $[1,A_I]$ and satisfies
 \begin{equation}
  \label{eq:type-alpha-second-moment}
  \sum_{n\leq Y}|\alpha(n)|^2\leq\delta^{-1}Y\qquad\text{for }Y\geq1,
 \end{equation}
 and for some integer $1\leq q\leq\delta^{-1}$, we have the total variation bound
 \begin{equation}
  \label{eq:type-I-TV}
  \|\beta\|_{\TV(\N;q)}\leq\delta^{-1}.
 \end{equation}
 \item A $(\delta,A_{II}^-,A_{II}^+)$ Type II sum is an arithmetic function $f=\alpha*\beta$ such that $\alpha$ is supported on $[A_{II}^-,A_{II}^+]$, satisfies \eqref{eq:type-alpha-second-moment}, and
 \begin{equation}
  \label{eq:type-II-moments}
  \sum_{n\leq Y}|\beta(n)|^4\leq\delta^{-2}Y\qquad\text{for } Y\geq1.
  \end{equation}
 \item A $(\delta,A_{I_2})$ Type $I_2$ sum is an arithmetic function $f=\alpha*\beta_1*\beta_2$ such that $\alpha$ is supported on $[1,A_{I_2}]$ and satisfies \eqref{eq:type-alpha-second-moment},
 and for some integer $1\leq q\leq\delta^{-1}$, we have the total variation bounds
 \begin{equation}
  \label{eq:type-I2-TV}
  \|\beta_1\|_{\TV(\N;q)}\leq\delta^{-1},
  \qquad \|\beta_2\|_{\TV(\N;q)}\leq\delta^{-1}.
 \end{equation}
\end{enumerate}
\end{definition}

We note for later use that \eqref{eq:type-II-moments} and Cauchy--Schwarz immediately give
\begin{align*}
 \sum_{n\leq Y}|\beta(n)|^2\leq\delta^{-1}Y.
\end{align*}

\subsection{The Type I case}

We first isolate the key nilsequence input in the proof, which is concatenation of local factorisations on dilated
progressions. 

\begin{lemma}[Concatenation of dilated factorisations]
\label{lem:concatenate-dilated-factorisations}
Let $d\geq1$ be an integer. There is $C_d\geq1$ with the following property. Let $D\geq2$, $Q\geq M\geq2$, and $1\leq s\leq d$, and let $G/\Gamma$ be an $s$-step filtered nilmanifold of degree at most $d$,
dimension at most $D$, and complexity at most $M$. Let $g\colon\Z\to G$ be polynomial, $X\geq2$, $2\leq H\leq X$, and $A\geq1/2$. Suppose that $\mathcal A\subset(A,2A]\cap\N$ satisfies $|\mathcal A|\geq A/Q$. For every $a\in\mathcal A$, let $r_a\in\Z$ and $q_a,N_a\in\N$ satisfy
\begin{equation}
 \label{eq:dilated-factorisation-progression}
 B_a=r_a+q_a[N_a]\subset(X/a,(X+H)/a],\qquad 1\leq q_a\leq Q,\qquad N_a\geq H/(AQ),
\end{equation}
and suppose that $n\mapsto g(a(r_a+q_an))$ has a $Q$-factorisation on $[N_a]$ through a subgroup $G_a$ of step at most $s-1$,  written
\begin{equation}
 \label{eq:dilated-factorisation-input}
 g(a(r_a+q_an))=\varepsilon_a(n)g_a(n)\gamma_a(n).
\end{equation}

Set $R=Q^{D^{C_d}}$. Then either $H\leq RA$, or $g$ has an $R$-factorisation on $(X,X+H]$ through a subgroup $G_*$ of step at most $s-1$, written
\begin{equation}
 \label{eq:dilated-factorisation-output} g=\varepsilon g_*\gamma.
\end{equation}
The sequence $\gamma(\cdot)\Gamma$ has period at most $R$, and $G_*/(G_*\cap\Gamma)$, with the induced filtration, has degree at most $d$, dimension at most $D$, and complexity at most $R$ in an
induced rational basis.
\end{lemma}

\begin{proof}
We allow all implicit constants in this proof depend to depend on $d$ but on no other parameters. If $N_a<2d+2$ for some $a$, then from~\eqref{eq:dilated-factorisation-progression} we have
\[
 H\leq AQ N_a < (2d+2)AQ\leq RA
\]
after enlarging $C_d$. Hence we may assume $N_a\geq2d+2$ for every
$a\in\mathcal A$.

\emph{Recovering characters.} We first show that $G_a$ is contained in the common kernel of some horizontal characters. 
To this end, we apply \cite[Lemma A.9]{Leng} to the $Q$-rational subgroup $G_a$. This gives us linearly independent horizontal characters $\eta_{a,1},\ldots,\eta_{a,t_a}$ with $1\leq t_a\leq D$ such that
\begin{equation}
 \label{eq:dilated-recovered-kernel}
 |\eta_{a,j}|\leq Q_0\coloneqq Q^{D^{c_0}}, \qquad \eta_{a,j}(G_a)=0,
 \qquad K_a\coloneqq\bigcap_{j=1}^{t_a}\ker\eta_{a,j}
 \text{ has step at most }s-1,
\end{equation}
where $c_0=c_0(d)$ is sufficiently large. The list is nonempty since $G$ has step $s$.

For each $0\leq i\leq d$, we choose $1\leq m_{a,i}\leq Q$ with $\gamma_a(i)^{m_{a,i}}\in\Gamma$, and set
\[
 \ell_a=\operatorname{lcm}(m_{a,0},\ldots,m_{a,d}) \leq Q^{d+1}.
\]
For any horizontal character $\eta$ the real polynomial $v(n)=\eta(\gamma_a(n))$ has degree at most $d$, and
$\ell_av(i)\in\Z$ for $0\leq i\leq d$. Thus
\begin{equation}
 \label{eq:dilated-rational-clearing}
 \ell_av(n)=\sum_{j=0}^d\binom nj\,
\ell_a\partial^jv(0)\in\Z
 \qquad\text{for every }n\in\Z,
\end{equation}
since $\ell_a\partial^jv(0)$ is an integer linear combination of $\ell_av(0),\ldots,\ell_av(j)$.

We set $u_{a,j}=\eta_{a,j}\circ\varepsilon_a$, which is a real polynomial of degree at most $d$. A horizontal character $\eta$ of height $|\eta|$ is $D|\eta|$-Lipschitz for the Mal'cev metric. Hence the smoothness of $\varepsilon_a$ implies
\begin{align*}
 |\partial u_{a,j}(n)|\leq D|\eta_{a,j}|d_G(\varepsilon_a(n+1),\varepsilon_a(n))\leq DQ_0Q/N_a\qquad\text{for }0\leq n<N_a.
\end{align*}
Since $\partial u_{a,j}$ has degree at most $d-1$ and $N_a\geq2d+2$, interpolation bounds the coefficient of $n^r$ in $\partial u_{a,j}(n)$ by $\ll_d DQ_0Q N_a^{-r-1}$ for every $0\leq r\leq d-1$, so we have
\begin{align*}
 \max_{1\leq r\leq d}N_a^r\max_{n\in[N_a]}|\partial^ru_{a,j}(n)|\ll_d DQ_0Q.
\end{align*}

Now, applying $\ell_a\eta_{a,j}$ to \eqref{eq:dilated-factorisation-input}, the $g_a$ term vanishes by \eqref{eq:dilated-recovered-kernel} and the $\gamma_a$ term is integer-valued by \eqref{eq:dilated-rational-clearing}. Hence
\begin{equation}
 \label{eq:dilated-local-character-smoothness}
 \left|\ell_a\eta_{a,j}\right|\leq Q_1,\qquad \left\|\ell_a\eta_{a,j}\circ g(a(r_a+q_a\cdot))\right\|_{C^\infty([N_a])} \leq Q_1,\qquad Q_1=Q^{D^{c_1}},
\end{equation}
where $c_1=c_1(d)>c_0$ is large enough.

\emph{Affine extrapolation.}
Write $x_a=\lfloor X/a\rfloor$ and
\[
 (X/a,(X+H)/a]\cap\Z=x_a+[L_a]
\]
for some $L_a\geq 1$.
Since $B_a\subseteq x_a+[L_a]$ and $a>A$, we have
\begin{equation}
 \label{eq:dilated-affine-parameters}
 N_a\leq L_a\leq H/a+1\leq QN_a+1\leq2QN_a, \qquad |r_a-x_a|\leq L_a+q_a\leq3QN_a.
\end{equation}
Let
$P_{a,j}(n)=\ell_a\eta_{a,j}(g(a(x_a+n)))$. Then the polynomial $P_{a,j}(r_a-x_a+q_an)$ has $C^\infty([N_a])$ norm at most $Q_1$. By the multiparameter extrapolation result of \cite[Lemma A.12]{Leng} and~\eqref{eq:dilated-affine-parameters}, there is a nonzero integer $v_{a,j}$ such that
\begin{equation}
 \label{eq:dilated-affine-extrapolation}
 |v_{a,j}|\leq Q^{O_d(1)},\qquad \|v_{a,j}P_{a,j}\|_{C^\infty([N_a])} \leq Q_1Q^{O_d(1)}
\end{equation}
for some $c_2=c_2(d)>0$. Since $N_a\leq L_a\leq2QN_a$, \cite[(2.3)]{MRSTT} and \eqref{eq:dilated-affine-extrapolation} give
\[
\|v_{a,j}P_{a,j}\|_{C^\infty([L_a])}
\ll_d (L_a/N_a)^d Q_1Q^{O_d(1)}
\leq Q^{D^{c_2}}.
\]
This can be rewritten as
\begin{equation}
 \label{eq:dilated-full-local-smoothness}
 \begin{gathered}
  |\eta'_{a,j}|\leq Q_2,\qquad \|\eta'_{a,j}\circ g(a\cdot)\|_{C^\infty((X/a,(X+H)/a])}
  \leq Q_2,\\
 Q_2=Q^{D^{c_2}},\qquad \eta'_{a,j}=v_{a,j}\ell_a\eta_{a,j}.
 \end{gathered}
\end{equation}
Since multiplication by nonzero integers preserves linear independence and kernels, the common kernel of the list
$\mathcal L'_a=(\eta'_{a,1},\ldots,\eta'_{a,t_a})$ is $K_a$.

\emph{A common list and scalar concatenation.}
Each list $\mathcal{L}_{a}'$ has at most $D$ characters, with integer coordinates bounded by $Q_2$ in the fixed Mal'cev basis. There are at most $(2Q_2+2)^{D^2}$ such ordered lists. Pigeonholing produces a list
$\mathcal L$ and a subset $\mathcal A'\subseteq\mathcal A$ such that
\begin{equation}
 \label{eq:dilated-common-list}
 |\mathcal A'|\geq
 \frac{A}{Q(2Q_2+2)^{D^2}}
 \geq A/Q_3,\qquad Q_3=Q^{D^{c_3}},
\end{equation}
where $c_3=c_3(d)>c_2$ is large enough. Since the common kernel of $\mathcal L$ equals $K_a$ for every $a\in\mathcal A'$, it has step at most $s-1$.

Let $I=(X,X+H]\cap\Z$ and $N=|I|$, and set $Q_4=Q^{D^{c_4}}$, where $c_4=c_4(d)>c_3$ is sufficiently large.
If $A<1$, then $\mathcal A'=\{1\}$, so the interval in~\eqref{eq:dilated-full-local-smoothness} is $I$. Thus
\eqref{eq:dilated-full-local-smoothness} yields the required bound
for each $\eta\in\mathcal L$ with $v_\eta=1$.
Suppose now that $A\geq1$. If $N<2A$, then
$H<N+1<2A+1\leq Q_4A$, giving one of the desired alternatives. If $N\geq2A$, we apply the concatenation result of \cite[Corollary 2.4]{MSTT} to each degree $d$ scalar polynomial $\eta\circ g$ on $I$, with density
$\delta_2=1/Q_3$. Equations \eqref{eq:dilated-full-local-smoothness} and \eqref{eq:dilated-common-list} verify the hypotheses of that lemma. We conclude that either $H\leq Q_4A$, which is a desired alternative, or that for each $\eta\in\mathcal L$ there is a nonzero
integer $v_\eta$ with
\begin{equation}
 \label{eq:dilated-global-character-smoothness}
 |v_\eta\eta|\leq Q_4,\qquad
 \|v_\eta\eta\circ g\|_{C^\infty(I)}\leq Q_4.
\end{equation}
For $A<1$, the same bound holds with $v_\eta=1$. The characters
$v_\eta\eta$ remain independent, and since $\ker(q\eta)=\ker(\eta)$ for any nonzero integer $q$, from~\eqref{eq:dilated-recovered-kernel} we know that the common kernel $G_{*}$ is of step at most $s-1$.

We apply Lemma \ref{lem:character-factorisation} to $g$ on $I$ with $t=1$ and the independent characters in \eqref{eq:dilated-global-character-smoothness}, taking $Q_4$ as the character bound. It gives a factorisation through $G_*$, with all rationality, complexity and smoothness parameters bounded by
\[
 (MQ_4)^{D^{C(d,1)}}\leq Q^{D^{C_d}},
\]
after increasing $C_d$ in the statement.
\end{proof}

We now prove the Type I proposition, which is also used in the Type
$I_2$ argument.

\begin{proposition}[Efficient Type I sums in short intervals]
\label{prop:typeI-standalone}
Let $k\geq2$ be an integer. There is $C_0=C_0(k)\geq1$ with the following property. Let $X\geq2$, $2\leq H\leq X$, $0<\delta\leq1/10$, and let $D\geq2$ and $1\leq s\leq k-1$ be integers. Let $M\geq2$ and let $G/\Gamma$ be an $s$-step filtered nilmanifold of degree at most $k-1$, dimension at most $D$, and complexity at most $M$, with $G^{(s)}$ one-dimensional. Let $g\colon\Z\to G$ be polynomial and let $F\colon G/\Gamma\to\C$ be a $1$-bounded, $M$-Lipschitz vertical nilcharacter of nonzero frequency $\xi$ satisfying $|\xi|\leq M/\delta$. Suppose that $f$ is a $(\delta,A_I)$ Type I sum, where $A_I\geq1$,
and that
\begin{equation}\label{eq:maxsumlarge}
 \left|\sum_{X<n\leq X+H}f(n)F(g(n)\Gamma)\right|^*
 \geq\delta H.
\end{equation}
Set
\begin{equation}
 \label{eq:typeI-single-parameter}
 \mathcal R=\left(\frac{M\log X}{\delta}\right)^{D^{C_0}}.
\end{equation}
Then at least one of the following cases holds.
\begin{itemize}[leftmargin=2.3em]
 \item \emph{Type I range.} We have
 \begin{equation}
  \label{eq:typeI-scale}
  H\leq \mathcal R A_I.
 \end{equation}
 \item \emph{Step descent.} The sequence $g$ has an $\mathcal R$-factorisation on $(X,X+H]$ through a subgroup  $G_*\leq G$ of step at most $s-1$, written
 \begin{equation}\label{eq:factor}
  g=\varepsilon g_*\gamma.
 \end{equation}
 The sequence $\gamma(\cdot)\Gamma$ has period at most $\mathcal R$. The nilmanifold $G_*/(G_*\cap\Gamma)$, with the induced filtration $(G_{*,i})_{i\geq 0}=(G_*\cap G_i)_{i\geq 0}$, has degree at most $k-1$, dimension at most $D$, and complexity at most $\mathcal R$ in an induced rational basis.
\end{itemize}
\end{proposition}

\begin{proof}
We set
\begin{equation*}
 f=\alpha*\beta,\qquad Z=M(\log X)/\delta.
\end{equation*}
All constants introduced in this proof depend only on $k$.

\emph{Dyadic reduction.}
Let $P_0\subset(X,X+H]$ attain the maximal sum in~\eqref{eq:maxsumlarge}, and define
\begin{equation}
 \label{eq:typeI-row-sum}
 S_a=\sum_{\substack{b\geq1\\ab\in P_0}}
 \beta(b)F(g(ab)\Gamma).
\end{equation}
We partition the positive integers into dyadic blocks $(A,2A]$, where $A=2^j$ and $j\geq-1$.  There are $O(\log X)$ nonempty blocks meeting the support of $\alpha$ in this sum.  At least one block, with
$A\leq A_I$, satisfies
\begin{equation*}
 \sum_{A<a\leq2A}|\alpha(a)S_a|
 \gg\delta H/\log X.
\end{equation*}
Cauchy--Schwarz and \eqref{eq:type-alpha-second-moment} give
\begin{equation}
 \label{eq:typeI-row-second-moment}
 \sum_{A<a\leq2A}|S_a|^2
 \gg\frac{\delta^3H^2}{A\log^2 X}.
\end{equation}
If $H<2A$, then $H\leq2A_I$ and the length conclusion follows. Otherwise, $|S_a|\leq2\delta^{-1}H/A$.  Comparing this upper bound with \eqref{eq:typeI-row-second-moment}, we obtain, for a sufficiently large constant $C_1$, a set $\mathcal A\subset(A,2A]$ such that
\begin{equation*}
 |\mathcal A|\geq Z^{-C_1}A,\qquad
 |S_a|\geq Z^{-C_1}H/A\quad\text{for all }a\in\mathcal A.
\end{equation*}

\emph{Removing the bounded variation factor.}
Let $q_\beta\leq\delta^{-1}$ satisfy \eqref{eq:type-I-TV}. For each fixed $a\in\mathcal A$, consider the progression
\begin{equation*}
 P_a\coloneqq\{b\in\N\colon ab\in P_0\}.
\end{equation*}
By \cite[Lemma 2.2(iii)]{MSTT} and monotonicity of total variation, we have
\begin{equation*}
 \begin{aligned}
 |S_a|
 &\leq\|\beta\|_{\TV(P_a;q_\beta)}
       \left|\sum_{b\in P_a}F(g(ab)\Gamma)\right|^*\\
 &\leq\delta^{-1}
       \left|\sum_{b\in P_a}F(g(ab)\Gamma)\right|^*.
 \end{aligned}
\end{equation*}
We choose a subprogression $B_a\subseteq P_a$ attaining the maximal sum, so that
\begin{equation}
 \label{eq:typeI-unweighted-prefix}
 \left|\sum_{b\in B_a}F(g(ab)\Gamma)\right|
 \geq\delta|S_a|\geq\delta Z^{-C_1}H/A.
\end{equation}
In particular, $|B_a|\geq\delta Z^{-C_1}H/A$, since  $F$ is $1$-bounded. We now conclude that
\begin{equation}
 \label{eq:typeI-dilate-correlation}
 |B_a|\geq Z^{-C_2'}H/A,\qquad
 \left|\sum_{b\in B_a}F(g(ab)\Gamma)\right|\geq Z^{-C_2'}|B_a|.
\end{equation}
If $|B_a|=1$ for some $a$, this already implies $H\leq Z^{C_2'}A_I$. We may therefore write
\begin{equation*}
 B_a=\{r_a+q_an\colon 1\leq n\leq|B_a|\},\qquad |B_a|\geq2,
\end{equation*}
where $q_a$ is the common difference of $B_a$. From the containment in $(X/a,(X+H)/a]$ and \eqref{eq:typeI-dilate-correlation} we have $q_a\leq2Z^{C_2'}$.

\emph{Local factorisations and concatenation.}
For each $a\in\mathcal A$, we apply  Theorem \ref{thm:Leng-equidistribution} to $n\mapsto g(a(r_a+q_an))$ on $[|B_a|]$ with $\delta_1=Z^{-C_2'}/2$. The hypothesis on $|\xi|$ holds since $\delta_1\leq\delta$, and \eqref{eq:typeI-dilate-correlation} gives the required correlation~\eqref{eq:Lengdelta}. We choose $Q=Z^{D^{C_3}}$ with $C_3=C_3(k)$ sufficiently large for the theorem's output and for
\begin{equation}
 \label{eq:typeI-local-factorisation-parameters}
 Q\geq M,\qquad
 Q\geq2Z^{C_2'},\qquad
 Q\geq Z^{C_1}.
\end{equation}
If conclusion (i) of the theorem holds for any $a\in\mathcal A$, then
\[
 H/A\leq Z^{C_2'}|B_a|\leq Z^{C_2'}Q,
\]
so $H\leq Z^{C_2'}QA_I$ and the Type I range follows after increasing $C_0$.

Otherwise, for every $a\in\mathcal A$, the same theorem provides  a $Q$-factorisation on $[|B_a|]$ through a subgroup $G_a$, with $G_{a}$ having step at most $s-1$ by the assumption that $G^{(s)}$ is one-dimensional. We write this conclusion as
\begin{equation}
 \label{eq:typeI-local-factorisations}
 g(a(r_a+q_an))
 =\varepsilon_a(n)g_a(n)\gamma_a(n).
\end{equation}
 Moreover,
\eqref{eq:typeI-dilate-correlation}, \eqref{eq:typeI-local-factorisation-parameters} and the bound for $q_a$ above give
\[
 |\mathcal A|\geq A/Q,\qquad
 |B_a|\geq H/(AQ),\qquad 1\leq q_a\leq Q.
\]
Thus the hypotheses of Lemma \ref{lem:concatenate-dilated-factorisations} hold with $d=k-1$. That lemma gives either
\[
 H\leq Q^{D^{C_{k-1}}}A\leq\mathcal RA_I
\]
or the desired global lower step factorisation with all bounds at most $\mathcal R$, after taking $C_0$ sufficiently large in \eqref{eq:typeI-single-parameter}.
\end{proof}

\subsection{The Type \texorpdfstring{$I_2$}{I2} case}

\begin{lemma}[Concatenation along quadratic paths]
\label{lem:concatenate-quadratic-path-factorisations} Fix integers $d\geq1$ and $1\leq s\leq d$. There is $C=C(d,s)\geq1$ with the following property.

Let $J=(x,x+L]\cap\Z$, where $L\geq2$, and let $G/\Gamma$ be an $s$-step filtered nilmanifold of degree at most $d$, dimension at most $D$, and complexity at most $W$, where $D\geq2$ and $W\geq16$. Let $h\colon\Z\to G$ be polynomial. Suppose that
\[
 K>W^{D^C},\qquad B>W^{D^C},\qquad Q>W^{D^C},
 \qquad
 L\leq16K^2BQ.
\]

Suppose that at least $BQ/W^2$ pairs of integers $(b,q)$ with $0\leq b\leq WB$ and $1\leq q\leq Q$ have the following property. For each such pair, choose integers $m_{b,q},n_{b,q}$ and set
$$z_{b,q}(t)=(m_{b,q}+tq)(n_{b,q}-tb).$$ Suppose that there is an arithmetic progression $P_{b,q}\subseteq[K]$ of at least $K/W$ terms such that
\[
z_{b,q}(t)\in J\qquad \textnormal{for all}\,\, t\in P_{b,q},
\]
and $h\circ z_{b,q}$, after affine parametrisation of $P_{b,q}$, has a $W$-factorisation through a subgroup of step at most $s-1$, with respect to the doubled filtration
\[
 G^{[2]}_r=G_{\lceil r/2\rceil}.
\]

Then $h$ has a $W^{D^C}$-factorisation on $J$ through a subgroup $G_*\leq G$ of step at most $s-1$, written
\[
 h=\varepsilon h_*\gamma\qquad\text{on }J.
\]
Moreover, $G_*/(G_*\cap\Gamma)$ has complexity at most $W^{D^C}$ in an induced rational basis, and $\gamma(\cdot)\Gamma$ has period at most $W^{D^C}$.
\end{lemma}

\begin{proof}
If $G$ has step at most $s-1$, take $G_*=G$, $h_*=h$, and $\varepsilon=\gamma=1_G$. Thus we may assume that $G$ has step $s$.

For each pair $(b,q)$ in the hypothesis, write $P_{b,q}=t_{b,q}+a_{b,q}[N_{b,q}]$ with $a_{b,q}\geq1$. Since $P_{b,q}\subseteq[K]$ and $N_{b,q}\geq K/W$, we have
\[
 N_{b,q}\geq K/W\geq2d+1,
 \qquad a_{b,q}(N_{b,q}-1)\leq K-1,
 \qquad a_{b,q}\leq2W.
\]
We abbreviate $t=t_{b,q}$, $a=a_{b,q}$, and $N=N_{b,q}$. The local factorisation gives, for $n\in[N]$,
\[
 h(z_{b,q}(t+an))
 =\varepsilon_{b,q}(n)h_{b,q}(n)\gamma_{b,q}(n),
\]
where $h_{b,q}$ takes values in a connected $W$-rational subgroup $H_{b,q}$ of step at most $s-1$. By \cite[Lemma A.9]{Leng}, there are linearly independent horizontal characters $\eta_1,\ldots,\eta_r$, $1\leq r\leq D$, which annihilate $H_{b,q}$, have heights at most $W^{D^c}$ for a constant $c=c(d,s)$, and whose common kernel has step at most $s-1$.

For $1\leq n\leq2d+1$, choose $1\leq r_n\leq W$ such that
$\gamma_{b,q}(n)^{r_n}\in\Gamma$, and set $r_0=\prod_{n=1}^{2d+1}r_n\leq W^{2d+1}$. Then the polynomial
$r_0\eta_i(\gamma_{b,q}(n))$ is an integer at these $2d+1$ points, for every $i$. Since this polynomial has degree at most $2d$ and is integer-valued at $2d+1$ consecutive integers, all of its finite differences at the first of these integers are integers. Therefore, the binomial expansion of this polynomial shows that it is integer-valued on $\mathbb Z$.

Moreover,
$u_i(n)=\eta_i(\varepsilon_{b,q}(n))$ is a real polynomial of degree
at most $2d$, and smoothness of $\varepsilon_{b,q}$ and the $D|\eta_i|$-Lipschitz property of $\eta_i$ give
\[
 |u_i(n)-u_i(n-1)|\leq D|\eta_i|W/N
 \qquad \textnormal{for}\,\,1\leq n\leq N.
\]
As in the proof of Lemma~\ref{lem:concatenate-dilated-factorisations}, interpolating this difference polynomial at spaced points of $[N]$ yields
\[
 \|r_0u_i\|_{C^\infty([N])}\ll_d r_0D|\eta_i|W.
\]
As $\eta_i(H_{b,q})=0$, the characters $r_0\eta_i$ satisfy
\begin{equation}
 \label{eq:I2-new-local-character}
 |r_0\eta_i|\leq W^{D^{c_0}},\qquad
 \|r_0\eta_i\circ h(z_{b,q}(t+a\cdot))\|_{C^\infty([N])}
 \leq W^{D^{c_0}}
\end{equation}
for some $c_0=c_0(d,s)$.

There are at most
$(2W^{D^{c_0}}+1)^{D^2}$ choices of these lists of multiplied characters, and there are at most $2W$ possible coefficients $a_{b,q}$. Pigeonholing these lists and coefficients gives a fixed $a\leq2W$, a nonempty independent list $\mathcal L$, and a set
\[
 \mathcal S\subseteq([0,WB]\cap\Z)\times([1,Q]\cap\Z)
\]
such that
\begin{equation}
 \label{eq:I2-new-common-data}
 \begin{gathered}
 |\mathcal S|\geq BQ/R,\\
 |\eta|+
 \|\eta\circ h(z_{b,q}(t_{b,q}+a\cdot))
   \|_{C^\infty([N_{b,q}])}\leq R
 \end{gathered}
\end{equation}
for every $\eta\in\mathcal L$ and $(b,q)\in\mathcal S$,
where $R=W^{D^{c_1}}$ for a constant $c_1=c_1(d,s)$. Moreover, the common kernel of $\mathcal L$ has step at most
$s-1$.

Write $J=y_0+[N_0]$, where $y_0=\min J-1$ and $N_0=|J|\asymp L$.
Fix $\eta\in\mathcal L$ and write
\[
 P_\eta(z)=\eta(h(z))
 =\sum_{\ell=0}^d\alpha_\ell(z-y_0)^\ell.
\]
We recover the nonconstant coefficients of this polynomial by descending through
$j=d,\ldots,1$. At the beginning of the step for $j$, let $r\geq1$
and $T\geq R$ satisfy
\begin{equation}
 \label{eq:I2-new-induction}
 r\leq T,\qquad
 \|r\alpha_\ell\|_{\R/\Z}\leq TL^{-\ell}
 \quad \textnormal{for all}\,\, j<\ell\leq d.
\end{equation}
For $j=d$ take $r=1$ and $T=R$. Set $F_j(z)=r\sum_{\ell\leq j}\alpha_\ell(z-y_0)^\ell$. For each \(\ell>j\), choose an integer \(m_\ell\) such that
\[
 \left|r\alpha_\ell-m_\ell\right|
 \leq T L^{-\ell}.
\]
Since \(z_{b,q}(t_{b,q}+an)-y_0\in\Z\) for every \(n\in\Z\), the polynomial
\[
 \sum_{\ell=j+1}^d
 m_\ell\bigl(z_{b,q}(t_{b,q}+an)-y_0\bigr)^\ell
\]
is integer-valued and therefore does not affect the $C^\infty$-norm. The remaining error is
\begin{align*}
 E(n)=\sum_{\ell=j+1}^d(r\alpha_\ell-m_\ell)\bigl(z_{b,q}(t_{b,q}+an)-y_0\bigr)^\ell.
\end{align*}
For $n\in[N_{b,q}]$, we have $z_{b,q}(t_{b,q}+an)\in J$, so $|E(n)|\ll_d T$ by \eqref{eq:I2-new-induction}. The polynomial $E$ has degree at most $2d$. Since $N_{b,q}\geq2d+1$, interpolation yields
\begin{align*}
 \max_{0\leq v\leq2d}N_{b,q}^v\max_{n\in[N_{b,q}]}|\partial^vE(n)|\ll_d T.
\end{align*}
We combine this estimate with \eqref{eq:I2-new-common-data} and $r\leq T$ to obtain
\begin{align*}
 \|F_j(z_{b,q}(t_{b,q}+a\cdot))\|_{C^\infty([N_{b,q}])}\ll_d TR.
\end{align*}

We have
\[
 z_{b,q}(t_{b,q}+an)
 =c_0+c_1n-bqa^2n^2
\]
for suitable real numbers \(c_0,c_1\). Since the coefficient of $z^j$ in $F_j(z)$ is $r\alpha_j$,
the coefficient of \(n^{2j}\) in
\(F_j(z_{b,q}(t_{b,q}+an))\) is
\[
 r\alpha_j(-bqa^2)^j.
\]
Therefore its \(2j\)-th finite difference is the constant
\[
 (2j)!r\alpha_j(-bqa^2)^j
 \qquad\pmod{1}.
\]
By the definition of the \(C^\infty([N_{b,q}])\)-norm,
\[
 N_{b,q}^{2j}
 \bigl\|(2j)!a^{2j}r\alpha_jb^jq^j\bigr\|_{\R/\Z}
 \ll_d TR.
\]
Since \(N_{b,q}\geq K/W\), we conclude that
\begin{equation}
 \label{eq:I2-new-path-coefficient}
 \bigl\|(2j)!a^{2j}r\alpha_jb^jq^j\bigr\|_{\R/\Z}
 \ll_d TR(W/K)^{2j}
 \leq (WRT)^{O_d(1)}K^{-2j}.
\end{equation}

Set $B_0=\lfloor WB\rfloor+1$ and $Q_0=\lfloor Q\rfloor$. The pairs $(b+1,q)$ with $(b,q)\in\mathcal S$ have density at least $\delta_0=(2WR)^{-1}$ in $[B_0]\times[Q_0]$. We apply the two-variable Vinogradov lemma \cite[Lemma A.11]{Leng} to
\begin{align*}
 (u,v)\mapsto (2j)!a^{2j}r\alpha_j(u-1)^jv^j,
\end{align*}
using \eqref{eq:I2-new-path-coefficient}. By \eqref{eq:I2-new-path-coefficient}, for each such pair,
\begin{align*}
 \left\|(2j)!a^{2j}r\alpha_j b^jq^j\right\|_{\R/\Z}
 \leq (WRT)^{O_d(1)}K^{-2j}.
\end{align*}
After increasing $C=C(d,s)$, the right-hand side is less than $\delta_0/10$. We may therefore apply
\cite[Lemma~A.11]{Leng} to
\begin{align*}
 (u,v)\mapsto(2j)!a^{2j}r\alpha_j(u-1)^jv^j.
\end{align*} The lemma gives an integer $w\geq1$ with $w\leq(WRT)^{O_d(1)}$, and the coefficient of $\binom uj\binom vj$ in the resulting polynomial is
\begin{align*}
 w(2j)!(j!)^2a^{2j}r\alpha_j.
\end{align*}
The coefficient bound in the lemma therefore gives
\begin{align*}
 \|w(2j)!(j!)^2a^{2j}r\alpha_j\|_{\R/\Z}
 \leq(WRT)^{O_d(1)}K^{-2j}B^{-j}Q^{-j}.
\end{align*}
With $r'=w(2j)!(j!)^2a^{2j}r$, we obtain
\begin{equation}
 \label{eq:I2-new-q-step}
r'/r\leq(WRT)^{O_d(1)},\qquad\|r'\alpha_j\|_{\R/\Z}\leq(WRT)^{O_d(1)}(K^2BQ)^{-j}\leq(WRT)^{O_d(1)}L^{-j},
\end{equation}
where the last inequality uses $L\leq16K^2BQ$. Choosing $T'=(WRT)^{c_d}$ with $c_d$ sufficiently large, we may
replace $(r,T)$ by $(r',T')$ in \eqref{eq:I2-new-induction} at degree $j-1$. We use the original set $\mathcal S$ at each step. After at most $d$ steps, we obtain an integer $r_\eta\geq1$ and a bound
$T_0\leq W^{D^{c_2}}$, where $c_2=c_2(d,s)$, such that
\begin{align*}
 r_\eta\leq T_0,\qquad
 \|r_\eta\alpha_\ell\|_{\R/\Z}
 \leq T_0L^{-\ell}\quad \textnormal{for all}\,\,1\leq\ell\leq d,
\end{align*}
with $T_0$ independent of $\eta\in\mathcal L$. The assumed lower bounds for $K,B,Q$, with $C=C(d,s)$ sufficiently large, justify the application of the two-variable Vinogradov lemma at every step.

We then replace each $\eta\in\mathcal L$ by $r_\eta\eta$. As $N_0\asymp L$, the coefficient estimates and integer approximations give
\begin{align*}
 |\eta|+\|\eta\circ h\|_{C^\infty(J)}
 \leq R_0,\qquad R_0=W^{D^{c_3}},
\end{align*}
for a constant $c_3=c_3(d,s)$. The list remains independent after multiplication and
its common kernel is unchanged. Denoting
$G_*=\bigcap_{\eta\in\mathcal L}\ker\eta$, we apply Lemma \ref{lem:character-factorisation} to $h$ on $J$, with complexity bound $W$ and character and smoothness bound $R_0$. The step bound for $G_*$ follows from the construction of $\mathcal L$ and the fact that multiplying its characters by nonzero integers does not change their common kernel.
We obtain $h=\varepsilon h_*\gamma$ through $G_*$, with all rationality, complexity and smoothness parameters and period at most $(WR_0)^{D^{O_d(1)}}$.
Increasing $C(d,s)$ absorbs all these bounds
into $W^{D^C}$.
\end{proof}

We are now ready to prove the efficient Type $I_2$ result. 

\begin{proposition}[Efficient Type $I_2$ sums in short intervals]
\label{prop:typeI2-standalone}
Let $k\geq2$ be an integer and let $1\leq s\leq k-1$. There is $C_*=C_*(k)\geq1$ with the following property.  Let
$X,H$ satisfy $X\geq2$ and $2\leq H\leq X$, and let $0<\delta\leq1/10$.  Let $D\geq2$ be an
integer, let $M\geq2$, and let $G/\Gamma$ be an $s$-step filtered nilmanifold of degree at most $k-1$, dimension
at most $D$, and complexity at most $M$, with $G^{(s)}$ one-dimensional.  Let $g\colon \Z\to G$ be polynomial, and let $F\colon G/\Gamma\to\C$ be a $1$-bounded $M$-Lipschitz vertical
nilcharacter with nonzero frequency $\xi$ satisfying $|\xi|\leq M/\delta$.  Suppose that $f$ is a
$(\delta,A_{I_2})$ Type $I_2$ sum with $A_{I_2}\geq1$, and that
\begin{equation}
 \label{eq:type-I-I2-correlation}
 \left|\sum_{X<n\leq X+H}f(n)F(g(n)\Gamma)\right|^*
 \geq\delta H.
\end{equation}
Define $\mathcal R$ by
\begin{equation}
 \label{eq:type-I-I2-envelope}
 \mathcal R=\left(\frac{M\log X}{\delta}\right)^{D^{C_*}}.
\end{equation}
Then at least one of the following cases holds.
\begin{itemize}[leftmargin=2.3em]
 \item \textit{Type $I_2$ range.} We have
 \begin{equation}
  \label{eq:typeI2-scale}
  H\leq\mathcal R X^{1/3}A_{I_2}^{2/3}.
 \end{equation}

 \item \textit{Step descent.} The sequence $g$ has an $\mathcal R$-factorisation on $(X,X+H]$ through a subgroup $G_*\leq G$ of step at most $s-1$, written
 \begin{equation}
  \label{eq:type-I-I2-factorization}
  g=\varepsilon g_*\gamma.
 \end{equation}
 The nilmanifold $G_*/(G_*\cap\Gamma)$, with induced filtration $G_{*,i}=G_*\cap G_i$, has degree at most $k-1$, dimension at most $D$, step at most $s-1$, and complexity at most $\mathcal R$ in an induced rational basis. The sequence $\gamma(\cdot)\Gamma$ has period at most $\mathcal R$.
\end{itemize}
\end{proposition}

\begin{proof} We may assume that $X$ is sufficiently large in terms of $k$, since otherwise the claim is immediate by taking $C_{*}$ large.

We write $f=\alpha*\beta_1*\beta_2$ and set
$Z=M(\log X)/\delta$. We assume that the Type $I_2$ range case fails, so
\begin{equation}
 \label{eq:I2-range-excluded}
 H>\mathcal R X^{1/3}A_{I_2}^{2/3},
\end{equation}
and prove the step descent conclusion. Since $H\leq X$, we have
\begin{equation}
 \label{eq:I2-outer-support-range}
 A_{I_2}<X\mathcal R^{-3/2}<X.
\end{equation}
All constants in this proof depend only on $k$.

\emph{Reduction to two-variable sums.}
We choose an arithmetic progression $P_0\subset(X,X+H]$ attaining \eqref{eq:type-I-I2-correlation}. By the pigeonhole principle, there exist dyadic numbers $M_0\geq1/2$ and $N_0\geq1/2$ such that
\[
\sum_{a\leq A_{I_2}}|\alpha(a)||S_a|\gg \frac{\delta H}{\log^2 X},
\]
where 
\begin{equation*}
 S_a=\sum_{M_0<m\leq2M_0}
 \sum_{\substack{N_0<n\leq2N_0\\X/a<mn\leq(X+H)/a}}
 \beta_1(m)\beta_2(n)\1_{P_0}(amn)F(g(amn)\Gamma).
\end{equation*}
Pigeonholing again in $a$ and applying Cauchy--Schwarz and \eqref{eq:type-alpha-second-moment} give
\begin{equation}
 \label{eq:I2-outer-second-moment}
\sum_{\substack{A<a\leq2A\\a\leq A_{I_2}}}|S_a|^2\gg\frac{\delta^3H^2}{A\log^6 X},
\end{equation}
for some dyadic $1/2\leq A\leq A_{I_2}$. We necessarily have $AM_0N_0\asymp X$.  Interchanging the last two variables if necessary, we may assume that
$M_0\leq N_0$.  Then
\begin{equation}
 \label{eq:I2-short-dyadic-variable}
 M_0\leq(2X/A)^{1/2}.
\end{equation}

We replace $\beta_1$ by $\beta_1\1_{(M_0,2M_0]}$ and $\beta_2$ by $\beta_2\1_{(N_0,2N_0]}$, and retain the notation $\beta_1,\beta_2$ for these restricted sequences.  If $q_0\leq\delta^{-1}$ is the common modulus in \eqref{eq:type-I2-TV}, these restrictions satisfy
\begin{equation}
 \label{eq:I2-restricted-TV}
 \|\beta_i\|_{\TV(\N;q_0)}\leq3\delta^{-1}
 \qquad\text{for }i=1,2,
\end{equation}
since restriction to an interval adds at most two endpoint jumps on each residue class:
\[
 \begin{aligned}
 \|\beta_i\1_{(V,2V]}\|_{\TV(\N;q_0)}
 &\leq\|\beta_i\|_{\TV(\N;q_0)}
+2\sum_{c\bmod q_0}
\sup_{n\equiv c\ (q_0)}|\beta_i(n)|\\
 &\leq3\|\beta_i\|_{\TV(\N;q_0)}
 \leq3\delta^{-1}
 \qquad\text{for }V\in\{M_0,N_0\}.
 \end{aligned}
\]

Since \eqref{eq:I2-range-excluded} holds and $a\leq A_{I_2}$, we have $H/a>\mathcal R(X/a)^{1/3}\geq4(X/a)^{1/3}$. Since $|\beta_i|\leq \delta^{-1}$, Shiu's bound
\cite[Lemma 2.17]{MSTT} gives
\begin{equation}
 \label{eq:I2-full-row-upper}
 |S_a|\leq\delta^{-2}
 \sum_{X/a<v\leq(X+H)/a}d_2(v)
 \ll\delta^{-2}\frac Ha\log(X/a).
\end{equation}
Comparing \eqref{eq:I2-outer-second-moment} and
\eqref{eq:I2-full-row-upper}, we can choose a constant $C_1$ sufficiently large such that there is a set $\mathcal A\subseteq(A,2A]\cap[1,A_{I_2}]\cap\N$ of at least $Z^{-C_1}A$ values of $a$, each satisfying
\begin{equation}
 \label{eq:I2-many-outer-dilates}
 |S_a|\geq Z^{-C_1}H/A.
\end{equation}
We choose $C_2'$ sufficiently large in terms of $C_1$ and set $\rho_0=Z^{-C_2'}$.  This choice ensures
\begin{equation}
 \label{eq:I2-reduced-coefficients}
 \begin{gathered}0<\rho_0<\min(1/10,1/\log(X/a)),\qquad
 q_0\leq\rho_0^{-1},\\
 \|\beta_i\|_{\TV(\N;q_0)}\leq\rho_0^{-1},
 \qquad |S_a|\geq\rho_0H/a,\end{gathered}
\end{equation}
and $\rho_0\leq\delta$, so $|\xi|\leq M/\rho_0$.

For each $a\in\mathcal A$, write
\begin{equation}
 \label{eq:I2-reduced-data}
 Y=X/a,\qquad L=H/a,\qquad h(n)=g(an),\qquad
 P=\{v\in(Y,Y+L]\colon av\in P_0\}.
\end{equation}
Since $a\leq A_{I_2}$, $a\in(A,2A]$, and $H\leq X$, \eqref{eq:I2-outer-support-range} and \eqref{eq:I2-short-dyadic-variable} give
\begin{equation}
 \label{eq:I2-reduced-ranges}
 Y>\mathcal R^{3/2}\geq3,\qquad
 L\leq Y,\qquad M_0\leq2Y^{1/2}.
\end{equation}
With $J=(M_0,2M_0]\cap\Z$,
\eqref{eq:I2-many-outer-dilates} and \eqref{eq:I2-reduced-coefficients} yield
\begin{equation}
 \label{eq:I2-preliminary-correlation}
 \left|\sum_{m\in J}\sum_{Y<mn\leq Y+L}
 \beta_1(m)\beta_2(n)\1_P(mn)F(h(mn)\Gamma)\right|
 \geq\rho_0L.
\end{equation}

Let $C=C(k)$ be sufficiently large, and set
\begin{equation}
 \label{eq:I2-threshold-definition}
 Z_0=M/\rho_0,\qquad
 T=Z_0^{D^C}\geq\max(32,M,2Z^{C_1}).
\end{equation}
Since $Z_0\leq Z^{C_2'+1}$, we can choose $C_*$ large enough so that
\begin{equation}
 \label{eq:I2-threshold-size}
 T^4\leq\mathcal R.
\end{equation}
The failure of the Type $I_2$ range case gives the stronger lower bound
\begin{equation}
 \label{eq:I2-long-reduced-interval}
 L>\mathcal R Y^{1/3}(A_{I_2}/a)^{2/3}
 \geq T^4Y^{1/3}.
\end{equation}
Thus $L>Y^{1/3}$ and $L\geq2$.

If $M_0\leq16T(LY)^{1/4}$, set $f_0=(\beta_1\1_J)*\beta_2$. The total variation hypotheses give
$|\beta_1(m)|\leq\rho_0^{-1}$, so
\begin{equation*}
 \sum_{m\leq U}|\beta_1(m)\1_J(m)|^2\leq\rho_0^{-2}U
 \quad\text{for every }U\geq1.
\end{equation*}
Thus $f_0$ is a $(\rho_0^2,2M_0)$ Type I sum in the sense of
Definition \ref{def:type-sums}(i). We apply Proposition \ref{prop:typeI-standalone} with $X=Y$, $H=L$, $\delta=\rho_0^2$, and $A_I=2M_0$.
The correlation in \eqref{eq:I2-preliminary-correlation} is at least
$\rho_0^2L$, and $|\xi|\leq M/\rho_0^2$. Let $\mathcal R_0$ be the parameter in \eqref{eq:typeI-single-parameter} for these data. Since $\log Y\leq2\rho_0^{-1}$, choosing $C$ sufficiently large gives
\begin{equation*}
 \mathcal R_0
 =\left(\frac{M\log Y}{\rho_0^2}\right)^{D^{C_0(k)}}
 \leq T.
\end{equation*}
The Type I range case of \ref{prop:typeI-standalone} gives
\begin{equation*}
 L\leq2\mathcal R_0M_0\leq2TM_0\leq32T^2(LY)^{1/4}.
\end{equation*}
Dividing by $L^{1/4}$ and raising to the power $4/3$ gives
\begin{equation*}
 L\leq32^{4/3}T^{8/3}Y^{1/3}\leq T^4Y^{1/3},
\end{equation*}
since $T\geq32$. Equation \eqref{eq:I2-long-reduced-interval} excludes this length
case, so the step descent case gives the factorisation
\begin{equation}
 \label{eq:I2-short-factor-factorisation}
 h=\varepsilon_a h_a\gamma_a\quad\text{on }(Y,Y+L]
\end{equation}
through a subgroup of step at most $s-1$, with induced complexity and rational-coset period at most $T$.

\emph{Hyperbola decomposition and path correlations.}
From now on, let $M_0>16T(LY)^{1/4}$, and fix $a\in\mathcal A$. Set
\begin{equation}
 \label{eq:I2-hyperbola-scales}
 Q=\left\lfloor\frac{M_0}{(LY)^{1/4}}\right\rfloor,
 \qquad K=\frac{LQ}{M_0},\qquad B=\frac{YQ}{M_0^2}.
\end{equation}
Since $M_0/(LY)^{1/4}>16T\geq512$, we have
\begin{equation}
 \label{eq:I2-hyperbola-parameter-ranges}
 \begin{gathered}
 \frac{M_0}{2(LY)^{1/4}}\leq Q\leq\frac{M_0}{(LY)^{1/4}},\\
 Q>T,\qquad K>T,\qquad B>T,\\
 \frac{L}{16}\leq K^2BQ=\frac{L^2YQ^4}{M_0^4}\leq L,
 \end{gathered}
\end{equation}
where the lower bound for $K$ follows from \eqref{eq:I2-long-reduced-interval}, which gives $K\geq\tfrac12L^{3/4}Y^{-1/4}>\tfrac12T^3>T$. The assumption $M_0>16T(LY)^{1/4}$ and \eqref{eq:I2-reduced-ranges} give
\[
 16T(LY)^{1/4}<M_0\leq2Y^{1/2},
 \qquad
 B\geq\tfrac14(Y/L)^{1/4}>2T.
\]
The bounds for $K^2BQ$ follow from the two bounds for $Q$.
In particular, $K>1$ and $Q>M_0/L$.

A hyperbola path is a finite progression of pairs $(m+tq,n-tb)$ with fixed integers $m,n,b,q$; along it the
product $(m+tq)(n-tb)$ is quadratic in $t$. We obtain paths inside $Y<mn\leq Y+L$ from the following input.

We apply \cite[Corollary 8.2]{MSTT} with its parameters $(X,H,M,\delta)=(Y,L,M_0,\rho_0)$, progression $P$, and function $(n,m)\mapsto F(h(mn)\Gamma)$. The required parameter range is
verified by \eqref{eq:I2-reduced-ranges}, \eqref{eq:I2-long-reduced-interval}, and \eqref{eq:I2-hyperbola-parameter-ranges}: in particular,
\begin{equation*}
 Y^{1/3}<L\leq Y,\qquad M_0/L<Q\leq M_0/(LY)^{1/4},
 \qquad J\subset[M_0,2M_0].
\end{equation*}
The modulus and total variation bounds, including $\rho_0<1/\log Y$, are \eqref{eq:I2-reduced-coefficients}, and its correlation hypothesis is \eqref{eq:I2-preliminary-correlation}. The number of pairs in that corollary is expressed here using
\begin{equation*}
 YQ^2/M_0^2=BQ.
\end{equation*}
It gives at least $\rho_0^{C_3}BQ$ coprime pairs $(b,q)$ satisfying
\begin{equation}
 \label{eq:I2-corollary-pairs}
 \rho_0^{C_3}Q\leq q\leq Q,\qquad b\asymp Yq/M_0^2,
\end{equation}
where $C_3$ is a sufficiently large absolute constant. For every such pair there are integers $m_{b,q},n_{b,q}$ and $K_{b,q}\geq1$ such that the progression
\begin{equation}
 \label{eq:I2-selected-progression}
 \mathfrak p_{b,q}=\{(m_{b,q}+tq,n_{b,q}-tb)\colon 1\leq t\leq K_{b,q}\}
 \subset\{(m,n)\in\Z^2\colon m\in J,\ Y<mn\leq Y+L\}
\end{equation}
has $K_{b,q}\leq K$. Write
\begin{equation*}
 I_{b,q}=[K_{b,q}],\qquad
 z_{b,q}(t)=(m_{b,q}+tq)(n_{b,q}-tb).
\end{equation*}
The correlation lower bound provided by the hyperbola decomposition result of \cite[Corollary 8.2]{MSTT} is
\begin{equation}
 \label{eq:I2-path-correlation}
 \left|\sum_{t\in I_{b,q}}F(h(z_{b,q}(t))\Gamma)\right|^*
 \geq\rho_0^{C_3}K.
\end{equation}

Let $\mathcal S$ be the set of these pairs and set $\delta_1=\rho_0^{C_3}$. By \eqref{eq:I2-corollary-pairs}, there is an absolute constant $C_{\mathrm h}\geq1$ such that
\begin{equation}
 \label{eq:I2-populated-pairs}
 |\mathcal S|\geq\delta_1BQ,\qquad
 0\leq b\leq C_{\mathrm h}B,\qquad
 \delta_1Q\leq q\leq Q
 \quad\text{for every }(b,q)\in\mathcal S.
\end{equation}

\emph{Factorisations on quadratic path progressions.}
For each retained pair, choose a progression $t_{b,q}+d_{b,q}[N_{b,q}]\subset I_{b,q}$ attaining the maximal sum in \eqref{eq:I2-path-correlation}. Since $F$ is
$1$-bounded,
\begin{equation}
 \label{eq:I2-witness-geometry}
 \delta_1K\leq N_{b,q}\leq K_{b,q}\leq K.
\end{equation}
Write $d=k-1$. We choose $C$ in \eqref{eq:I2-threshold-definition} large enough that $\delta_1T\geq2d+2$. Then
$N_{b,q}\geq2d+2$ by \eqref{eq:I2-hyperbola-parameter-ranges}.
We can expand 
\begin{equation*}
 z_{b,q}(t)=m_{b,q}n_{b,q}
 +(qn_{b,q}-bm_{b,q})t-bqt^2.
\end{equation*}
Write $h(z)=\prod_{i=0}^dh_i^{\binom zi}$ with $h_i\in G_i$, using the Taylor representation in \cite[Definition 2.5]{Leng}. Each $\binom{z_{b,q}(t)}i$ is integer-valued of degree at most $2i$. Thus $t\mapsto h(z_{b,q}(t))$, and its affine restriction $n\mapsto h(z_{b,q}(t_{b,q}+d_{b,q}n))$, are polynomial for the doubled filtration $G^{[2]}_r=G_{\lceil r/2\rceil}$, of degree at most $2d$. This filtration has the same nilpotency step and Mal'cev basis.

We apply Theorem \ref{thm:Leng-equidistribution} with degree $2d$, one variable, and correlation parameter $\delta_1/2$ to $n\mapsto h(z_{b,q}(t_{b,q}+d_{b,q}n))$ on $[N_{b,q}]$.
The frequency condition follows from $|\xi|\leq M/\rho_0$ and $\delta_1/2\leq\rho_0$. The short case of Theorem \ref{thm:Leng-equidistribution} would imply
\begin{equation*}
 K\leq\delta_1^{-1}N_{b,q}
 \leq\delta_1^{-1}Z_0^{D^{C_5}}\leq T,
\end{equation*}
contrary to \eqref{eq:I2-hyperbola-parameter-ranges}, after increasing $C$. The factorisation conclusion in the other case of Theorem \ref{thm:Leng-equidistribution} gives a
$Z_0^{D^{C_5}}$-factorisation of
$n\mapsto h(z_{b,q}(t_{b,q}+d_{b,q}n))$ on $[N_{b,q}]$ through a connected subgroup of step at most $s-1$. The induced nilmanifold complexity and the period of the rational coset sequence are at most $Z_0^{D^{C_5}}$ as well,
after increasing $C_5$.
Choose a constant $C_6$ large enough that, with
\begin{equation}
 \label{eq:I2-local-path-parameter}
 W=Z_0^{D^{C_6}},
\end{equation}
the subgroup rationality and induced complexity, smoothness, pointwise rationality, and period are all at most $W$,
and $W\geq\max(16,M,C_{\mathrm h},4\delta_1^{-1})$.

We now apply Lemma
\ref{lem:concatenate-quadratic-path-factorisations}.
The pair count and coordinate hypotheses of that lemma follow from \eqref{eq:I2-populated-pairs}, since
\[
 |\mathcal S|\geq\delta_1BQ\geq BQ/W^2,
 \qquad b\leq C_{\mathrm h}B\leq WB
 \quad\text{for every }(b,q)\in\mathcal S.
\]
We take the interval in that lemma to be $(Y,Y+L]\cap\Z$. The path and subprogression hypotheses of that lemma follow from \eqref{eq:I2-selected-progression} and \eqref{eq:I2-witness-geometry}. The scale conditions are \eqref{eq:I2-hyperbola-parameter-ranges}. If $\widetilde{C}$ denotes the fixed exponent in Lemma~\ref{lem:concatenate-quadratic-path-factorisations}, we choose the exponent $C$ in
\eqref{eq:I2-threshold-definition} sufficiently large that
$T\geq W^{D^{\widetilde{C}}}$.
Equation \eqref{eq:I2-hyperbola-parameter-ranges} then gives
$K>W^{D^{\widetilde{C}}}$,
$B>W^{D^{\widetilde{C}}}$, and
$Q>W^{D^{\widetilde{C}}}$.
Thus the path lemma provides a connected lower-step subgroup $G_a\leq G$ and a factorisation
\begin{equation}
 \label{eq:I2-long-factorisation}
 h=\varepsilon_a h_a\gamma_a\quad\text{on }(Y,Y+L]
\end{equation}
with subgroup rationality, induced complexity, smoothness, rationality, and rational coset period all at most $T$.

\emph{Concatenation in the outer variable.}
The short case gives \eqref{eq:I2-short-factor-factorisation}, while the long case gives \eqref{eq:I2-long-factorisation}. Therefore, for every $a\in\mathcal A$ the sequence $h(n)=g(an)$ admits a factorisation
\begin{equation}
 \label{eq:I2-all-outer-factorisations}
 g(an)=\varepsilon_a(n)h_a(n)\gamma_a(n)
 \quad\text{for every }n\in(X/a,(X+H)/a],
\end{equation}
where $h_a$ takes values in a connected $T$-rational subgroup $G_a$ of step at most $s-1$, $\varepsilon_a$ is $(T,(X/a,(X+H)/a])$-smooth, and $\gamma_a$ is $T$-rational.

For each $a\in\mathcal A$, set
\begin{equation*}
 B_a=(X/a,(X+H)/a]\cap\Z=\lfloor X/a\rfloor+[N_a].
\end{equation*}
Since $a\leq2A$ and $H>4A$ by \eqref{eq:I2-range-excluded}, we have $N_a\geq2$ and
$N_a\geq H/(2A)-1\geq H/(AT)$. Also $|\mathcal A|\geq Z^{-C_1}A\geq A/T$ by
\eqref{eq:I2-threshold-definition}.
We apply Lemma \ref{lem:concatenate-dilated-factorisations} with $q_a=1$, $r_a=\lfloor X/a\rfloor$ and
$Q=T$, with its local factorisation input given by \eqref{eq:I2-all-outer-factorisations}, translated to $[N_a]$.
The lemma gives either
\begin{equation*}
 H\leq T^{D^{C_{\mathrm{dil}}}}A
 \leq\mathcal R X^{1/3}A_{I_2}^{2/3},
\end{equation*}
after increasing $C_*$, or a global lower-step factorisation \eqref{eq:type-I-I2-factorization}. The first alternative contradicts \eqref{eq:I2-range-excluded}, while the second has the subgroup, induced complexity, smoothness, rationality and period bounded by $\mathcal R$, again after increasing $C_*$.
This proves the proposition.
\end{proof}

\section{An efficient non-abelian Type II inverse theorem}
\label{sec:row-column-typeII}

In this section we prove the non-abelian Type II estimate with growing dimension needed for Theorem \ref{thm:corrected-main}.

For $X,H,A\in\N$ with $2\leq H\leq X$, $AH/X\geq4$ and $H/A\geq4$, we cover $(A,2A]$ by $O(X/H)$ subintervals 
\begin{equation}
 \label{eq:rowcol-cover-intervals1}
 \begin{aligned}
 I_{A'}&=\left(A',\left(1+\frac HX\right)A'\right],
 \end{aligned}
 \end{equation}
 with each point belonging to at most $O(1)$ such intervals. We also
define the corresponding intervals
\begin{equation} \label{eq:rowcol-cover-intervals2}
\begin{aligned}
 J_{A'}&=\left(\left(1+\frac HX\right)^{-1}\frac X{A'},\left(1+\frac HX\right)\frac X{A'}\right].
 \end{aligned}
\end{equation}
The intervals have cardinalities $|I_{A'}|\asymp AH/X$ and $|J_{A'}|\asymp H/A$, and
\begin{equation}
 \label{eq:rowcol-cover-containment}
 (X/a,(X+H)/a]\subset J_{A'}\subset(X/(2A),2X/A]
 \qquad\text{for every }a\in I_{A'}.
\end{equation}
Note that for a given $b$, the condition $b\in J_{A'}$ restricts $A'$ to an interval of length $O(AH/X)$, which contains only $O(1)$ left endpoints of the $I_{A'}$, so every point belongs to at most $O(1)$ intervals $J_{A'}$.

We first prove a bilinear concatenation result which will be a key step in the proof of the Type II estimate.

\begin{lemma}[Bilinear concatenation and step reduction]
\label{lem:rowcol-bilinear-concatenation}
Let $d,s$ be integers with $2\leq s\leq d$. There is $C=C(d,s)\geq1$ with the following property. Let $D,X,H,A\in\N$ and $M\geq2$, where $D\geq2$ and
\begin{equation*}
 M^{D^C}\max\{A,X/A\}<H\leq X.
\end{equation*}
Let $G/\Gamma$ be an $s$-step filtered nilmanifold of degree at most $d$, dimension at most $D$ and complexity at most $M$, with $G^{(s)}$ one-dimensional, and let $g\in\poly(\Z,G)$. Let $\xi$ be a nonzero vertical character on $G^{(s)}$, write $\Xi=(\xi,-\xi,-\xi,\xi)$, and define
\begin{equation}
 \label{eq:rowcol-fourmap}
 \mathbf g(a,a',b,b')=(g(ab),g(a'b),g(ab'),g(a'b')).
\end{equation}
Suppose that for at least $X/(MH)$ of the intervals $I_{A'}$ in \eqref{eq:rowcol-cover-intervals1} there are horizontal characters $\lambda_1,\ldots,\lambda_t$ on $G^4$, with $1\leq t\leq4D$, satisfying
\begin{equation}
 \label{eq:rowcol-concat-local-characters}
 |\lambda_j|\leq M,\qquad
 \|\lambda_j\circ\mathbf g\|_{C^\infty(I_{A'}^2\times J_{A'}^2)}\leq M
 \qquad\text{for every }1\leq j\leq t,
\end{equation}
and
\begin{equation}
 \label{eq:rowcol-concat-local-bracket}
 \Xi([h_1,\ldots,h_s])=0
 \qquad\text{whenever }h_1,\ldots,h_s\in\bigcap_{j=1}^t\ker\lambda_j.
\end{equation}
Then $g$ has an $M^{D^C}$-factorisation on $(X,X+H]$ through a subgroup $G_*\leq G$ of step at most $s-1$. The nilmanifold $G_*/(G_*\cap\Gamma)$, with the induced filtration, has complexity at most $M^{D^C}$ in an induced rational basis, and the rational factor in the factorisation is periodic modulo $\Gamma$ with period at most $M^{D^C}$.
\end{lemma}

\begin{proof}
\emph{A common list of characters.}
For each interval $I_{A'}$ satisfying the hypotheses, let the characters in the statement be $\lambda_{A',1},\ldots,\lambda_{A',t_{A'}}$. Each of these characters on $G^4$ has at most $4D$ integer coordinates, all of absolute value at most $2M$. Since $t_{A'}\leq4D$, there are at most $M^{O(D^2)}$ choices of the character list.

We write $\lambda_{A',j}=(\lambda_{A',j,1},\ldots,\lambda_{A',j,4})$, where the components $\lambda_{A',j,i}$ are horizontal characters on $G$. By the pigeonhole principle, there is a family $\mathcal I_1$ of intervals $I_{A'}$ on which the character lists and their length $t$ are independent of $A'$, and
\begin{equation}
 \label{eq:rowcol-shared-pairs}
 |\mathcal I_1|\geq X/(RH),\qquad R=M^{D^{C_0}},
\end{equation}
where $C_0=C_0(d,s)$ is sufficiently large. Let the common list of characters for $I_{A'}\in\mathcal I_1$ be $\lambda_1,\ldots,\lambda_t$. We write $\lambda_j=(\lambda_{j,1},\ldots,\lambda_{j,4})$, where the components $\lambda_{j,i}$ are horizontal characters on $G$, and define the row and column characters
\begin{equation}
 \label{eq:rowcol-theta-definitions}
\theta_{j}^{\textnormal{row}}=\lambda_{j,1}+\lambda_{j,3},\qquad
\theta_{j}^{\textnormal{col}}=\lambda_{j,1}+\lambda_{j,2}.
\end{equation}

\emph{Smoothness and concatenation.}
For $I_{A'}\in\mathcal I_1$, we consider two diagonal restrictions $\mathbf{g}(a,a',b,b)$ and $\mathbf{g}(a,a,b,b')$ of $\mathbf{g}$. Their character compositions can be written as
\begin{align*}
 \lambda_j(\mathbf g(a,a',b,b))&=\theta_j^{\textnormal{row}}(g(ab))+(\lambda_{j,2}+\lambda_{j,4})(g(a'b)),\\
 \lambda_j(\mathbf g(a,a,b,b'))&=\theta_j^{\textnormal{col}}(g(ab))+(\lambda_{j,3}+\lambda_{j,4})(g(ab')).
\end{align*}
Taking differences in the $a$ variable in the first identity or in the $b$ variable in the second identity makes the second term on the right vanish. We use \eqref{eq:rowcol-concat-local-characters} and bound the zeroth difference by $1/2$ to obtain
\begin{equation*}
 \|\theta_j^{\textnormal{row}}\circ g(b\,\cdot)\|_{C^\infty(I_{A'})}\leq M,\qquad
 \|\theta_j^{\textnormal{col}}\circ g(a\,\cdot)\|_{C^\infty(J_{A'})}\leq M
\end{equation*}
for every $a\in I_{A'}$, every $b\in J_{A'}$ and every $1\leq j\leq t$.

For $b\in J_{A'}$, the interval $(X/b,(X+H)/b]$ has length $\asymp |I_{A'}|$ and lies in the tenfold enlargement of $I_{A'}$. By~\cite[(2.3)]{MRSTT}, we can transfer the smoothness bound to the enlarged interval to conclude that
\begin{equation}
 \label{eq:rowcol-reciprocal-smoothness}
 \begin{aligned}
 \|\theta_j^{\textnormal{row}}\circ g(b\,\cdot)\|_{C^\infty(X/b,(X+H)/b]}&\ll_d M,\\
 \|\theta_j^{\textnormal{col}}\circ g(a\,\cdot)\|_{C^\infty(X/a,(X+H)/a]}&\leq M
 \end{aligned}
\end{equation}
for the same $a,b,j$.

Since the intervals $I_{A'}$ have bounded overlap, from \eqref{eq:rowcol-shared-pairs} we conclude that
\begin{equation*}
 \left|\bigcup_{I_{A'}\in\mathcal I_1}I_{A'}\right|\gg A/R,\qquad
 \left|\bigcup_{I_{A'}\in\mathcal I_1}J_{A'}\right|\gg X/(AR).
\end{equation*}
The second union lies in $(X/(2A),2X/A]$, so one of the dyadic intervals $(B_0,2B_0]$, with either $B_0=X/(2A)$ or $B_0=X/A$, contains $\gg B_0/R$ of its elements. The length hypothesis gives $B_0\geq1$ and $H\geq2\max\{A,B_0\}$. We apply the concatenation result of \cite[Corollary 2.4]{MSTT} to $\theta_j^{\textnormal{col}}\circ g$ at scale $A$ and to $\theta_j^{\textnormal{row}}\circ g$ at scale $B_0$, using density parameter $c_d/R$ for a sufficiently small $c_d>0$. Equation \eqref{eq:rowcol-reciprocal-smoothness} verifies the local hypotheses. The length alternative of that corollary would give
\begin{equation*}
 H\leq M^{D^{O_{d,s}(1)}}\max\{A,X/A\},
\end{equation*}
contradicting our hypothesis when $C$ is sufficiently large. Therefore, each $\theta_j^{\textnormal{row}}$ and $\theta_j^{\textnormal{col}}$ has a positive integer multiple $\widehat\theta_j^{\textnormal{row}}$ or $\widehat\theta_j^{\textnormal{col}}$, respectively, such that
\begin{equation}
 \label{eq:rowcol-concat-global-characters}
 \begin{aligned}
 |\widehat\theta_j^{\textnormal{row}}|&\leq R^{C_1},&\qquad |\widehat\theta_j^{\textnormal{col}}|&\leq R^{C_1},\\
 \|\widehat\theta_j^{\textnormal{row}}\circ g\|_{C^\infty(X,X+H]}&\leq R^{C_1},& \|\widehat\theta_j^{\textnormal{col}}\circ g\|_{C^\infty(X,X+H]}&\leq R^{C_1},
 \end{aligned}
\end{equation}
for every $1\leq j\leq t$, where $C_1=C_1(d,s)$ is sufficiently large. Note that the multiplication of the characters preserves their kernels.

\emph{Step decrease and factorisation.}
We set
\begin{equation*}
 G_*=\bigcap_{j=1}^{t}(\ker\theta_j^{\textnormal{row}}\cap\ker\theta_j^{\textnormal{col}}),
\end{equation*}
and choose one interval $I_{A'}\in \mathcal{I}_1$. For $z_1,\ldots,z_s\in G_*$, the elements
\begin{equation*}
 w_1=(z_1,z_1,1,1),\qquad w_2=(z_2,1,z_2,1),\qquad
 w_i=(z_i,z_i,1,1)\quad\text{for }3\leq i\leq s
\end{equation*}
belong to $\bigcap_{1\leq j\leq t}\ker\lambda_j$, since
\begin{align*}
\lambda_j((z_1,z_1,1,1))&=\lambda_{j,1}(z_1)+\lambda_{j,2}(z_1)=\theta_j^{\textnormal{col}}(z_1)=0,\\
\lambda_j((z_2,1,z_2,1))&=\lambda_{j,1}(z_2)+\lambda_{j,3}(z_2)=\theta_j^{\textnormal{row}}(z_2)=0
\end{align*}
by \eqref{eq:rowcol-theta-definitions}. In $G^4$, commutators are taken coordinatewise, so $$[w_1,w_2]=([z_1,z_2],[z_1,1],[1,z_2],1)=([z_1,z_2],1,1,1).$$
Applying commutators iteratively to this shows that $[w_1,\ldots, w_s]= ([z_1,\ldots,z_s],1,1,1)$, so \eqref{eq:rowcol-concat-local-bracket} gives
\begin{equation}
 \label{eq:rowcol-bracket-isolation}
 0=\Xi([w_1,\ldots,w_s])=\xi([z_1,\ldots,z_s]).
\end{equation}
The homomorphism $\xi$ is not identically zero, and hence is injective on the one-dimensional group $G^{(s)}$. Hence $G_*$ has step at most $s-1$.

We choose a maximal linearly independent sublist of the nonzero characters $\widehat\theta_j^{\textnormal{row}}, \widehat\theta_j^{\textnormal{col}}$ in \eqref{eq:rowcol-concat-global-characters}. The sublist is nonempty, since otherwise $G_*=G$, which contradicts $G_{*}$ having step at most $s-1$. The common kernel of the sublist is $G_*$. We apply Lemma \ref{lem:character-factorisation} to $g$ on $(X,X+H]$ and this sublist, with character bound $R^{C_1}$. This gives a factorisation through $G_*$ with smoothness, rationality and induced complexity bounded by $M^{D^{O_{d,s}(1)}}$. Lemma \ref{lem:character-factorisation} gives a bound of the same shape for the period of the rational factor modulo $\Gamma$. Choosing $C=C(d,s)$ sufficiently large proves the lemma.
\end{proof}

We now prove the non-abelian Type II estimate.

\begin{proposition}[Efficient non-abelian Type II inverse theorem]
\label{cor:efficient-typeII-inverse}
Let $d,s$ be integers with $2\leq s\leq d$. There is $C_{d,s}\geq1$ with the following property. Let $D,X,H,A\in\N$ and $M\geq2$ satisfy $D\geq2$, $2\leq H\leq X$ and $1\leq A\leq X$, and let $0<\delta\leq1/10$. Let $G/\Gamma$ be an $s$-step filtered nilmanifold of degree at most $d$, dimension at most $D$ and complexity at most $M$, with $G^{(s)}$ one-dimensional. Let $g\in\poly(\Z,G)$, and let $F\colon G/\Gamma\to\C$ be a $1$-bounded $M$-Lipschitz vertical nilcharacter with nonzero frequency $\xi$ satisfying $|\xi|\leq M/\delta$. Suppose that $f=\alpha*\beta$ is a $(\delta,A,2A)$ Type II sum, with $\alpha$ supported on $(A,2A]$, and
\begin{equation}
 \label{eq:rowcol-bilinear-correlation}
 \left|\sum_{X<n\leq X+H}f(n)F(g(n)\Gamma)\right|^*\geq\delta H.
\end{equation}
Set $Q=(M/\delta)^{D^{C_{d,s}}}$. Then either
\begin{equation}
 \label{eq:rowcol-bilinear-short}
 H\leq Q\max\{A,X/A\},
\end{equation}
or $g$ has a $Q$-factorisation on $(X,X+H]$ through a subgroup $G_*\leq G$ of step at most $s-1$. In the latter case, $G_*/(G_*\cap\Gamma)$, with the induced filtration, has complexity at most $Q$ in an induced rational basis, and the rational factor has period at most $Q$ modulo $\Gamma$.
\end{proposition}

\begin{proof} \emph{Reduction to single-variable sums.}
We assume that \eqref{eq:rowcol-bilinear-short} fails, so $AH/X>Q$ and $H/A>Q$. Choose a progression $P\subset(X,X+H]$ attaining \eqref{eq:rowcol-bilinear-correlation}, so
\begin{equation}
 \label{eq:rowcol-selected-bilinear}
 \left|\sum_{A<a\leq2A}\alpha(a)S_a\right|\geq\delta H,\qquad
 S_a=\sum_{b:\,ab\in P}\beta(b)F(g(ab)\Gamma).
\end{equation}
For each $a$, the inner range is a progression in $(X/a,(X+H)/a]$.

We truncate $\beta$ at $\delta^{-4}$, writing $\beta_{\leq}=\beta\1_{|\beta|\leq\delta^{-4}}$ and $\beta_>=\beta-\beta_{\leq}$. Each $S_a$ has at most $2H/A$ terms, and each $b$ in the definition of $S_a$ lies in  $(X/(2A),2X/A]$ and appears for at most $3AH/X$ values of $a$. Since $|\beta_>(b)|^2\leq\delta^8|\beta(b)|^4$, Cauchy--Schwarz and \eqref{eq:type-II-moments} give
\begin{equation}
 \label{eq:rowcol-tail-second-moment}
 \begin{aligned}
 \sum_{A<a\leq2A}\left|\sum_{b:\,ab\in P}\beta_>(b)F(g(ab)\Gamma)\right|^2
&\leq \delta^8\cdot (2H/A)\cdot (3AH/X)\sum_{b\leq2X/A}|\beta(b)|^4\\
 &\leq12\delta^6H^2/A.
 \end{aligned}
\end{equation}
Also, \eqref{eq:type-alpha-second-moment} gives $\sum_{A<a\leq2A}|\alpha(a)|^2\leq2\delta^{-1}A$. We set
\begin{align*}
 S_{a,\leq}=\sum_{b:\,ab\in P}\beta_{\leq}(b)F(g(ab)\Gamma).
\end{align*}
Recalling \eqref{eq:rowcol-selected-bilinear}, a second application of Cauchy--Schwarz yields
\begin{equation}
 \label{eq:rowcol-truncated-bilinear}
 \left|\sum_{A<a\leq2A}\alpha(a)S_{a,\leq}\right|
 \geq\delta H-\sqrt{24/\delta}\,\delta^3H\geq\delta H/2,
\end{equation}
where we used $\delta\leq1/10$.

Let $\mathcal A$ be the set of $a\in(A,2A]$ for which $|S_{a,\leq}|\geq\delta^4\cdot (2H/A)$. The values $a\notin\mathcal A$ contribute at most
\begin{align*}
\delta^4(2H/A)\sum_{A<a\leq2A}|\alpha(a)|\leq2\sqrt{2/\delta}\,\delta^4H\leq\delta H/4.
\end{align*}
Since $|S_{a,\leq}|\leq\delta^{-4}(2H/A)$, Cauchy--Schwarz applied to $|\sum_{a\in \mathcal A}\alpha(a)S_{a,\leq}|$ gives
\begin{equation}
 \label{eq:rowcol-good-row-count}
 \frac{\delta H}{4}\leq\delta^{-4}(2H/A)\sqrt{\frac{2A}{\delta}|\mathcal A|},\qquad
 |\mathcal A|\geq\frac{\delta^{11}}{128}A\geq\delta^{14}A.
\end{equation}
Denoting $\rho'=\delta^{14}$, we conclude that at least $\rho'A$ integers $a\in(A,2A]$ satisfy
\begin{equation}
 \label{eq:rowcol-many-good-rows}
\left|\sum_{X/a<b\leq(X+H)/a}\beta_{\leq}(b)F(g(ab)\Gamma)\right|^*
 \geq\rho' H/A.
\end{equation}

\emph{Removing the coefficient.} We use the cover of $(A,2A]$ by intervals $I_{A'}$ from \eqref{eq:rowcol-cover-intervals1},~\eqref{eq:rowcol-cover-intervals2}. For a small absolute constant $c>0$, the intervals $I_{A'}$ containing fewer than $c\rho'AH/X$ elements of $\mathcal A$ contain at most $\rho'A/2$ such elements in total. Since each interval has $O(AH/X)$ elements, at least $\gg\rho'X/H$ intervals $I_{A'}$ contain at least $c\rho'AH/X$ elements of $\mathcal A$. We call these intervals good.

We fix a good interval and write $I=I_{A'}$, $J=J_{A'}$ and $z_a(b)=F(g(ab)\Gamma)$. For each $a\in I\cap\mathcal A$, we choose a progression $R_a\subset(X/a,(X+H)/a]\subset J$ on which the sum in \eqref{eq:rowcol-many-good-rows} has absolute value at least $\rho'H/A$. Since $|\beta_{\leq}(b)z_a(b)|\leq\delta^{-4}\leq(\rho')^{-1}$, we have $|R_a|\geq(\rho')^2H/A$. Let $q_a$ be the common difference of $R_a$. The bound $|J|\ll H/A$ gives $q_a\ll(\rho')^{-2}$.

We increase $C_{d,s}$ so that $(\rho')^2H/A\geq16$ and set $\ell=\lfloor(\rho')^2H/(8A)\rfloor$. We round the endpoints of the integer interval determined by each $R_a$ inwards to the grid $(\min J)+\ell\Z$. This removes at most $2\ell$ terms and changes the sum by at most $2\ell(\rho')^{-1}\leq\rho'H/(4A)$. Let $r_a\bmod q_a$ be the residue class containing $R_a$. There are $O((\rho')^{-4})$ choices for $(q_a,r_a\bmod q_a)$ and $O((\rho')^{-4})$ pairs of grid endpoints. By the pigeonhole principle, some rounded progression $R\subset J$ satisfies
\begin{align*}
 \left|\sum_{b\in R}\beta_{\leq}(b)z_a(b)\right|\geq\rho'H/(2A)
\end{align*}
for $\gg(\rho')^9|I|$ values of $a\in I$.

We set $c(b)=\beta_{\leq}(b)\1_R(b)$. Since $|I|\asymp AH/X$ and $|J|\asymp H/A$, we have
\begin{align*}
 (\rho')^{11}&\ll\E_{a\in I}\left|\E_{b\in J}c(b)z_a(b)\right|^2.
\end{align*}
We expand the square and apply Cauchy--Schwarz in $(b,b')$ to obtain
\begin{align*}
 \left(\E_{a\in I}\left|\E_{b\in J}c(b)z_a(b)\right|^2\right)^2
 &\leq\left(\E_{b\in J}|c(b)|^2\right)^2\E_{b,b'\in J}\left|\E_{a\in I}z_a(b)\overline{z_a(b')}\right|^2\\
 &=\left(\E_{b\in J}|c(b)|^2\right)^2\E_{a,a'\in I}\left|\E_{b\in J}z_a(b)\overline{z_{a'}(b)}\right|^2.
\end{align*}
Consequently, using $\E_{b\in J}|c(b)|^2\leq(\rho')^{-2}$, every good interval satisfies
\begin{equation}
 \label{eq:rowcol-Leng-correlation}
 \E_{a,a'\in I_{A'}}\left|\E_{b\in J_{A'}}F(g(ab)\Gamma)\overline{F(g(a'b)\Gamma)}\right|^2\gg(\rho')^{26}.
\end{equation}

\emph{Horizontal characters and factorisation.}
Let $\mathbf g$ be the polynomial sequence in \eqref{eq:rowcol-fourmap}. Define
\begin{equation}
 \label{eq:rowcol-corner-product}
 \mathbf F(x_1,x_2,x_3,x_4)=F(x_1)\overline{F(x_2)F(x_3)}F(x_4).
\end{equation}
Note that the function $\mathbf F$ is a $1$-bounded vertical nilcharacter on $G^4/\Gamma^4$ with nonzero frequency $\Xi=(\xi,-\xi,-\xi,\xi)$. Expanding \eqref{eq:rowcol-Leng-correlation} gives
\begin{equation}
 \label{eq:rowcol-product-correlation}
 \E_{a,a'\in I_{A'}}\E_{b,b'\in J_{A'}}\mathbf F(\mathbf g(a,a',b,b')\Gamma^4)\gg(\rho')^{26}.
\end{equation}

We equip $G^4$ with the doubled product filtration $((G_{\lceil j/2\rceil})^4)_{j\geq0}$ and write the Taylor expansion
\begin{align*}
 g(n)=\prod_{i=0}^d g_i^{\binom ni},\qquad g_i\in G_i.
\end{align*}
For each $i$, the function $(a,b)\mapsto\binom{ab}{i}$ is an integer-valued polynomial of total degree at most $2i$. Thus the map $(a,b)\mapsto g_i^{\binom{ab}{i}}$ is polynomial for the doubled filtration
\begin{align*}
 G^{[2]}_r=G_{\lceil r/2\rceil},
\end{align*}
since a derivative of total order $r$ of this map can be nontrivial only when $r\leq2i$, in which case
\begin{align*}
 G_i\subseteq G_{\lceil r/2\rceil}=G^{[2]}_r.
\end{align*}
Polynomial maps for a filtration are closed under pointwise multiplication, so $(a,b)\mapsto g(ab)$ is polynomial for $G^{[2]}_\bullet$. The same argument applies to $g(a'b)$, $g(ab')$ and $g(a'b')$. Therefore
\begin{align*}
 \mathbf g(a,a',b,b')=(g(ab),g(a'b),g(ab'),g(a'b'))
\end{align*}
is polynomial of degree at most $2d$ for the doubled product filtration.

Note that the product nilmanifold has step $s$, dimension at most $4D$ and complexity at most $M^{D^{O_d(1)}}$. The product basis and \cite[Lemma B.3]{Leng} give $\|\mathbf F\|_{\Lip}\leq M^{D^{O_d(1)}}$, while $$|\Xi|\leq |\xi|+|\xi|+|\xi|+|\xi|\leq4M/\rho'.$$

We apply the quantitative equidistribution result of \cite[Theorem 4]{Leng} on each good rectangle $I_{A'}^2\times J_{A'}^2$, after translating its intervals to start at $1$, with correlation parameter $c(\rho')^{26}$ for a sufficiently small absolute $c>0$. The side lengths are $\gg H/\max\{A,X/A\}$, so the assumption $H>Q\max\{A,X/A\}$ excludes the short length alternative of that theorem after increasing $C_{d,s}$. The other alternative gives \eqref{eq:rowcol-concat-local-characters} and \eqref{eq:rowcol-concat-local-bracket} with $(M/\rho')^{D^{C_1}}$ in place of $M$ for some $C_1=C_1(d,s)$, using \eqref{eq:Cinfty-norm-comparison}. There are $\gg\rho' X/H\geq X/((M/\rho')^{D^{C_1}}H)$ good intervals, after increasing $C_1$. If $C_2=C_2(d,s)$ is the constant in Lemma \ref{lem:rowcol-bilinear-concatenation}, we take $C_{d,s}$ sufficiently large that $((M/\rho')^{D^{C_1}})^{D^{C_2}}\leq Q$. That lemma then applies with parameter $(M/\rho')^{D^{C_1}}$ and gives the desired factorisation and all its bounds.
\end{proof}

\section{Model functions and decomposition in short intervals}
\label{sec:short-arithmetic-models}

In this section, we first approximate the Cram\'er and Siegel models by Type I sums and compare their short interval Gowers norms; this is done in Subsection~\ref{subsec:local}. We then decompose in Lemma~\ref{lem:growing-decomposition} the functions $\Lambda-\Lambda_w$ and $\Lambda-\widetilde\Lambda_w$ into Type I, Type $I_2$, and Type II sums, with an an estimate twisted by $n^{iT}$ for the Type II components of high frequency, with quasipolynomial bounds. Finally, in Lemma~\ref{lem:abelian-typeII-reduction} we reduce abelian type II sums to the sums twisted by $n^{iT}$ and in Proposition~\ref{prop:short-major-arcs} we handle sums of the von Mangoldt function minus the Siegel correction over arithmetic progressions.

\subsection{Models functions over short intervals}\label{subsec:local}

The following lemma shows that the Cram\'er model and its twists are close to Type I sums in $\ell^1$ norm. This has been proved in many settings (see e.g.~\cite[Lemma 2.19]{MRSTT}), but we have been unable to find a reference that fully captures both the short interval and error term aspects below. 

\begin{lemma}[Type I truncation of the Cram\'er model and its exceptional twist]
\label{lem:local-model-typeI}
There is some $X_0\geq3$ such that the following holds. Let $X\geq X_0$, let $2\leq w\leq\exp((\log X)^{1/10})$, and let $I\subset[X,3X]\cap\Z$ be an interval of cardinality $H\geq X^{1/2}$. Then there is a sequence $a(\cdot)$, depending only on $X,w$ and supported on $[1,X^{1/3}]$, such that, for every $Y\geq1$,
\begin{align*}
 \sum_{d\leq Y}|a(d)|^2\ll Y(\log(2w))^2
\end{align*}
and
\begin{equation}
 \label{eq:local-model-L1}
 \sum_{n\in I}|\Lambda_w(n)-(a*1)(n)|\leq CH\exp(-(\log X)^{4/5}).
\end{equation}
If the level $w$ exceptional pair $(\beta,\chi)$ exists and has exceptional conductor $q$, then there exist a sequence $a_{\chi}(\cdot)$ supported on $[1,X^{1/3}]$ and a sequence $b_{\chi}(\cdot)$, such that for every $Y\geq 1$,
\begin{equation}
 \label{eq:twisted-model-coefficient-bounds}
 \|b_\chi\|_{\TV(\N;q)}\leq 2q,\qquad \sum_{d\leq Y}|a_\chi(d)|^2\ll Y(\log(2w))^2,
\end{equation}
we have
\begin{align}\label{eq:twisted-approximation}
 \sum_{n\in I}|\Lambda_w(n)\chi(n)n^{\beta-1}-(a_\chi*b_\chi)(n)|\leq H\exp(-c(\log X)^{4/5}).
\end{align}
\end{lemma}

\begin{proof}
Let $D_0=\lfloor X^{1/3}\rfloor$, and let $\lambda_d^-$ and $\lambda_d^+$ be the lower and upper bound linear sieve weights of level $D_0$. We set
\begin{align*}
 a(d)=\frac{P(w)}{\varphi(P(w))}\lambda_d^-.
\end{align*}
Then $\textnormal{supp}(a)\subset [1,D_0]$ and
\begin{equation}
 \label{eq:local-sieve-sandwich}
 0\leq\Lambda_w(n)-(a*1)(n)\leq\frac{P(w)}{\varphi(P(w))}\sum_{d\mid n}(\lambda_d^+-\lambda_d^-).
\end{equation}
Summing \eqref{eq:local-sieve-sandwich} over $I$, using the fact that $\sum_{n\in I}\1_{d\mid n}=H/d+O(1)$ and applying the fundamental lemma of sieve theory \cite[Lemma 6.3]{IK} and the Mertens bound
\begin{align}\label{eq:Mertens}
\frac{P(w)}{\varphi(P(w))}\ll \log w,    
\end{align}
we obtain, for some constant $c>0$,
\begin{align*}
 \sum_{n\in I}|\Lambda_w(n)-(a*1)(n)|
 &\ll H(\log(2w))\exp(-\log D_0/\log w)+D_0\log(2w)\\
 &\ll H\exp(-(\log X)^{9/10}/10),
\end{align*}
which implies the claimed bound.
Also,
\begin{align*}
 \sum_{d\leq Y}|a(d)|^2\leq Y\left(\frac{P(w)}{\varphi(P(w))}\right)^2\ll Y(\log(2w))^2.
\end{align*}

For the exceptional case, define
\begin{equation}
 \label{eq:twisted-long-factor}
 b_\chi(n)=\chi(n)n^{\beta-1},\qquad a_\chi(d)=a(d)b_\chi(d).
\end{equation}
The function $b_\chi$ is completely multiplicative and satisfies $|b_\chi|\leq1$. Thus
\begin{align*}
 (a_\chi*b_\chi)(n)=\sum_{d\mid n}a(d)b_\chi(d)b_\chi(n/d)=b_\chi(n)(a*1)(n),\qquad |a_\chi(d)|\leq|a(d)|,
\end{align*}
which proves~\eqref{eq:twisted-approximation} and the second moment bound. On each residue class modulo $q$, $b_\chi$ is monotone with absolute value at most one. Summing the variation bounds over the residue classes, we obtain
\begin{align*}
 \|b_\chi\|_{\TV(\N;q)}\leq\sum_{a\bmod q}2=2q.
\end{align*}
\end{proof}

We next establish a mixed model asymptotic for Cram\'er models of different levels that we will use in Lemmas \ref{lem:local-exceptional} and \ref{lem:local-model-comparison}. This is a variant of~\cite[Proposition 5.2]{TT}.

\begin{lemma}[Mixed model asymptotic]
\label{lem:mixed-cutoff-cube}
Let $k\geq2$ be an integer and let $\eps>0$. Let
$I=x+[H]\subset[X,3X]\cap\Z$, where $x,H\in\Z$ and
\[
 X^{1/2+\eps}\leq H\leq X^{1-\eps}.
\]
For $\omega\in\{0,1\}^k$, write
\[
 \psi_\omega(n,\mathbf h)=n+\omega\cdot\mathbf h,
\]
and let $\mathcal C_I=\{(n,\mathbf h)\in\Z^{k+1}\colon \psi_\omega(n,\mathbf h)\in I\,\, \textnormal{for all}\,\, \omega\in \{0,1\}^k\}$.

Let
\[
 2\leq w\leq z\leq\exp((\log X)^{1/10}),
\]
and let $\chi$ be a primitive real character of conductor $q\leq w$. There is a quantity $B_{\chi,w}\geq0$
such that
\[
 B_{\chi,w}\ll_k
 \left(\frac q{\varphi(q)}
 \|\chi\|_{U^k(\Z/q\Z)}\right)^{2^k}
\]
and, uniformly for every choice $v_\omega\in\{w,z\}$,
\begin{align}
 \label{eq:mixed-cutoff-cube}
 \E_{(n,\mathbf h)\in\mathcal C_I}
 \prod_{\omega\in \{0,1\}^k}
 \chi(\psi_\omega(n,\mathbf h))
 \Lambda_{v_\omega}(\psi_\omega(n,\mathbf h))
 =B_{\chi,w}
 +O_{k,\eps}\left(
 \frac{B_{\chi,w}}{w}
 +\exp(-c(\log X)^{4/5})\right)
\end{align}
for some $c=c(k,\eps)>0$. If $v_\omega=w$ for every $\omega$, the
term $B_{\chi,w}/w$ may be omitted.
\end{lemma}

\begin{proof}
Write
\begin{align*}
 \Omega_I=\{(y,\mathbf t)\in\R^{k+1}:
 y+\omega\cdot\mathbf t\in(x,x+H]
 \text{ for every }\omega\in\{0,1\}^k\}.
\end{align*}
We split the left-hand side of~\eqref{eq:mixed-cutoff-cube}, up to a small error of $O(X^{-1/2})$, say,  into residue classes as
\begin{align}
\E_{(a,\mathbf{b})\in (\Z/q\Z)^{k+1}}\E_{(m,\mathbf u)\in\mathbb{Z}^{k+1}\cap \frac{1}{q}(\Omega_I-(a,\mathbf{b}))}
 \prod_{\omega\in \{0,1\}^k}
 \chi(\psi_\omega(a,\mathbf b))
 \Lambda_{v_\omega}(q\psi_\omega(m,\mathbf u)+a+\omega\cdot \mathbf{b})\eqqcolon \Sigma.    
\end{align}
The set $\frac{1}{q}(\Omega_I-(a,\mathbf{b}))$ in the $(m,\mathbf{u})$ summation has volume
\begin{align*}
 q^{-k-1}\operatorname{vol}(\Omega_I)
\end{align*}
and diameter $O_k(H/q)$. Set
\begin{align*}
 N=\frac Hq\gg X^{1/2+\eps/2}.
\end{align*}
We may then apply the sieve calculation in \cite[proof of Proposition 5.2]{TT}, using the residue decomposition in
\cite[proof of Theorem 2.5, equation (5.15) and below]{TT}. Note that this calculation is uniform in the constant coefficients $a+\omega\cdot\mathbf b$. Now we obtain
\begin{align}\label{eq:chi-sum}
\begin{aligned}
\Sigma&=\left(\frac{q}{\varphi(q)}\right)^{2^k}\E_{(a,\mathbf b)\in(\Z/q\Z)^{k+1}}
 \prod_{\omega\in \{0,1\}^k}\chi(a+\omega\cdot\mathbf b)\prod_{\substack{p\leq w\\p\nmid q}}\beta_p(\{0,1\}^k)\prod_{\substack{w<p\leq z\\p\nmid q}}\beta_p(S_p)\\
 &\qquad +O(\exp(-c(\log X)^{4/5})),  
 \end{aligned}
\end{align}
where 
\begin{align*}
 \beta_p(S)=\E_{\mathbf r\in(\Z/p\Z)^{k+1}}
 \prod_{\omega\in S}
 \frac p{p-1}\1_{\psi_\omega(\mathbf r)\ne0\pmod p},\qquad S_p=\{\omega\in \{0,1\}^k\colon v_{\omega}\geq p\}.
\end{align*}

For distinct $\omega,\omega'$, the forms $\psi_\omega,\psi_{\omega'}$ are linearly independent modulo every
prime. Hence 
\begin{align*}
 \beta_p(S)
 =1+O_k(p^{-2})
\end{align*}
for all $S\subset \{0,1\}^k$. Hence,
\begin{align*}
\prod_{\substack{p\leq w\\p\nmid q}}\beta_p(\{0,1\}^k)\prod_{\substack{w<p\leq z\\p\nmid q}}\beta_p(S_p)\ll_k 1.    
\end{align*}
Now the claim follows from
\begin{align*}
 \E_{(a,\mathbf b)\in(\Z/q\Z)^{k+1}}
 \prod_\omega\chi(a+\omega\cdot\mathbf b)=\|\chi\|_{U^k(\Z/q\Z)}^{2^k}.
\end{align*}
\end{proof}

\begin{lemma}[Gowers norm bound for character-twisted models]
\label{lem:local-exceptional}
Let $k\geq2$ be an integer and let $\eps>0$. There are constants $c=c(k,\eps)>0$, $C=C(k,\eps)\geq1$, and $X_0=X_0(k,\eps)\geq3$ such that the following holds.

Let $X\geq X_0$ and let $I=x+[H]\subset[X,3X]\cap\Z$, where $x,H\in\Z$ and
\begin{align*}
 X^{1/2+\eps}\leq H\leq X^{1-\eps},\qquad 3\leq q\leq w\leq\exp((\log X)^{1/10}).
\end{align*}
Let $\chi$ be a primitive real Dirichlet character of conductor $q$, and let $1/2<\beta<1$. Then
\begin{equation}
 \label{eq:local-exceptional-bound}
 \|\Lambda_w\chi(\cdot)(\cdot)^{\beta-1}\|_{U^k(I)}\leq CX^{\beta-1}q^{-2^{-k-2}}+C\exp(-c(\log X)^{4/5}).
\end{equation}
\end{lemma}

\begin{proof}
We take $z=w$ in Lemma~\ref{lem:mixed-cutoff-cube} and use the definition of Gowers norms to obtain
\begin{align*}
 \|\Lambda_w\chi\|_{U^k(I)}\ll_k\frac q{\varphi(q)}\|\chi\|_{U^k(\Z/q\Z)}+\exp(-c(\log X)^{4/5}/2^k).
\end{align*}
By \cite[proof of Lemma 5.6]{TT} (which on inspection does not require the character there to be exceptional) and the crude bound $q/\varphi(q)\ll_k q^{2^{-k-3}}$, we have
\begin{equation}
 \label{eq:exceptional-character-bound}
 \frac q{\varphi(q)}\|\chi\|_{U^k(\Z/q\Z)}\ll_k q^{2^{-k-3}}q^{-2^{-k-1}+2^{-k-3}}=q^{-2^{-k-2}}.
\end{equation}
Hence,
\begin{equation}
 \label{eq:local-exceptional-unweighted}
 \|\Lambda_w\chi\|_{U^k(I)}\ll_{k,\eps}q^{-2^{-k-2}}+\exp(-c(\log X)^{4/5}).
\end{equation}

Since $x\asymp X$, the mean value theorem gives
\begin{align*}
 \sup_{n\in I}|n^{\beta-1}-x^{\beta-1}|\ll X^{\beta-1}H/X.
\end{align*}
The triangle inequality and $\|f\|_{U^k(I)}\leq\|f\|_\infty$ therefore give
\begin{equation}
 \label{eq:exceptional-weight-freezing}
 \|\Lambda_w\chi(\cdot)(\cdot)^{\beta-1}\|_{U^k(I)}\leq x^{\beta-1}\|\Lambda_w\chi\|_{U^k(I)}+O\bigl(X^{\beta-1}(H/X)\log(2w)\bigr).
\end{equation}
The error in \eqref{eq:exceptional-weight-freezing} is $O(X^{-\varepsilon/2})$ since $H\leq X^{1-\eps}$. Combining \eqref{eq:local-exceptional-unweighted} and \eqref{eq:exceptional-weight-freezing} proves \eqref{eq:local-exceptional-bound}.
\end{proof}

\begin{lemma}[Comparison of models]
\label{lem:local-model-comparison}
Let $k\geq2$ be an integer and let $\eps>0$. There are constants $c=c(k,\eps)>0$ and $X_0=X_0(k,\eps)\geq3$ such that the following holds.

Let $X\geq X_0$ and let $I=x+[H]\subset[X,3X]\cap\Z$, where $x,H\in\Z$ and
\begin{align*}
 X^{1/2+\eps}\leq H\leq X^{1-\eps},\qquad 2\leq w\leq z\leq\exp((\log X)^{1/10}).
\end{align*}
Then
\begin{equation}
 \label{eq:local-model-comparison}
 \|\widetilde\Lambda_z-\widetilde\Lambda_w\|_{U^k(I)}\ll_{k,\eps}w^{-2^{-k-2}}+\exp(-c(\log X)^{4/5}).
\end{equation}
\end{lemma}

\begin{proof}
We first consider
\begin{align*}
 \|\chi(\widetilde\Lambda_z-\widetilde\Lambda_w)\|_{U^k(I)},    
\end{align*}
where $\chi$ is a primitive real character $\chi$ of conductor $q\leq w$, with $\chi=q=1$ allowed. We apply the definition of Gowers norms and Lemma~\ref{lem:mixed-cutoff-cube} to obtain, for some nonnegative quantities 
\begin{align*}
B_{\chi,w}\ll_k \left(\frac q{\varphi(q)}\|\chi\|_{U^k(\Z/q\Z)}\right)^{2^k}    
\end{align*}
the estimate
\begin{align*}
\|\chi(\widetilde\Lambda_z-\widetilde\Lambda_w)\|_{U^k(I)}^{2^k}=\sum_{\mathbf{j}\in \{0,1\}^{\{0,1\}^k}}(-1)^{\sum_{\omega}(1-j_{\omega})}B_{\chi,w}+O_{k,\varepsilon}(B_{k,\omega}/w+w^{-1}).    
\end{align*}
The main term simplifies to $(1-1)^{2^k}=0$, and by~\eqref{eq:exceptional-character-bound} we have $B_{k,w}\ll q^{-2^{k-2}}$. 

We therefore obtain
\begin{equation}
 \label{eq:local-cramer-comparison}
 \|\chi(\Lambda_z-\Lambda_w)\|_{U^k(I)}\ll_{k,\eps}w^{-2^{-k}}.
\end{equation}

If the same exceptional pair $(\beta,\chi)$ occurs at both levels, then
\begin{align*}
 \widetilde\Lambda_z(n)-\widetilde\Lambda_w(n) &=(\Lambda_z(n)-\Lambda_w(n))(1-\chi(n)n^{\beta-1}).
\end{align*}
We freeze $n^{\beta-1}$ at $x^{\beta-1}$ as in \eqref{eq:exceptional-weight-freezing}. Since $|\Lambda_z-\Lambda_w|\ll\log(2z)$, this gives
\begin{align*}
 \|\widetilde\Lambda_z-\widetilde\Lambda_w\|_{U^k(I)}
 &\leq\|\Lambda_z-\Lambda_w\|_{U^k(I)}+x^{\beta-1}\|\chi(\Lambda_z-\Lambda_w)\|_{U^k(I)}\\
 &\quad+O((H/X)\log(2z)).
\end{align*}
Applying \eqref{eq:local-cramer-comparison} with $\chi=1$ and with $\chi$ being the exceptional character proves \eqref{eq:local-model-comparison} in the case of a common exceptional pair. If neither pair exists, we use \eqref{eq:local-cramer-comparison} with $\chi=1$ directly.

Lastly, suppose that an exceptional pair $(\beta_z,\chi_z)$ exists at level $z$ but not at level $w$. Then the conductor $q_z$ satisfies $q_z>w$. Lemma \ref{lem:local-exceptional} gives
\begin{equation}
 \label{eq:new-exceptional-correction}
 \|\Lambda_z\chi_z(\cdot)(\cdot)^{\beta_z-1}\|_{U^k(I)}\ll_{k,\eps}w^{-2^{-k-2}}+\exp(-c(\log X)^{4/5}).
\end{equation}
If an exceptional pair $(\beta_w,\chi_w)$ exists at level $w$ but not at level $z$, then $1-\beta_w>c_0/\log z$. Thus
\begin{equation}
 \label{eq:old-exceptional-correction}
 \sup_{n\in I}|\Lambda_w(n)\chi_w(n)n^{\beta_w-1}|\ll\log(2w)X^{-c_0/\log z}\ll\exp(-c_0(\log X)^{9/10}/2).
\end{equation}
The two estimates \eqref{eq:new-exceptional-correction} and \eqref{eq:old-exceptional-correction} also apply when different pairs $(\beta_w,\chi_w)$ and $(\beta_z,\chi_z)$ occur at levels $w$ and $z$. Combining them with~\eqref{eq:new-exceptional-correction} and the $\chi=1$ case of \eqref{eq:local-cramer-comparison} and the triangle inequality, we obtain
\begin{align*}
 \|\widetilde\Lambda_z-\widetilde\Lambda_w\|_{U^k(I)}\ll_{k,\eps}w^{-2^{-k-2}}+\exp(-c(\log X)^{4/5}).
\end{align*}
for some $c>0$.
\end{proof}

\subsection{Arithmetic decomposition}

We will need the following Dirichlet polynomial estimate for the sequences $\mu, 1, \log$ arising from Heath-Brown's identity. While this is standard, we were unable to locate a reference with the desired level of uniformity. 

\begin{lemma}[Dirichlet polynomials at high frequency]
\label{lem:high-frequency-DP}
Fix $0<\alpha\leq1$, $B\geq1$, and $0<\vartheta\leq B$. There is $X_0=X_0(\alpha,\vartheta,B)\geq3$ with the following property. 

Let $X\geq X_0$, let $L>0$ and $u\in\R$, and let $r,q\in\N$ satisfy
\begin{align}\begin{aligned}\label{eq:rqu}
 1\leq r\leq\exp((\log X)^{1/10}),\qquad
 1\leq q\leq\exp((\log X)^{1/10}),\\
 X^\alpha\leq L\leq X,\qquad
 X^\vartheta\leq|u|\leq X^B.
 \end{aligned}
\end{align}
Let $\psi$ be a Dirichlet character modulo $q$ and let $a_0$ be one of the arithmetic functions $\mu$, $1$, or $\log$. Then, for every integer interval $J\subset[L,2L]$,
\begin{equation}
 \label{eq:high-frequency-DP}
 \left|\sum_{\ell\in J}
 \frac{a_0(r\ell)\psi(\ell)}{\ell^{1/2+iu}}\right|\leq L^{1/2}\exp(-(\log X)^{1/10}).
\end{equation}
\end{lemma}

\begin{proof}
For $a_0=1$, we apply \cite[Lemma~3.9(i)]{MSTT} with scales $X_{*}=X^B$ and length $L\geq X_{*}^{\alpha/B}$. The modulus condition there holds for large $X$, since $q\leq X^{\alpha/2}$. Hence, for some $\delta=\delta(\alpha/B)>0$, we obtain
\begin{align*}
 \left|\sum_{\ell\in J}\frac{\psi(\ell)}{\ell^{1/2+iu}}\right|
 \ll_{\alpha,B}L^{1/2}\left(X^{-B\delta}+\frac{\log X}{(|u|+1)^{1/2}}\right),
\end{align*}
which in light of~\eqref{eq:rqu} proves \eqref{eq:high-frequency-DP} for $a_0=1$.

For $a_0=\log$, we can apply partial summation, gaining an extra factor of $O(\log X)$, since $\log(r\ell)$ has supremum and total variation $O(\log X)$ on $[L,2L]$. The power saving for $a_0=1$ absorbs this factor and yields \eqref{eq:high-frequency-DP}.

We now assume that $a_0=\mu$. We may assume that $r$ is squarefree. Let $\psi^{*}$ be the primitive character inducing $\psi$, of conductor $q^{*}$. For
$\operatorname{Re}s>1$,
\begin{equation*}
\sum_{\ell\geq1}\frac{\mu(r\ell)\psi(\ell)}{\ell^s} =\frac{\mu(r)E_{r,q}(s)}{L(s,\psi^{*})},
\qquad
E_{r,q}(s)=\prod_{\substack{p\mid rq\\ p\nmid q^{*}}}
(1-\psi^*(p)p^{-s})^{-1}.
\end{equation*}

Set
\begin{align*}
\eta=c(\log X)^{-2/3}(\log\log X)^{-1/3},
\end{align*}
where $c=c(B)>0$ is sufficiently small. Since
$|u|\geq X^{\vartheta}$ and $q^{*}\leq\exp((\log X)^{1/10})$, the Vinogradov--Korobov zero-free region gives
\begin{align*}
L(\sigma+i\tau,\psi^{*})\ne0
\end{align*}
whenever
\begin{align*}
\sigma\geq1-2\eta,\qquad
|u|/2\leq|\tau|\leq3|u|/2.
\end{align*}
The standard bound for $|L(s,\psi^{*})|^{-1}$ in a slightly smaller zero-free region gives
\begin{align*}
|L(\sigma+i\tau,\psi^*)|^{-1}
\ll (\log X)^{O_B(1)}
\end{align*}
for $\sigma\geq1-\eta$ in the same range of $\tau$. Also,
\begin{align*}
|E_{r,q}(\sigma+i\tau)|
\ll (\log X)^{O_(1)}
\end{align*}
there, since $\eta\log(rq)\ll (\log X)^{-1/2}$.

We apply Perron's formula to $\mu(r\ell)\psi(\ell)$, truncated at height $|u|/2$, and shift the contour from
$\operatorname{Re}s=1+1/\log(2L)$ to $\operatorname{Re}s=1-\eta$. Throughout the shifted contour the
imaginary part of the $L$-function argument lies between $|u|/2$ and $3|u|/2$. Hence
\begin{align*}
\left|\sum_{\ell\in J}
\frac{\mu(r\ell)\psi(\ell)}{\ell^{1/2+iu}}\right|
\ll_ L^{1/2}(\log X)^{O_(1)} \left(L^{-\eta}+|u|^{-1}+L^{-1}\right).
\end{align*}
Since $L\geq X^{1/\alpha}$ and $|u|\geq X^{\vartheta}$, we see that the left-hand side is at most
\begin{align*}
L^{1/2}\exp(-(\log X)^{1/10})
\end{align*}
for sufficiently large $X$.
\end{proof}

We are now ready to state and prove a quantitative combinatorial decomposition, modelled on~\cite[Lemma 4.3]{MSTT}, but with savings of up to quasipolynomial size in the Type II Archimedean conclusion.

\begin{lemma}[Combinatorial decomposition with a growing parameter]
\label{lem:growing-decomposition}
Fix $0<\eps<3/16$. There is a constant $\mathbf C_1\geq1$, depending only on $\eps$, with the following property. 

For every $A\geq1$ there are $K\geq1$, $c>0$, and $X_0\geq3$, depending only on $\eps,A$, such that the following holds. Let $X\geq X_0$ and let $I\subset[X,3X]\cap\Z$ be an interval of cardinality $H$ satisfying
\begin{equation}
 \label{eq:decomposition-interval-range}
 X^{5/8+\eps}\leq H\leq X^{1-\eps}.
\end{equation}
Let $w,R$ satisfy
\begin{align}
 \label{eq:decomposition-envelope-range}
 &2\leq w\leq\exp((\log X)^{1/10}),\qquad (\log X)^K\leq R,\qquad R^{K}\leq\exp((\log X)^{1/10}).
\end{align}
Assume one of the following two cases.
\begin{enumerate}[label=\textnormal{(\alph*)}]
 \item Define $\Lambda_w^\dagger=\Lambda_w$.
 \item Define $\Lambda_w^\dagger=\widetilde\Lambda_w$, and suppose that no level $w$ exceptional pair exists or that the exceptional conductor is at most $R$.
\end{enumerate}
Then there are an integer $1\leq N\leq R^{\mathbf C_1}$, coefficients $c_1,\ldots,c_N\in\C$ with $\sum_{\nu=1}^N|c_\nu|\leq R^{\mathbf C_1}$, arithmetic functions $f_1,\ldots,f_N$, and an error $E_I\colon I\to\C$ such that
\begin{equation}
 \label{eq:decomposition-exact}
 \Lambda(n)-\Lambda_w^\dagger(n)=\sum_{\nu=1}^Nc_\nu f_\nu(n)+E_I(n)
\end{equation}
for every $n\in I$, with
\begin{equation}
 \label{eq:decomposition-strong-error}
 \|E_I\|_{\ell^1(I)}\leq H\exp(-(\log X)^{1/10}),
\end{equation}
and for each $\nu$, at least one of the following alternatives holds.
\begin{enumerate}[label=\textnormal{(\roman*)}]
 \item The function $f_\nu$ is a $(R^{-\mathbf C_1},\mathbf C_1X^{5/8})$ Type I sum.
 \item The function $f_\nu$ is a $(R^{-\mathbf C_1},\mathbf C_1X^{7/16})$ Type $I_2$ sum.
 \item The function $f_\nu$ is a $(R^{-\mathbf C_1},\mathbf C_1^{-1}X^{3/8},\mathbf C_1X^{5/8})$ Type II sum and satisfies
 \begin{equation}
  \label{eq:decomposition-TypeII-second-moment}
  \sum_{n\in I}|f_\nu(n)|^2\leq HR^{\mathbf C_1}.
 \end{equation}
Moreover, for every interval $I_1\subset I$ of cardinality $H_1\geq HR^{-A}$, every integer $1\leq q_0\leq R^{A}$, every residue class $a\pmod{q_0}$, and every real $T$ satisfying
 \begin{equation}
  \label{eq:decomposition-frequency-range}
  R^{K}X/H\leq|T|\leq X^{A},
 \end{equation}
we have
 \begin{equation}
  \label{eq:decomposition-high-T}
  \left|\sum_{\substack{n\in I_1\\n\equiv a\pmod {q_0}}}f_\nu(n)n^{iT}\right|\leq H_1R^{-A}.
 \end{equation}
\end{enumerate}
\end{lemma}

\begin{proof}
\emph{Decomposition and coefficient bounds.} It suffices to prove this decomposition separately for $\Lambda$ and $\Lambda_w^{\dagger}$. We first deal with $\Lambda$. We apply Heath-Brown's identity in the form of \cite[Lemma~2.16, with $L=10$]{MSTT} to $\Lambda$, divide the variables dyadically, and discard components whose support does not meet $I$. After removing zero components and factors supported on $\{1\}$, each remaining convolution has the form
\begin{equation}
 \label{eq:decomposition-convolution-scales}
 \begin{aligned}
 f_\nu& =a^{(1)}*\cdots*a^{(m)},\qquad m\leq20,\\
 \operatorname{supp}(a^{(i)})&\subset(N_i,2N_i],\qquad X_f=\prod_{i=1}^mN_i,\\
 N_i&\geq1,\qquad 2^{-20}X\leq X_f\leq3X,\\
 \alpha_i& =\frac{\log N_i}{\log X_f},\qquad \alpha_S=\sum_{i\in S}\alpha_i,
 \end{aligned}
\end{equation}
where each factor is the restriction of $1$, $\log$, or $\mu$ to its dyadic interval, at most one factor is logarithmic, and $N_i\ll X^{1/10}$ whenever $a^{(i)}$ is a restriction of $\mu$.
Thus $\alpha_i\geq0$ and $\sum_i\alpha_i=1$. Applying \cite[Lemma~2.20(i),(ii)]{MSTT} with $\theta=5/8$ as in \cite[proof of Lemma~4.3]{MSTT}, each component $f_{\nu}$ is of one of the stated type I, Type $I_2$, and Type II forms.

By \cite[Lemma~4.3(i)--(iii) and its proof]{MSTT}, the number of components, their coefficient mass, and the moment and total variation estimates for the coefficients are $(\log X)^{O_\eps(1)}$. Moreover, $|f_\nu(n)|\leq\log(3X)d_{20}(n)$ for $n\in I$. Applying Shiu's bound (\cite[Lemma~2.17]{MSTT}) to $d_{20}^2$ with modulus $1$ gives
\begin{align*}
 \sum_{n\in I}|f_\nu(n)|^2\ll_\eps H(\log X)^{O_\eps(1)}.
\end{align*}
We will choose $\mathbf C_1$ after adding the model components at the end of the proof.

\emph{High frequency estimate.} We fix a Type II component $f_{\nu}$ of the form~\eqref{eq:decomposition-convolution-scales}. By \cite[Lemma~2.20(ii)]{MSTT}, as used in \cite[proof of Lemma~4.3]{MSTT},
we may choose a partition
\begin{align*}
\{1,\ldots,m\}=I_0\sqcup J\sqcup J'
\end{align*}
such that
\begin{align*}
1/4\leq\alpha_{I_0}\leq 1/2,
\qquad
|\alpha_J-\alpha_{J'}|\leq 1/4.
\end{align*}

We divide any logarithmic convolution factor by $\log(3X)$ and denote the resulting factors by $\widetilde a^{(i)}$. For each subset $S\subset\{1,\ldots,m\}$, we write $b_S$ for the convolution of the factors $\widetilde a^{(i)}$ with $i\in S$. After dividing the supports of $b_{I_0},b_J,b_{J'}$ into $O(1)$ dyadic intervals, we need to bound sums of the form
\begin{equation}
 \label{eq:decomposition-trilinear-sum}
 \sum_{\substack{m_1m_2\ell\in I_1\\m_1\sim M_1,\ m_2\sim M_2,\ \ell\sim L\\m_1m_2\ell\equiv a\pmod{q_0}}}
 b_J(m_1)m_1^{iT}b_{J'}(m_2)m_2^{iT}b_{I_0}(\ell)\ell^{iT},
\end{equation}
where
\begin{align*}
 M_1\asymp X_f^{\alpha_J},\qquad M_2\asymp X_f^{\alpha_{J'}},\qquad
 L\asymp X_f^{\alpha_{I_0}},\qquad LM_1M_2\asymp X.
\end{align*}

Since $X_f\asymp X\asymp y$, the conditions
\begin{align*}
 |\alpha_J-\alpha_{J'}|\leq1/4,\qquad \alpha_{I_0}\leq1/2
\end{align*}
verify the corresponding scale hypotheses of the Baker--Harman--Pintz type estimate of \cite[Lemma~3.5(ii)]{MSTT}. Choose $K$ sufficiently large in terms of $\eps,A$ and set $W_*=R^{K/10}$. Then
\begin{align*}
 H_1\geq y^{5/8+\eps/2},\qquad H_1W_*^4\leq y,\qquad
 W_*\leq X^{\eps/1000}.
\end{align*}

We now verify the Dirichlet polynomial hypothesis of \cite[Lemma~3.5(ii), equation~(3.8)]{MSTT}. With
\begin{align*}
 v_\ell=b_{I_0}(\ell)\ell^{iT},
\end{align*}
we must prove
\begin{equation}
 \label{eq:decomposition-DP-hypothesis}
 \left|\sum_{\ell\sim L/r}\frac{v_{r\ell}\chi(\ell)}{\ell^{1/2+it}}\right|
 \ll (L/r)^{1/2}W_*^{-1/3}
\end{equation}
for every $r\mid(a,q_0)$, every character $\chi$ modulo
$q_0/(a,q_0)$, and every $|t|\leq yW_*^4/H_1$.

Since $|I_0|\leq20$ and $\alpha_{I_0}\geq1/4$, there is $i_0\in I_0$ with $\alpha_{i_0}\geq1/80$. Expanding the convolution gives
\begin{align*}
 b_{I_0}(r\ell)=\sum_{\substack{\prod_{i\in I_0}n_i=r\ell}}\,
 \prod_{i\in I_0}\widetilde a^{(i)}(n_i).
\end{align*}
We enumerate $I_0=\{i_1,\ldots,i_m\}$, where $m\leq20$. For each tuple $(n_i)$ in this sum, we allocate the factor $r$ as in \cite[proof of Proposition~3.4(ii), equations~(3.23)--(3.24)]{MSTT}:
set
\[
 R_0=r,\qquad 
 r_{i_j}=(n_{i_j},R_{j-1}),\qquad
 R_j=R_{j-1}/r_{i_j}.
\]
Then
\[
 \prod_{i\in I_0}r_i=r,\qquad
 n_i=r_i u_i,\qquad
 (u_{i_j},R_j)=1.
\]
We group the sum according to the allocation $(r_i)_{i\in I_0}$. There are at most $r^m\leq r^{20}$ such allocations. The coprimality conditions are detected by principal characters; combined with the
original character modulo $q_0/(a,q_0)$, these give characters of modulus dividing
\[
 \frac{rq_0}{(a,q_0)}\leq q_0.
\]
Write $s=r_{i_0}$. Since $s\leq r\leq R^A=X^{o(1)}$,
\[
 N_{i_0}/s\geq X^{1/80-o(1)}\geq X^{1/100}.
\]
Moreover,
\begin{align*}
 |t|\ll (X/H)R^AW_*^4,
\end{align*}
whereas \eqref{eq:decomposition-frequency-range} gives $|T|\geq R^KX/H$. Thus, after choosing $K$ sufficiently large in
terms of $A$,
\begin{equation}
 \label{eq:decomposition-shifted-frequency}
 |T|/2\leq|T-t|\leq3|T|/2,\qquad X^{\eps/2}\leq|T-t|\leq X^{A+1}.
\end{equation}

Fixing the variables other than $u_{i_0}$ leaves an interval $J_0\subset(N_{i_0}/s,2N_{i_0}/s]$ and a sum
\begin{align*}
 \sum_{u\in J_0}
 \frac{a_0(su)\psi(u)}{u^{1/2+i(t-T)}},
\end{align*}
up to a scalar of absolute value at most one, where $a_0\in\{1,\log,\mu\}$ and $\psi$ has modulus at most $q_0$.
We apply
Lemma \ref{lem:high-frequency-DP}. In each case,
\begin{equation}
 \label{eq:decomposition-selected-factor-DP}
 \left|\sum_{u\in J_0}\frac{a_0(su)\psi(u)}{u^{1/2+i(t-T)}}\right|
 \leq (N_{i_0}/s)^{1/2}W_*^{-1}.
\end{equation}

Every remaining normalised factor satisfies
\begin{align*}
 \sum_{u\in(N_i/r_i,2N_i/r_i]}\frac{|\widetilde a^{(i)}(r_i u)|}{u^{1/2}}
 \leq3(N_i/r_i)^{1/2}.
\end{align*}
Since $L\asymp\prod_{i\in I_0}N_i$, summing over the remaining variables and the at most $r^{20}$ factor extractions gives
\begin{align*}
 \left|\sum_{\ell\sim L/r}
 \frac{v_{r\ell}\chi(\ell)}{\ell^{1/2+it}}\right|
 \ll r^{20}(L/r)^{1/2}W_*^{-1}.
\end{align*}
As $r\leq R^A$ and $W_*=R^{K/10}$, choosing $K$ sufficiently large in terms of $A$ makes the last expression at most
\[
 (L/r)^{1/2}W_*^{-1/3}.
\]
This proves \eqref{eq:decomposition-DP-hypothesis}.

We can therefore apply \cite[Lemma~3.5(ii)]{MSTT} to \eqref{eq:decomposition-trilinear-sum}. This gives
\begin{align*}
 \left|\sum_{\substack{n\in I_1\\n\equiv a\pmod {q_0}}}f_\nu(n)n^{iT}\right|
 \ll_\eps H_1R^{2A-K/30}(\log X)^{O_\eps(1)}.
\end{align*}
For $K$ sufficiently large, $\log X\leq R^{1/K}$ bounds this by $O_\eps(H_1R^{-A-1})$. Increasing $X_0$ then gives \eqref{eq:decomposition-high-T}.

\emph{Handling the model.} By Lemma \ref{lem:local-model-typeI}, we can approximate $\Lambda_w$ by $a*1$, with $a$ supported on $[1,X^{1/3}]$. In case \textnormal{(b)}, 
\begin{align*}
\widetilde\Lambda_w=a*1-a_\chi*b_\chi+E,\qquad \|E\|_{\ell^1(I)}
\leq H\exp(-(\log X)^{4/5}).
\end{align*}
The $\ell^1$ approximation errors are absorbed by \eqref{eq:decomposition-strong-error} after increasing $X_0$. Equation \eqref{eq:twisted-model-coefficient-bounds} and $q\leq R$ give
\begin{align*}
 \|b_\chi\|_{\TV(\N;q)}\leq2R,\qquad \sum_{d\leq Y}|a_\chi(d)|^2\ll Y(\log(2w))^2\leq2RY
\end{align*}
for every $Y\geq1$ and sufficiently large $X$. The coefficient $a$ satisfies the same second moment bound, so both model components are $((2R)^{-1},X^{1/3})$ Type I sums. We choose $\mathbf C_1=\mathbf C_1(\eps)$ sufficiently large to absorb the absolute constants in the three support ranges. Since $R\geq(\log X)^K$, increasing $K$ if necessary also gives
\begin{align*}
 (\log X)^{O_\eps(1)}\leq R^{\mathbf C_1},
 \qquad (\log(2w))^2\leq R,
\end{align*}
so the bounds for the number of components, coefficient mass, coefficient norms, and Type II second moment have the required form. The at most two model components are then Type I sums with parameter at least $R^{-\mathbf C_1}$. This proves \eqref{eq:decomposition-exact}--\eqref{eq:decomposition-TypeII-second-moment}.
\end{proof}

\subsection{Abelian correlations and major arcs}

The next lemma shows that a Type II correlation with a polynomial phase implies that the phase resembles $n^{iT}$ on residue classes. For large $|T|$, we use the high frequency estimate from Lemma~\ref{lem:growing-decomposition}. For small $|T|$, we remove the $n^{iT}$ factor to obtain the correlation of the original arithmetic function.

\begin{lemma}[Abelian Type II reduction]
\label{lem:abelian-typeII-reduction}
Let $d\geq1$ be an integer. There is $C_d\geq1$ with the following property.

Let $X,H\in\N$ satisfy $3\leq X$ and $2\leq H\leq X$, let $0<\delta<1/\log X$, and let $A_{II}^+\geq A_{II}^-\geq1$. Let $f$ be a $(\delta,A_{II}^-,A_{II}^+)$ Type II sum, let $P_0\colon\Z\to\R$ be a polynomial of degree at most $d$, let $\mathcal A\colon\N\to\C$ be an arithmetic function, let $P\subset(X,X+H]$ be a nonempty arithmetic progression, and let $T_0\geq1$. Suppose that
\begin{align}\label{eq:length-alternative}
 H>C_d\delta^{-C_d}\max(A_{II}^+,X/A_{II}^-)
\end{align}
and
\begin{align}
 \left|\sum_{X<n\leq X+H}f(n)\e(P_0(n))\right|^*
 &\geq\delta H,
 \label{eq:abelian-typeII-component-correlation}\\
 \left|\sum_{n\in P}\mathcal A(n)\e(P_0(n))\right|
 &\geq\delta H.
 \label{eq:abelian-typeII-original-correlation}
\end{align}
Assume moreover that
\begin{equation}
 \label{eq:abelian-typeII-high-frequency-excluded}
 \left|\sum_{X<n\leq X+H}f(n)n^{iT}\right|^*
 <\delta^{C_d}H
\end{equation}
whenever
\begin{align*}
 T_0X/H\leq|T|\leq
 C_d\delta^{-C_d}(X/H)^{d+1}.
\end{align*}
Then there is an arithmetic progression $P'\subset P$ such that
\begin{equation}
 \label{eq:abelian-typeII-major-arc-output}
 \textnormal{mod}(P')\leq C_d\delta^{-C_d}\textnormal{mod}(P),
 \qquad
 \left|\sum_{n\in P'}\mathcal A(n)\right|
 \geq\frac{\delta^{C_d}}{1+T_0}H.
\end{equation}
\end{lemma}
\begin{proof}
We write $I=(X,X+H]\cap\Z$. We apply the abelian Type II inverse theorem of \cite[Theorem~4.2(iii)]{MSTT} to the polynomial phase $\e( P_0(n))$, viewed as a nilsequence on $\R/\Z$. After increasing $C_d$, condition \eqref{eq:length-alternative} excludes the length alternative in the theorem. Thus there exist an integer $q$ and a real number $T$ such that
\begin{align}
 \label{eq:abelian-phase-approximation}
 1\leq q\leq C_d\delta^{-C_d},\qquad
 |T|&\leq C_d\delta^{-C_d}(X/H)^{d+1},\notag\\
 \|e(P_0(\cdot))(\cdot)^{-iT}\|_{\TV(I;q)}&\leq C_d\delta^{-C_d}.
\end{align}
We denote $v(n)=e(P_0(n))n^{-iT}$. Then the partial summation bound of \cite[Lemma~2.2(iii), equation~(2.5)]{MSTT} gives
\begin{align}
 \label{eq:abelian-component-variation-transfer}
 \delta H&\leq\left|\sum_{n\in I}f(n)n^{iT}v(n)\right|^*
 \leq\|v\|_{\TV(I;q)}\left|\sum_{n\in I}f(n)n^{iT}\right|^*.
\end{align}
Combining \eqref{eq:abelian-phase-approximation} and \eqref{eq:abelian-component-variation-transfer}, and increasing $C_d$, we obtain
\begin{align}
 \label{eq:abelian-archimedean-correlation}
 \left|\sum_{n\in I}f(n)n^{iT}\right|^*\geq\delta^{C_d}H.
\end{align}
Now, by \eqref{eq:abelian-typeII-high-frequency-excluded}, we must have
\begin{align*}
|T|<T_0X/H.
\end{align*}

Write $u(n)=\e(P_0(n))=v(n) n^{iT}$. For each $b\pmod q$, set
\begin{align*}
 P_b=P\cap(b+q\Z).
\end{align*}
Applying summation by parts on each nonempty $P_b$, as in
\cite[proof of Lemma~2.2(iii)]{MSTT}, gives
\begin{align*}
 \left|\sum_{n\in P_b}\mathcal A(n)u(n)\right|
\leq\|u\|_{\TV(P_b)}\max_{y\geq 1}\left|\sum_{n\in P_b\cap[1,y]}\mathcal A(n)\right|.
\end{align*}
Summing over $b\pmod q$ then yields
\begin{align}
 \label{eq:abelian-residue-summation}
 \left|\sum_{n\in P}\mathcal A(n)u(n)\right|
 &\leq\sum_{b\pmod q}\|u\|_{\TV(P_b)}
\max_{y\geq 1}\left|\sum_{n\in P_b\cap [1,y]}\mathcal A(n)\right|\notag\\
&\leq\|u\|_{\TV(P;q)}\max_{y\geq 1}\max_{b\pmod q}\left|\sum_{n\in P_b\cap [1,y]}\mathcal A(n)\right|.
\end{align}
Note that by the mean value theorem and the definition of total variation norm, for any $b\pmod q$ we have
\begin{align}\label{eq:abelian-phase-variation}\begin{aligned}
\|u\|_{\TV(P_b)}&\leq \|v\|_{\TV(P_b)}+\sum_{n,n'\in P_b \,\,\textnormal{consecutive}}|(n')^{iT}-n^{iT}|\\
&\leq \|v\|_{\TV(P_b)}+|T|\log (1+H/X).
\end{aligned}
\end{align}
Recalling that $P_b$ for $b\pmod q$ form a partition of $P$ and summing over $b\pmod q$, we obtain
\begin{align}
\label{eq}
\|u\|_{\TV(P;q)}&\leq \|v\|_{\TV(P;q)}+q|T|\log(1+H/X)\ll C_d\delta^{-C_d}(1+T_0),
\end{align}
where we used $q\leq C_d\delta^{-C_d}$, $|T|<T_0X/H$, and $\log(1+H/X)\leq H/X$.

Now, combining \eqref{eq:abelian-typeII-original-correlation}, 
\eqref{eq:abelian-residue-summation} and \eqref{eq:abelian-phase-variation}, and increasing $C_d$, we obtain a
nonempty initial segment $P'$ of some $P_b$ such that
\begin{align*}
 \left|\sum_{n\in P'}\mathcal A(n)\right|
 \geq\frac{\delta^{C_d}}{1+T_0}H.
\end{align*}
If $P'$ contains at least two elements, then $P_b$, and hence $P'$, has
modulus
\[
 \lcm(q,\textnormal{mod}(P))
 \leq q\,\textnormal{mod}(P)
 \leq C_d\delta^{-C_d}\textnormal{mod}(P).
\]
If $P'$ is a singleton we take the modulus to be $1$. This proves
\eqref{eq:abelian-typeII-major-arc-output}.
\end{proof}

We finally give an estimate for sums of $\Lambda-\widetilde\Lambda_w$ on short progressions that we will use in the next section to rule out the conclusion~\eqref{eq:abelian-typeII-major-arc-output} of Lemma~\ref{lem:abelian-typeII-reduction}. The argument is related to Huxley's proof of the prime number theorem in short intervals, but the uniformity required here in both the saving and the progression modulus  requires a modification of that proof.

\begin{proposition}[Short major arcs with the Siegel correction]
\label{prop:short-major-arcs}
Let $\eps>0$. There is $X_0=X_0(\eps)\geq3$ such that the following holds. 

Let $X\geq X_0$, let $I\subset[X,3X]\cap\Z$ be an interval of cardinality $H\geq X^{7/12+\eps}$, and let $2\leq w\leq\exp((\log X)^{1/10})$. Let $P\subset I$ be a nonempty arithmetic progression of modulus $r\leq w$. Then
\begin{equation}
 \label{eq:short-major-arc-progression-bound}
 \left|\sum_{n\in P}(\Lambda(n)-\widetilde\Lambda_w(n))\right|\leq H\exp(-(\log X)^{1/10}).
\end{equation}
If no level $w$ exceptional pair exists, or if its conductor $q_\beta$ does not divide $r$, then
\begin{equation}
 \label{eq:short-major-arc-plain-progression-bound}
 \left|\sum_{n\in P}(\Lambda(n)-\Lambda_w(n))\right|\leq H\exp(-(\log X)^{1/10}).
\end{equation}
\end{proposition}

\begin{proof}
It suffices to consider $0<\eps\leq1/12$. We set
\begin{align*}
 y=\min P-1,\qquad Y=\operatorname{span}(P), \qquad J=y+[Y].
\end{align*}
Then $J\subset I$, $Y\leq H$, and $P=J\cap (a+r\Z)$ for some $a$. If $Y<X^{7/12+\eps/2}$, the pointwise bounds $\Lambda(n)\leq\log(3X)$, $\Lambda_w(n)\ll\log(2w)$ and $|\widetilde\Lambda_w(n)|\ll\log(2w)$ give
\begin{align}
 \label{eq:major-arc-short-hull}
\sum_{n\in P}\bigl(\Lambda(n)+\Lambda_w(n)+|\widetilde\Lambda_w(n)|\bigr)\ll  Y\log X\leq HX^{-\eps/2}\log X\ll_\eps H\exp(-2(\log X)^{1/10}).
\end{align}
This proves~\eqref{eq:short-major-arc-progression-bound} and~\eqref{eq:short-major-arc-plain-progression-bound} in that case. We may henceforth assume $Y\geq X^{7/12+\eps/2}$. 

If $(a,r)>1$, there is a prime $p\mid(a,r)$, and since $p\leq r\leq w$, both models vanish on $P$, and only powers of $p$ can contribute to the von Mangoldt sum. Their total contribution is $O((\log X)^2)$, which is certainly $O_\eps(Y\exp(-2(\log X)^{1/10}))$. We may therefore also assume $(a,r)=1$. If the level $w$ exceptional pair $(\beta,\chi)$ exists and $q_{\beta}$ is the exceptional conductor, we write $d=q_\beta$; otherwise we set $d=1$ and remove all terms involving $\chi$ and $\beta$. We will snow estimate the sums of $\Lambda$ and the model functions separately and show that the exceptional term in the von Mangoldt sum agrees with the correction in $\widetilde\Lambda_w$, and that this correction term vanishes when $d\nmid r$.

\emph{The von Mangoldt sum.} We set
\begin{align*}
 T=X^{5/12-\eps/4},\qquad r_*=\max\{r,2\}.
\end{align*}
Since $r_*=X^{o(1)}$, for all sufficiently large $X$, we have
\begin{equation}
 \label{eq:major-arc-density-range}
 (2r_*T)^{12/5+\eps/2}\leq X^{1-\eps/4},\qquad
 X/T\leq YX^{-\eps/4}.
\end{equation}

For each character $\psi\pmod r$, we write $\psi^*$ for the primitive character inducing it. Let $N_r(\sigma,V)$ be the number of nontrivial zeros of 
$$\prod_{q\mid r}\,\,\,\prod_{\psi^*\pmod q\,\,\textnormal{primitive}}L(s,\psi^*),$$ with multiplicity, having real part at least $\sigma$ and imaginary part of absolute value at most $V$. The log-free zero density estimate (see e.g. \cite[Theorem~2.1, equation~(2.1)]{TZ}) gives
\begin{equation}
 \label{eq:major-arc-zero-density}
 N_r(\sigma,V)\ll_\eps(r_*V)^{(12/5+\eps/2)(1-\sigma)}
 \quad\text{for }0\leq\sigma\leq1\text{ and }V\geq1.
\end{equation}
Since $\log r\leq \log w\leq (\log X)^{1/10}$, the Vinogradov--Korobov zero-free region gives for any zero $\rho$ of height at mist $2T$ the bound
\begin{align}
 \label{eq:major-arc-zero-free-width}
 \operatorname{Re}\rho\leq1-\eta,\qquad
 \eta\asymp(\log X)^{-2/3}(\log\log X)^{-1/3},
\end{align}
apart from a possible exceptional real zero. Every primitive conductor in the family divides $r$ and is at most $w$. A real zero not selected by the level $w$ convention satisfies $1-\operatorname{Re}\rho>c_0/\log w>\eta$ for sufficiently large $X$. Thus we need to remove only the selected exceptional zero. When the exceptional pair exists, the selected zero occurs in the family exactly when $d\mid r$.

We will use an asterisk on a sum over zeros to mean that the exceptional zero is omitted. By partial summation, with \eqref{eq:major-arc-density-range}--\eqref{eq:major-arc-zero-free-width}, we have
\begin{align}
 \sum_{\psi\bmod r}\ \asum_{\substack{\rho\colon L(\rho,\psi^*)=0\\|\operatorname{Im}\rho|\leq2T}}X^{\operatorname{Re}\rho-1}
 &\ll_\eps X^{-\eps/4}+(\log X)\int_0^{1-\eta}X^{-\eps(1-\sigma)/4}\,\mathrm d\sigma\notag\\
 &\ll_\eps X^{-\eps\eta/4}\ll_\eps \exp(-2(\log X)^{1/10}),
 \label{eq:major-arc-zero-sum}
\end{align}
where we used $\eta\log X\gg_\eps(\log X)^{1/3}(\log\log X)^{-1/3}$ for the last inequality.

We now apply character orthogonality in the form
\begin{align}\label{eq:orthogonal}
 \sum_{n\in P}\Lambda(n)
 =\frac1{\varphi(r)}\sum_{\psi\bmod r}\overline{\psi(a)}
       \sum_{n\in J}\Lambda(n)\psi(n).
\end{align}
We evaluate each inner sum using the truncated explicit formula \cite[Chapter~17, (9)--(11), and Chapter~19, (7)--(8)]{Davenport}. For each primitive character, we choose a height $T_{\psi^{*}}$ in $[T,2T]$ separated from its zero ordinates by at least $1/(\log T)^2$ and subtract the explicit formulas at $y+Y$ and $y$. Thus, writing $\psi^*$ for the primitive character inducing $\psi$, we have
\begin{align}\label{eq:explicit-formula}
 \sum_{y<n\leq y+Y}\Lambda(n)\psi^*(n)
 =\1_{\psi^*=1}Y
 -\sum_{\substack{\rho\colon L(\rho,\psi^*)=0\\|\operatorname{Im}\rho|\leq2T_{\psi^{*}}}}
   \frac{(y+Y)^\rho-y^\rho}{\rho}
 +O\bigl((X/T+X^{1/2})(\log X)^2\bigr).
\end{align}
Each nonexceptional zero $\rho$ contributes a term bounded by
\begin{align}
 \label{eq:major-arc-zero-integral}
 \left|\frac{(y+Y)^\rho-y^\rho}{\rho}\right|
 =\left|\int_y^{y+Y}t^{\rho-1}\,\mathrm dt\right|
 \ll YX^{\operatorname{Re}\rho-1}.
\end{align}
Passing from a primitive character to the induced character changes only prime powers with prime divisor of $r$, of total weight $O((\log X)^2)$ for each character. After averaging over the characters modulo $r$, the total contribution of the error term in~\eqref{eq:explicit-formula}, the nonexceptional zeros, and these prime powers is contribute at most
\begin{align}
 \label{eq:major-arc-averaged-remainder}
 &\frac{Y}{\varphi(r)}\sum_{\psi\bmod r}\,\,
    \asum_{\substack{\rho\colon L(\rho,\psi^*)=0\\|\operatorname{Im}\rho|\leq2T}}
         X^{\operatorname{Re}\rho-1}
   +(X/T+X^{1/2})(\log X)^2
 \ll_\eps Y\exp(-2(\log X)^{1/10}),
\end{align}
by \eqref{eq:major-arc-density-range} and \eqref{eq:major-arc-zero-sum}. The principal character gives the main term $Y/\varphi(r)$ to~\eqref{eq:orthogonal}. When the exceptional pair exists and $d\mid r$, the character induced by $\chi$ gives to~\eqref{eq:orthogonal} the exceptional term 
\begin{align*}
 -\frac{\chi(a)}{\varphi(r)}\frac{(y+Y)^\beta-y^\beta}{\beta}=-\frac{\chi(a)}{\varphi(r)} \int_y^{y+Y}t^{\beta-1}\,\mathrm dt.
\end{align*}
Hence
\begin{equation}
 \label{eq:major-arc-Lambda-explicit}
 \sum_{n\in P}\Lambda(n)=\frac Y{\varphi(r)}
 -\frac{\1_{d\mid r}\chi(a)}{\varphi(r)}
       \int_y^{y+Y}t^{\beta-1}\,\mathrm dt+O_\eps(Y\exp(-2(\log X)^{1/10})).
\end{equation}

\emph{The model sums.} Recall that $P=J\cap (a+r\Z)$. For estimating the sums of the model functions, we split into residue classes modulo $\ell=\lcm(r,d)$ so that, when the exceptional pair exists, $\chi$ is constant on each residue class being sieved. We set the sifting level to be $D=\lfloor X^{1/4}\rfloor$. Every prime factor of $\ell$ is at most $w$, and $\ell\leq rw\leq w^2$. For any $u$ corpime to $\ell$, any positive integer $k$ coprime to $\ell$, and $0\leq t\leq Y$, the Chinese remainder theorem gives
\begin{align}
 \label{eq:major-arc-sieve-remainder}
 \#\{y<n\leq y+t:n\equiv u\pmod\ell,\ k\mid n\}
 =\frac{t}{\ell k}+O(1).
\end{align}
We apply the fundamental lemma of sieve theory \cite[Lemma~6.3]{IK} to the primes $p\leq w$ not dividing $\ell$, with level $D$, as in \cite[proof of Proposition~4.5]{MTW}. Since $\log D/\log w\gg(\log X)^{9/10}$, equation \eqref{eq:major-arc-sieve-remainder} gives,  for $0\leq t\leq Y$,
\begin{align}
 \label{eq:major-arc-model-unit-class}
 \sum_{\substack{y< n\leq y+t\\ n\equiv u\pmod\ell}}\Lambda_w(n)
 =\frac{P(w)}{\varphi(P(w))}\prod_{\substack{p\leq w\\p\nmid\ell}}\left(1-\frac1p\right)\frac{t}{\ell}+O\left(\frac{t\exp(-2(\log X)^{1/10})}{\varphi(\ell)}+D\log(2w)\right).
\end{align}
The main term simplifies to $t/\varphi(\ell)$.
Summing the error terms over all coprime residue classes $u\pmod{\ell}$, with coefficients of absolute value at most one, gives a contribution of $O(Y\exp(-2(\log X)^{1/10})+D\ell\log(2w))=O_\eps(Y\exp(-2(\log X)^{1/10}))$, since $D\ell=X^{1/4+o(1)}$.

Since $r\mid \ell$ and $(a,r)=1$, we can write the residue class $a\pmod r$ defining $P$ as a union of $\varphi(\ell)/\varphi(r)$ reduced residue classes modulo $\ell$. When the exceptional pair exists, primitivity of $\chi$ and character orthogonality imply
\begin{align}
 \label{eq:major-arc-exceptional-fibre}
 \sum_{\substack{u\pmod\ell\\u\equiv a\pmod r\\(u,\ell)=1}}\chi(u)
 =\1_{d\mid r}\chi(a)\frac{\varphi(\ell)}{\varphi(r)}.
\end{align}
We sum \eqref{eq:major-arc-model-unit-class} over the residue classes $u\pmod \ell$ corresponding to $a\pmod r$, first with coefficient $1$ and then, when the exceptional pair exists, with coefficient $\chi(u)$. This gives
\begin{align}
 \sum_{\substack{y<n\leq y+t\\n\equiv a\pmod r}}\Lambda_w(n)
 &=\frac{t}{\varphi(r)}+O_\eps(Y\exp(-2(\log X)^{1/10})),\notag\\
 \sum_{\substack{y<n\leq y+t\\n\equiv a\pmod r}}\Lambda_w(n)\chi(n)
 &=\frac{\1_{d\mid r}\chi(a)t}{\varphi(r)}+O_\eps(Y\exp(-2(\log X)^{1/10})),
 \label{eq:major-arc-model-initial-sums}
\end{align}
for $0\leq t\leq Y$. Using partial summation, we can add the weight $n^{\beta-1}$ in the second sum, while keeping the same error term.
Thus, taking $t=Y$, we obtain
\begin{align}
 \sum_{n\in P}\Lambda_w(n)
 &=\frac Y{\varphi(r)}+O_\eps(Y\exp(-2(\log X)^{1/10})),\notag\\
\sum_{n\in P}\Lambda_w(n)\chi(n)n^{\beta-1}&=\frac{\1_{d\mid r}\chi(a)}{\varphi(r)}\int_y^{y+Y}t^{\beta-1}\,\mathrm dt+O_\eps(Y\exp(-2(\log X)^{1/10})).
 \label{eq:model-local-exceptional}
\end{align}
Comparing \eqref{eq:model-local-exceptional} with \eqref{eq:major-arc-Lambda-explicit}, and using the definition of $\widetilde\Lambda_w$, we conclude that
\begin{align}
 \label{eq:major-arc-residue-cancellation}
 \sum_{n\in P}\bigl(\Lambda(n)-\widetilde\Lambda_w(n)\bigr)
 =O_\eps(Y\exp(-2(\log X)^{1/10})).
\end{align}
If no exceptional pair exists, or if $d\nmid r$, the exceptional main term in~\eqref{eq:model-local-exceptional} vanishes, so the same estimate holds with $\widetilde\Lambda_w$ replaced by $\Lambda_w$.

We have now shown in any case that the sums~\eqref{eq:short-major-arc-progression-bound} and~\eqref{eq:short-major-arc-plain-progression-bound} are $O_\eps(H\exp(-2(\log X)^{1/10}))$. Enlarging $X_0(\eps)$ makes this bound $\leq H\exp(-(\log X)^{1/10})$, proving \eqref{eq:short-major-arc-progression-bound} and \eqref{eq:short-major-arc-plain-progression-bound}.
\end{proof}

\section{Step descent and proofs of the main results}\label{sec:step-descent}

We first show in Proposition~\ref{prop:step-descent} that if the von Mangoldt minus its model correlates with a nilsequence with suitable quantitative bounds, then we can obtain a similar correlation bound on a nilmanifold of smaller step. We then use induction on this conclusion in Proposition~\ref{prop:large-cutoff-uniformity}.

\begin{proposition}[Step descent from nilsequence correlation]
\label{prop:step-descent}
Let $k\geq2$ be an integer and let $0<\eps<3/16$. There are constants $C\geq1$ and $X_0\geq3$, depending only on $k,\eps$, with the following property.

Let $X\geq X_0$, let $I=x+[H]\subset[X,3X]\cap\Z$, where $x,H\in\Z$ and
\[
 X^{5/8+\eps}\leq H\leq X^{1-\eps},
\]
and let $2\leq w\leq\exp((\log X)^{1/10})$. Let $D\geq2$ be an integer and
\[
 R\geq\max\{D,2,(\log X)^C\},\qquad
 Q=R^{(D+2)^C},
 \qquad Q^C\leq w.
\]
Let $G/\Gamma$ be an $s$-step filtered nilmanifold of degree at most $k-1$, dimension at most $D$ and complexity at most $R$, where $1\leq s\leq k-1$. Let $g\colon\Z\to G$ be polynomial and let $F\colon G/\Gamma\to\C$ be $1$-bounded and $R$-Lipschitz.

Choose $\Lambda_w^\dagger$ in one of the following two ways:
\begin{enumerate}[label=\textnormal{(\alph*)}]
 \item $\Lambda_w^\dagger=\widetilde\Lambda_w$, and the level $w$ exceptional conductor, if it exists, is at most $R$;
 \item $\Lambda_w^\dagger=\Lambda_w$, and the level $w$ exceptional conductor, if it exists, is larger than $Q$.
\end{enumerate}
If
\begin{align}\label{eq:maximalF}
 \left|\sum_{n\in I}(\Lambda-\Lambda_w^\dagger)(n)F(g(n)\Gamma)\right|^*
 \geq H/R,
\end{align}
then there are a filtered nilmanifold $G'/\Gamma'$ of degree at most $k-1$, dimension at most $D$, complexity at most $Q$ and step at most $s-1$, a polynomial sequence $g'\colon\Z\to G'$, and a $1$-bounded $Q$-Lipschitz function $F'\colon G'/\Gamma'\to\C$ such that
\[
 \left|\sum_{n\in I}(\Lambda-\Lambda_w^\dagger)(n)F'(g'(n)\Gamma')\right|^*\geq H/Q.
\]
\end{proposition}

\begin{proof}
All constants in this proof depend only on $k,\eps$. We choose $C$ sufficiently large in terms of these constants and increase $X_0$ when needed.

\emph{Vertical reduction.} Choose a progression $P\subset I$ attaining the maximum norm in~\eqref{eq:maximalF}. The estimate~\eqref{eq:Mertens} gives, for some absolute constant $C_0$, the bound 
\begin{align}\label{eq:coefficient-mass}
 |(\Lambda-\Lambda_w^\dagger)(n)|\leq C_0\log X\leq R\qquad\text{for every }n\in I.
\end{align}
We apply Lemma \ref{lem:vertical-reduction} on $I$ to $(\Lambda-\Lambda_w^\dagger)\1_P/(C_0\log X)$, with correlation parameter $1/(RC_0\log X)\geq 1/R^2$. In the zero frequency case, we obtain a nilsequence of step at most $s-1$, complexity and Lipschitz norm at most $S$, and correlation at least $H/S$ on $P$, where 
$$S=R^{(D+2)^{c_1}},$$ 
where $c_1=c_1(k)$ is sufficiently large.  Since $S\leq Q$, this proves the proposition in that case.

In the nonzero frequency case, we obtain a quotient nilmanifold $\widehat G/\widehat\Gamma$ of step $s$, degree at most $k-1$ and dimension at most $D$, with $(\widehat G)^{(s)}$ one-dimensional. We also obtain
a polynomial $\widehat g$ and a $1$-bounded vertical nilcharacter $\widehat F$ with nonzero frequency $\xi$, with all complexity, Lipschitz and frequency bounds at most $S$, and
\begin{equation}
 \label{eq:descent-selected-vertical-correlation}
 \left|\sum_{n\in P}(\Lambda-\Lambda_w^\dagger)(n)\widehat F(\widehat g(n)\widehat\Gamma)\right|\geq H/S.
\end{equation}

\emph{Arithmetic decomposition.} 
Let $\tilde{c}$ be the exponent $\mathbf C_1$ in Lemma \ref{lem:growing-decomposition}. We choose $a=a(k,\eps)$ sufficiently large (depending on $\tilde{c}$ in particular) and set $\delta=S^{-a}$. We then choose $A_0=A_0(k,\eps,a)$ sufficiently large. Since
\begin{align*}
 R\geq(\log X)^C,\qquad
 \log S\leq\frac{(\log X)^{1/10}}{C},
\end{align*}
choosing $C$ sufficiently large makes the Lemma \ref{lem:growing-decomposition} applicable with $A=A_0$ and growing parameter $S$.
In case~\textnormal{(a)} of the current statement, we use Lemma \ref{lem:growing-decomposition} to decompose $\Lambda-\widetilde\Lambda_w$, and the exceptional conductor, if it exists, is at most $R\leq S$. In case~\textnormal{(b)}, we use the lemma to decompose $\Lambda-\Lambda_w$, for which the lemma requires no condition on the exceptional conductor. In either case we obtain a decomposition
\begin{align*}
 (\Lambda-\Lambda_w^\dagger)(n)=\sum_\nu c_\nu f_\nu(n)+E_I(n)
 \qquad\text{for every }n\in I,
\end{align*}
with
\begin{align}\label{eq:coefficient-mass2}
  \sum_\nu|c_\nu|\leq S^{\tilde c},\qquad
 \|E_I\|_{\ell^1(I)}
 \leq H\exp\bigl(-c(\log X)^{1/10}\bigr)
 \leq \frac{H}{2S},
\end{align}
where the last inequality follows after increasing $C$. Moreover, $f_{\nu}$ is either a $(S^{-\tilde{c}},\tilde{c}X^{5/8})$ Type $I$ sum or a $(S^{-\tilde{c}},\tilde{c}X^{7/16})$ Type $I_2$ sum or a $(S^{-\tilde{c}},\tilde{c}^{-1}X^{3/8},\tilde{c}X^{5/8})$ Type $II$ sum, and for the Type II components we also have
\begin{align}\label{eq:coefficient-mass3}
\sup_{\nu}\sum_{n\in I}|f_{\nu}(n)|^2\leq HS^{\tilde{c}}.
\end{align}
Using the triangle inequality and the bound~\eqref{eq:coefficient-mass2}, applied to \eqref{eq:descent-selected-vertical-correlation}, we produce a component $f_\nu$ satisfying
\begin{equation}
 \label{eq:dichotomy-selected-component}
 \left|\sum_{n\in I}f_\nu(n)\widehat F(\widehat g(n)\widehat\Gamma)\right|^*\geq\frac{H}{2S^{\tilde{c}+1}}.
\end{equation}

\emph{Factorisation in the Type I, Type $I_2$ and non-abelian Type II cases.} Note that by assumption we have 
\begin{equation}
 \label{eq:dichotomy-analytic-scale-comparison}
 X/2\leq x\leq3X,\qquad H\leq x^{1-\eps/2},
\end{equation}
Set $B=S^{(D+2)^{c_2}}$, where $c_2=c_2(k,\eps)$ is to be chosen sufficiently large.  Then we have $B\leq Q\leq x^{\eps/4000}$ for sufficiently large $X$. 

Since $\delta=S^{-a}\leq 1/(2S^{\tilde{c}+1})$ by the assumption that $a$ is large in terms of $\tilde{c}$, the right-hand side of~\eqref{eq:dichotomy-selected-component} is at least $\delta H$, and since $\delta\leq S^{-\tilde{c}}$, the component $f_{\nu}$ is of Type I, Type $I_2$ or Type II with parameter $\delta$ in Definition~\ref{def:type-sums}.

If $f_\nu$ is Type I, we apply Proposition \ref{prop:typeI-standalone} to $f_{\nu}$ and $\widehat F(\widehat g(\cdot)\widehat\Gamma)$, with parameters $\delta, S$. If $f_\nu$ is Type $I_2$, we apply Proposition \ref{prop:typeI2-standalone} with the same parameters. If $f_\nu$ is Type II and $\widehat G$ is non-abelian, we restrict the first convolution factor in $f_{\nu}$ to a dyadic interval $(V,2V]\subset [\frac{1}{2}\tilde{c}^{-1}X^{3/8},2\tilde{c}X^{5/8}]$ while retaining a correlation of size at least $\delta H/(2\log X)\geq \delta^2 H$, and then apply Proposition \ref{cor:efficient-typeII-inverse} with correlation parameter $\delta^2$, with the frequency hypothesis for the vertical character following from $|\xi|\leq S$. In all three cases, the inverse parameter given by these propositions is at most
\[
 S^{(D+2)^{O_{k,\eps}(1)}}\leq B
\]
after choosing $c_2$ sufficiently large. In each of the three cases, the length alternative
$$H\ll_{k,\eps}BX^{5/8}$$
is excluded since $H\geq X^{5/8+\varepsilon}$.
$B\leq Q\leq X^{\varepsilon/10}$ for sufficiently large $X$. Thus each of these cases gives a $B$-factorisation
\begin{equation}
 \label{eq:descent-factorization}
 \widehat g=\varepsilon g_*\gamma
\end{equation}
on $I$ through a connected $B$-rational subgroup $G_*\leq\widehat G$ of step at most $s-1$, with induced nilmanifold complexity and period of $\gamma(\cdot)\widehat\Gamma$ at most $B$.

\emph{Excluding the abelian Type II case.} Suppose instead that $f_\nu$ is a $(S^{-\tilde{c}},\tilde{c}^{-1}X^{3/8},\tilde{c}X^{5/8})$ Type II sum and that $\widehat G$ is abelian. Then $s=1$ and $\widehat G/\widehat\Gamma$ is a one-dimensional torus, so
\begin{align*}
 \widehat F(\widehat g(n)\widehat\Gamma)=b_0\e(P_0(n)),
\end{align*}
where $P_0$ is a real polynomial of degree at most $k-1$ and $0<|b_0|\leq1$. Dividing by $b_0$ preserves the lower bounds in \eqref{eq:descent-selected-vertical-correlation} and \eqref{eq:dichotomy-selected-component}. Since $|\Lambda-\Lambda_w^\dagger|\leq R\leq S$, we have $|P|\geq H/S^2\geq2$ and, in the notation~\eqref{eq:modP} we have
\begin{align*}
 \textnormal{mod}(P)\leq\frac{H-1}{|P|-1}\leq2S^2.
\end{align*}
We will apply the abelian Type II reduction of Lemma \ref{lem:abelian-typeII-reduction} at scale $x$, with parameter $\delta$, function $f_\nu$, progression $P$ and $T_0=2S^K$, with $K$ sufficiently large. Then both hypotheses~\eqref{eq:abelian-typeII-component-correlation} and~\eqref{eq:abelian-typeII-original-correlation} hold. Let $b=C_{k-1}$ be the constant in
Lemma \ref{lem:abelian-typeII-reduction}.
The parameter and length hypotheses follow from $\delta=S^{-a}<1/\log x$ and
\begin{align*}
 b\delta^{-b}\max\{\tilde cX^{5/8},\tilde c xX^{-3/8}\}
 \ll_{k,\eps} S^{ab}X^{5/8}\ll HX^{-\varepsilon/10},
\end{align*}
say.

To verify~\eqref{eq:abelian-typeII-high-frequency-excluded}, suppose that there are a progression $P'\subset I$ and a real $T$ with
\begin{align}\label{eq:Tbounds}
 T_0x/H\leq|T|\leq b\delta^{-b}(x/H)^k,
\end{align}
such that
\begin{equation}
 \label{eq:dichotomy-high-branch-lower}
 \left|\sum_{n\in P'}f_\nu(n)n^{iT}\right|\geq HS^{-ab}.
\end{equation}
Recalling the notation~\eqref{eq:modP}, if $\operatorname{span}(P')<H/S^{A_0}$ or $\textnormal{mod}(P')>S^{A_0}$, then $|P'|\leq2H/S^{A_0}$. The second moment bound \eqref{eq:coefficient-mass3} and Cauchy--Schwarz give
\begin{align*}
 \sum_{n\in P'}|f_\nu(n)|\leq\sqrt2H S^{(\tilde{c}-A_0)/2}<HS^{-ab},
\end{align*}
contradicting \eqref{eq:dichotomy-high-branch-lower}. Otherwise, the final claim of Lemma~\ref{lem:growing-decomposition} is applicable to the progression $P'$, noting that by~\eqref{eq:Tbounds} we have
\begin{align*}
 |T|\geq2S^Kx/H\geq S^KX/H,\qquad |T|\leq  X^{A_0}.
\end{align*}
Thus \eqref{eq:decomposition-high-T} bounds~\eqref{eq:dichotomy-high-branch-lower}  by $HS^{-A_0}<HS^{-ab}$, a contradiction.

We have now verified~\eqref{eq:abelian-typeII-high-frequency-excluded}. Lemma~\ref{lem:abelian-typeII-reduction} now produces a progression $P_1\subset P$ such that
\begin{equation}
 \label{eq:dichotomy-major-branch-lower}
 \textnormal{mod}(P_1)
 \leq b\delta^{-b}\textnormal{mod}(P)\leq Q,\qquad
 \left|\sum_{n\in P_1}(\Lambda-\Lambda_w^\dagger)(n)\right|
 \geq\frac{HS^{-ab}}{1+2S^K}
 \geq HS^{-ab-K-2}.
\end{equation}
We have
\[
 \textnormal{mod}(P_1)\leq Q,\qquad
 Q\leq w^{1/C}\leq\exp((\log X)^{1/10}/C),\qquad
 S^{ab+K+3}\leq Q,
\]
on choosing $C$ large enough.
We may therefore apply the major arc Proposition \ref{prop:short-major-arcs} on $I$. In case~\textnormal{(a)}, we use \eqref{eq:short-major-arc-progression-bound}. In case~\textnormal{(b)}, the exceptional conductor, if it exists, is larger than \(Q\geq\textnormal{mod}(P_1)\), so it cannot divide \(\textnormal{mod}(P_1)\), and we use \eqref{eq:short-major-arc-plain-progression-bound}. In either case, we obtain
\begin{align*}
 \left|\sum_{n\in P_1}(\Lambda-\Lambda_w^\dagger)(n)\right|
 \leq H\exp(-(\log X)^{1/10})
 \ll_{k,\eps}HS^{-ab-K-3},
\end{align*}
which contradicts \eqref{eq:dichotomy-major-branch-lower} for sufficiently large $X$. Hence the abelian Type II case cannot occur, so the factorisation conclusion~\eqref{eq:descent-factorization} always holds.

\emph{A lower-step nilsequence.} Now we have the factorisation \eqref{eq:descent-factorization} and the correlation \eqref{eq:descent-selected-vertical-correlation} on the original progression $P$. Since $B\geq S$, we can replace the lower bound in \eqref{eq:descent-selected-vertical-correlation} by $H/B$. From $B\leq Q\leq X^{\varepsilon/10}$ we have $H\geq100B^4$ for sufficiently large $X$. We partition $I$ into $O(B^4)$ consecutive intervals of cardinality at most $H/(100B^4)$, and intersect them with $P$ and the residue classes modulo the period $t\leq B$ of $\gamma(\cdot)\widehat\Gamma$. This partitions $P$ into $O(B^5)$ progressions, on each of which $\gamma(n)\widehat\Gamma$ is constant.

For each such progression $\tilde P$, we choose $n_{\tilde P}\in\tilde P$. By smoothness of $\varepsilon$, we have
\begin{align*}
 d_{\widehat G}(\varepsilon(n),\varepsilon(n_{\tilde P}))
 \leq \frac{B|n-n_{\tilde P}|}{H}
 \leq \frac{1}{100B^3}
 \qquad\text{for every }n\in\tilde P.
\end{align*}
Since $\gamma(n)\widehat\Gamma=\gamma(n_{\tilde P})\widehat\Gamma$ on $\tilde P$, right invariance of the group metric and the Lipschitz bound for $\widehat F$ give
\begin{align*}
 \Bigl|\widehat F(\varepsilon(n)g_*(n)\gamma(n)\widehat\Gamma)
       -\widehat F(\varepsilon(n_{\tilde P})g_*(n)\gamma(n_{\tilde P})\widehat\Gamma)\Bigr|\leq B\,d_{\widehat G}(\varepsilon(n),\varepsilon(n_{\tilde P}))
 \leq\frac{1}{100B^2}
\end{align*}
for every $n\in \tilde{P}$.
Using $|\Lambda-\Lambda_w^\dagger|\leq B$, we bound the total change in the weighted correlation by
\begin{align*}
 \frac{1}{100B^2}\sum_{n\in P}|(\Lambda-\Lambda_w^\dagger)(n)|
 \leq\frac{H}{100B}.
\end{align*}
The original correlation in \eqref{eq:descent-selected-vertical-correlation} is at least $H/B$, so the triangle inequality and the bound $O(B^5)$ for the number of progressions give a progression $\tilde P'\subset P$ such that
\begin{equation}
 \label{eq:descent-selected-cell-correlation}
 \left|\sum_{n\in\tilde P'}(\Lambda-\Lambda_w^\dagger)(n)
 \widehat F(\varepsilon(n_0)g_*(n)\gamma_0\widehat\Gamma)\right|
 \gg B^{-5}\frac{H}{B}=\frac{H}{B^6},
\end{equation}
where $n_0\in\tilde P'$ and
\begin{align*}
 \gamma(n)\widehat\Gamma=\gamma_0\widehat\Gamma
 \qquad\text{for every }n\in\tilde P'.
\end{align*}

By \cite[Lemma B.6]{Leng}, we can choose $\gamma_0$ in this coset with Mal'cev coordinates bounded in absolute value by $1$. The element $\gamma(n_0)$ is $B$-rational, and $\gamma_0$ differs from $\gamma(n_0)$ by right multiplication by an element of $\widehat\Gamma$. Thus \cite[Lemma B.13(iii)]{Leng} gives rationality at most $B^{(D+2)^{O_k(1)}}$ for $\gamma_0$.

Define
\begin{align*}
 G'=G_*,
 \qquad \Gamma'=G_*\cap\gamma_0\widehat\Gamma\gamma_0^{-1},
 \qquad g'=g_*,\qquad  F'(h\Gamma')=\widehat F(\varepsilon(n_0)h\gamma_0\widehat\Gamma).
\end{align*}
For $\gamma'\in\Gamma'$, we have
\begin{align*}
 h\gamma'\gamma_0\widehat\Gamma=h\gamma_0\widehat\Gamma,
\end{align*}
so $F'$ is well defined.

We first bound the complexity of $G'/\Gamma'$. The rational conjugation result \cite[Lemma B.15]{Leng} shows that
\begin{align*}
 \gamma_0^{-1}G_*\gamma_0
\end{align*}
is $B^{(D+2)^{O_k(1)}}$-rational. We apply \cite[Lemma B.12]{Leng} to obtain an adapted Mal'cev basis for
\begin{align}\label{eq:conjugate-subnilmanifold}
 (\gamma_0^{-1}G_*\gamma_0)/
 ((\gamma_0^{-1}G_*\gamma_0)\cap\widehat\Gamma).
\end{align}
Conjugating this basis by $\gamma_0$ gives a basis for $G'/\Gamma'$. Since each $\widehat G_i$ is normal in $\widehat G$, the basis is adapted to the filtration $(G'_i)_{i\geq 0}=(G_*\cap\widehat G_i)_{i\geq 0}$. This implies that $G'/\Gamma'$ has complexity $\leq B^{(D+2)^{c_3}}$
for a sufficiently large constant $c_3=c_3(k)$.

To bound the Lipschitz norm, we write
\begin{align*}
 F'(h\Gamma')
 =\widehat F\bigl(\varepsilon(n_0)\gamma_0
   (\gamma_0^{-1}h\gamma_0)\widehat\Gamma\bigr).
\end{align*}
By the bound $d_{\widehat G}(\varepsilon(n_0),1_{\widehat G})\leq B$, the bounded coordinates of $\gamma_0$, and \cite[Lemma B.3]{Leng}, we see that the Mal'cev coordinates of $\varepsilon(n_0)\gamma_0$ are bounded by $B^{(D+2)^{O_k(1)}}$. The left-translation estimate of \cite[Lemma B.4]{Leng} therefore bounds the Lipschitz norm of
\begin{align*}
 z\widehat\Gamma\longmapsto
 \widehat F(\varepsilon(n_0)\gamma_0z\widehat\Gamma)
\end{align*}
by $B^{(D+2)^{O_k(1)}}$. The metric comparison \cite[Lemma B.9]{Leng} gives the same Lipschitz bound on the conjugate subnilmanifold~\eqref{eq:conjugate-subnilmanifold}. Conjugating back by $\gamma_0$, with the corresponding adapted Mal'cev basis on $G'/\Gamma'$, therefore gives
\[
 \|F'\|_{\Lip}\leq B^{(D+2)^{c_3}},
\]
after increasing $c_3$.

The induced filtration has degree at most $k-1$, the group $G'$ has dimension at most $D$ and step at most $s-1$, and $F'$ is $1$-bounded. Since $B=R^{(D+2)^{c_1+c_2}}$, choosing $C$ sufficiently large absorbs these bounds and the loss $O(B^6)$ into $Q$. Equation \eqref{eq:descent-selected-cell-correlation}, with $\tilde P'\subset I$, therefore gives
\begin{align*}
 \left|\sum_{n\in I}(\Lambda-\Lambda_w^\dagger)(n)F'(g'(n)\Gamma')\right|^*
 \geq H/Q.
\end{align*}
\end{proof}

We are now in a position to prove Theorem~\ref{thm:corrected-main} with two additional restrictions, namely $w=\exp((\log X)^{1/10})$ and an upper bound for $H$. 

\begin{proposition}[Uniformity for the maximal cutoff]
\label{prop:large-cutoff-uniformity}
Fix an integer $k\geq2$ and $0<\eps<3/16$. There is a constant $c=c(k,\eps)>0$ with the following property. Let $X\geq3$, let $H\geq1$, and let $w\geq2$. Suppose that
\begin{equation}
 \label{eq:large-cutoff-H-range}
 X^{5/8+\eps}\leq H\leq X^{1-\eps}
\end{equation}
and
\begin{equation}
 \label{eq:large-cutoff-w-range}
 w=\exp\bigl((\log X)^{1/10}\bigr).
\end{equation}
Then
\begin{equation}
 \label{eq:large-cutoff-uniformity}
 \|\Lambda-\widetilde\Lambda_w\|_{U^k(X,X+H]}\ll_{k,\eps}\exp\bigl(- (\log X)^c\bigr).
\end{equation}
\end{proposition}

\begin{proof}
We may assume that $X$ is sufficiently large in terms of $k,\eps$ by adjusting the implicit constant. Let $I=(X,X+H]\cap\Z=\lfloor X\rfloor+[N]$, where $N=|I|$. Since $|N-H|<1$, we have
\begin{equation}
 \label{eq:main-integer-length-range}
 X^{5/8+\eps/2}\leq N\leq X^{1-\eps/2}.
\end{equation}

\emph{Reduction to nilsequence correlation.} Set 
$$q_0=\exp((\log w)^{1/2}).$$
If a level $w$ exceptional conductor $q$ exists and $q\leq q_0$, take 
\begin{align}\label{eq:LambdaR01}
\Lambda_w^\dagger=\widetilde\Lambda_w\qquad \textnormal{and}\qquad R_0=q_0.    
\end{align} 
Otherwise take 
\begin{align}\label{eq:LambdaR02}
\Lambda_w^\dagger=\Lambda_w\qquad \textnormal{and}\qquad R_0=\exp((\log w)^{1/4}). 
\end{align}
Observe that the difference $\widetilde\Lambda_w-\Lambda_w^\dagger$
vanishes unless the level $w$ exceptional conductor exceeds
$q_0$. Thus Lemma \ref{lem:local-exceptional} produces,
for some $c_0=c_0(k,\varepsilon)>0$, the bound
\begin{equation}
 \label{eq:large-conductor-removal}
 \|\widetilde\Lambda_w-\Lambda_w^\dagger\|_{U^k(I)}
 \ll_{k,\eps}q_0^{-2^{-k-2}}+\exp\bigl(-c_0(\log X)^{4/5}\bigr)\ll \exp(-2(\log X)^{c}),
\end{equation}
if $c>0$ is chosen small enough. 

Suppose for the sake of contradiction that
\begin{equation}
 \label{eq:main-failure-assumption}
 \|\Lambda-\widetilde\Lambda_w\|_{U^k(I)}>\exp(-(\log X)^{c}).
\end{equation}
From~\eqref{eq:Mertens} we have $|(\Lambda-\Lambda_w^{\dagger})(n)|\ll\log X$ for $n\in I$. Thus $$f=2\exp(-(\log X)^{c})(\Lambda-\Lambda_w^{\dagger})$$
is $1$-bounded, and  \eqref{eq:main-failure-assumption} and \eqref{eq:large-conductor-removal} give 
$$\|f\|_{U^k(I)}>\exp(-2(\log X)^{c}).$$
We apply the quasipolynomial inverse theorem (Lemma \ref{lem:LSS-inverse}) with $\delta=\exp(-2(\log X)^{c})$ to obtain a nilmanifold $G/\Gamma$ of degree and step at most $k-1$, of dimension at most $D\geq2$, and of complexity at most $M$, and a $1$-bounded nilsequence $F(g(\cdot)\Gamma)$ of Lipschitz norm at most $M$, such that
\begin{equation}
\label{eq:main-normalised-correlation}
 \left|\sum_{n\in I}f(n)F(g(n)\Gamma)\right|\geq\rho N
\end{equation}
where
\begin{equation}
 \label{eq:deterministic-inverse-bounds}
 \max\{M,\rho^{-1}\}\leq\exp((\log w)^{1/4}),\qquad (D+2)^{kC}\leq (\log w)^{1/4},
\end{equation}
with $C$ the constant in Proposition \ref{prop:step-descent}.
Hence, by the choice of $R_0$ in~\eqref{eq:LambdaR01} or~\eqref{eq:LambdaR02}, we have
\begin{equation}
 \label{eq:main-initial-correlation}
 \left|\sum_{n\in I}(\Lambda-\Lambda_w^{\dagger})(n)F(g(n)\Gamma)\right|\geq\rho N \geq N/R_0.
\end{equation}

\emph{Step descent.} Define
\begin{equation}
 \label{eq:stage-reserve-inequalities}
 R_j=R_0^{(D+2)^{jC}}\qquad\text{for }0\leq j\leq k.
\end{equation}
Equation \eqref{eq:deterministic-inverse-bounds} gives $\log R_k\leq (\log w)^{3/4}$ in both cases~\eqref{eq:LambdaR01} and~\eqref{eq:LambdaR02}, and $R_k\leq q_0$ when $\Lambda_w^\dagger=\Lambda_w$. Since $\log w\leq(\log X)^{1/10}$, we have
\begin{align}\label{eq:R0bounds}
 R_0\geq\max\{D,2,(\log X)^C\},\log R_k\leq (\log X)^{3/40},\qquad R_k^C\leq w.
\end{align}
In case~\eqref{eq:LambdaR01} the exceptional conductor is at most $R_0$. In case~\eqref{eq:LambdaR02} no exceptional pair exists, or its conductor exceeds $q_0\geq R_k$.

Write $s$ for the step of $G$ and set $F_0=F$. Starting from \eqref{eq:main-initial-correlation}, we repeatedly apply Proposition \ref{prop:step-descent} on $I$, with dimension bound $D$ and input bound $R_j$. The output bound is $R_{j+1}$, and the estimates~\eqref{eq:R0bounds} and~\eqref{eq:stage-reserve-inequalities} provide the size and conductor hypotheses of that proposition. Each application of the proposition keeps the degree and dimension bounds the same, lowers the step, and gives on some nilmanifold $G_{j+1}/\Gamma_{j+1}$ of step $\leq s-j-1$ of complexity at most $R_{j+1}$ a $1$-bounded Lipschitz function $F_{j+1}$ for which
\begin{equation}
 \label{eq:main-stage-correlation}
 \left|\sum_{n\in I}(\Lambda-\Lambda_w^\dagger)(n)
 F_{j+1}(g_{j+1}(n)\Gamma_{j+1})\right|^*
 \geq N/R_{j+1},\qquad
 \|F_{j+1}\|_{\Lip}\leq R_{j+1}.
\end{equation}
After at most $k-1$ applications we reach step zero. We write $R_*=R_{j_*}$ for the bound at that stage, where $j_*\leq k-1$.

\emph{The final contradiction.} The final nilsequence $F_{j_*}$ is constant of modulus at most one, so some progression $P\subset I$ satisfies
\begin{equation}
 \label{eq:final-reserved-bounds}
 \left|\sum_{n\in P}(\Lambda-\Lambda_w^{\dagger})(n)\right|\geq N/R_*.
\end{equation}
Since $|\Lambda-\Lambda_w^{\dagger}|\leq R_0\leq R_*$ and $R_k\leq X^{\varepsilon/100}$, we have $|P|\geq N/R_*^2\geq2$. Therefore
\begin{align*}
 \textnormal{mod}(P)\leq\frac{N-1}{|P|-1}\leq2R_*^2\leq R_k,
\end{align*}
where the last inequality follows from $R_*\leq R_{k-1}$ and $(D+2)^C\geq4$.

Finally, we apply the major arc result of Proposition \ref{prop:short-major-arcs} on $I$. The required length condition follows from \eqref{eq:main-integer-length-range}, and the modulus is admissible as $R_k^C\leq w$ and $\log R_k\leq (\log X)^{3/40}$. In case~\eqref{eq:LambdaR02} the exceptional conductor, if it exists, exceeds $R_k$ and hence does not divide the modulus of $P$. Hence, in any case, for some $c'=c'(k,\varepsilon)>0$ we have
\begin{align*}
 \left|\sum_{n\in P}(\Lambda-\Lambda_w^{\dagger})(n)\right|\ll_{k,\eps}N\exp\bigl(-c'(\log X)^{1/10}\bigr)\ll N/R_k^2,
\end{align*}
say. This contradicts \eqref{eq:final-reserved-bounds}, since $R_*\leq R_k$. Thus \eqref{eq:main-failure-assumption} cannot hold, proving the claimed bound~\eqref{eq:large-cutoff-uniformity} for sufficiently large $X$, as desired.
\end{proof}

Finally, we now prove the full Theorem~\ref{thm:corrected-main} by comparing model functions.

\begin{proof}[Proof of Theorem \ref{thm:corrected-main}]
We may assume that $X$ is sufficiently large in terms of $\varepsilon$ and $k$ by adjusting implicit constants.

We first prove the bound when $0<\eps<3/16$ and 
\begin{align}\label{eq:Hrange}
X^{5/8+\eps}\leq H\leq X^{1-\eps}.    
\end{align}
For sufficiently large $X$, we set $z=\exp((\log X)^{1/10})$. This cutoff satisfies \eqref{eq:large-cutoff-w-range}, so Proposition \ref{prop:large-cutoff-uniformity} gives, for some $c_1=c_1(k,\eps)>0$,
\begin{align*}
 \|\Lambda-\widetilde\Lambda_z\|_{U^k(X,X+H]}
 \ll_{k,\eps}\exp\bigl(-(\log X)^{c_1}\bigr).
\end{align*}
Then the model comparison of Lemma \ref{lem:local-model-comparison} applies with $\eps/2$ in place of $\eps$ for large $X$
and gives
\begin{align*}
 \|\widetilde\Lambda_z-\widetilde\Lambda_w\|_{U^k(X,X+H]}\ll_{k,\eps}w^{-a_k}+\exp(-c(\log X)^{4/5}).    
\end{align*}
Applying the triangle inequality gives \eqref{eq:corrected-main-bound} after decreasing $c$. Now the result follows in the range~\eqref{eq:Hrange}.

Now let $0<\eps<5/8$ and let $H$ satisfy \eqref{eq:main-H-range}.  We set $$\eta=\varepsilon/100,\qquad  I=(X,X+H]\cap\Z,\qquad N=|I|,\qquad L=\lfloor X^{5/8+\eps/2}\rfloor.$$
We partition $I$
into the consecutive integer intervals
\begin{align*}
 J_j=(X+jL,X+(j+1)L]\cap\Z
 \quad\text{for }0\leq j<B\coloneqq \lfloor N/L\rfloor
\end{align*}
and a final interval $R$ of fewer than $L$ integers.  Then $|J_j|=L$ for all $j$, and the left endpoint $Y_j=X+jL$ of $J_j$ belongs to $[X,3X]$. Moreover, for large $X$, we have
\begin{align*}
 Y_j^{5/8+\eta}\leq L\leq Y_j^{1-\eta},
 \qquad
 w\leq\exp\bigl((\log Y_j)^{1/10}\bigr).
\end{align*}
The result in the shorter range~\eqref{eq:Hrange} with parameter $\eta$ applies to every $J_j$.  We write 
$$f=\Lambda-\widetilde\Lambda_w,\qquad \alpha_k=(k+1)/2^k,\qquad \rho=w^{-a_k}+\exp(-(\log X)^{c_\eta}),$$ where $c_\eta>0$ is the exponent in that result.  Since $\|\1_{J_j}\|_{U^k(\Z)}\asymp_k L^{\alpha_k}$, we have
\begin{equation}
 \label{eq:block-gowers-bound}
 \|f\1_{J_j}\|_{U^k(\Z)}\ll_{k,\eps} \rho L^{\alpha_k}
 \qquad\text{for }0\leq j<B.
\end{equation}
By the triangle inequality and the bound $|f\1_{I}|\ll \log X$, we also have
\begin{align}\label{eq:block-error}
\|f\1_{I}\|_{U^k(\mathbb{Z})}=\|\sum_{0\leq j<B}f\1_{J_j}\|_{U^k(\Z)}+O(L^{\alpha_k}(\log X)).    
\end{align}
We expand out
\begin{align*}
 \left\|\sum_{0\leq j<B}f\1_{J_j}\right\|_{U^k(\mathbb Z)}^{2^k}
 =\sum_{\mathbf j\in([0,B)\cap\Z)^{2^k}}\,
 \sum_{x,h_1,\ldots,h_k\in\mathbb Z}\,
 \prod_{\omega\in\{0,1\}^k}
 (f\1_{J_{\mathbf j(\omega)}})(x+\omega\cdot\mathbf h).
\end{align*}
Note that the number of tuples $\mathbf{j}$ that have a nonzero contribution is $O(B^{k+1})$, since if $x\in J_{\mathbf{j}(\mathbf{0})}$ and $x+\mathbf{e}_{\ell}\cdot \mathbf{h}\in J_{\mathbf{j}(\mathbf{e}_{\ell})}$ for all $1\leq \ell\leq k$, then for any $\omega\in \{0,1\}^k$ we have
\begin{align*}
 x+\omega\cdot\mathbf h
 \in (1-|\omega|)J_{\mathbf j(\mathbf0)}
 +\sum_{\ell:\,\omega_\ell\ne0}J_{\mathbf j(\mathbf e_\ell)},
\end{align*}
where $|\omega|=\omega_1+\cdots+\omega_k$. The set on the right-hand side is contained in an interval of length at most $(2k-1)L$.

Now, from the Gowers Cauchy--Schwarz inequality, \eqref{eq:block-gowers-bound} and~\eqref{eq:block-error} we see that
\begin{align*}
\|f\1_{I}\|_{U^k(\mathbb{Z})}\ll_{k,\varepsilon} B^{(k+1)/2^k}\cdot \rho L^{\alpha_k}+O(L^{\alpha_k}(\log X))\ll \rho N^{(k+1)/2^k}.     
\end{align*}
Decreasing the exponent $c$ if necessary, the claimed bound~\eqref{eq:corrected-main-bound} follows.
\end{proof}

\end{document}